\documentclass[a4paper,11pt,reqno]{amsart}
\usepackage{amsmath,amssymb,amsfonts,amsthm}
\usepackage{latexsym,textcomp}
\usepackage{dsfont}
\usepackage[german, english]{babel}
\usepackage{hyperref}
\usepackage[capitalize,nameinlink]{cleveref}
\usepackage{enumerate}
\usepackage[utf8]{inputenc}
\usepackage{faktor}

\usepackage{color}

\newtheorem{lemma}{Lemma}[section]
\newtheorem{theorem}[lemma]{Theorem}

\newtheorem{proposition}[lemma]{Proposition}
\newtheorem{corollary}[lemma]{Corollary}

\theoremstyle{definition}
\newtheorem{definition}[lemma]{Definition}
\newtheorem{example}[lemma]{Example}
\newtheorem{remark}[lemma]{Remark}
\newtheorem*{assumption}{Standing assumption}

\numberwithin{equation}{section}
\newcommand{\comment}[1]{}

\newcommand{\cE}{{\mathcal E}}

\newcommand{\cF}{{\mathcal F}}

\newcommand{\cL}{{\mathcal L}}

\newcommand{\R}{{\mathbb R}}
\newcommand{\IK}{{\mathbb K}}
\newcommand{\IC}{{\mathbb C}}

\newcommand{\N}{{\mathbb N}}

\newcommand{\fA}{{\mathfrak A}}

\newcommand{\eps}{{\varepsilon}}

\newcommand{\vr}{{\varrho}}
\renewcommand{\rho}{\varrho}

\newcommand{\as}[1]{\langle #1\rangle}
\newcommand{\av}[1]{\left\Vert #1\right\Vert}

\newcommand{\Hm}[1]{\leavevmode{\marginpar{\tiny%
$\hbox to 0mm{\hspace*{-0.5mm}$\leftarrow$\hss}%
\vcenter{\vrule depth 0.1mm height 0.1mm width \the\marginparwidth}%
\hbox to 0mm{\hss$\rightarrow$\hspace*{-0.5mm}}$\\\relax\raggedright
#1}}}

\newcommand{\tom}{\overset{\mu}{\to}} 
 
\DeclareMathOperator*{\esssup}{ess\,sup}
\DeclareMathOperator*{\essinf}{ess\,inf}

\begin{document}

\title[Nonlinear Dirichlet forms]{The extended Dirichlet space and criticality theory for nonlinear Dirichlet forms}

%

\author{Marcel Schmidt}
\address{Mathematisches Institut, Universität Leipzig, 04109 Leipzig, Germany.} \email{marcel.schmidt@math.uni-leipzig.de} 

\author{Ian Zimmermann}
\address{Institut für Mathematik, Friedrich-Schiller-Universität Jena, 07743 Jena, Germany.} \email{ian.zimmermann@uni-jena.de} 

\maketitle

\begin{abstract}
In this paper we establish the existence of the extended Dirichlet space for nonlinear Dirichlet forms under mild conditions. We employ it to introduce and characterize criticality (recurrence) and subcriticality (transience) and establish basics of a potential theory.  
\end{abstract}





\section{Introduction}

Dirichlet forms were introduced by Beurling and Deny \cite{BD58,BD59}.  They (and their generalizations) have since become an important tool in various mathematical areas such as geometric analysis and stochastic analysis but also in mathematical physics. The leitmotif is that an order structure of the underlying Hilbert space can be employed to study properties of the generator and the induced semigroup,   which under additional assumptions can even be related to a Markov process through a Feynman-Kac type formula - see \cite{MR,FOT,CF12} for a modern account. While this theory is linear - the functionals are bilinear forms and the associated operators are linear - key ideas clearly extend to some nonlinear functionals and operators.

In this text we treat nonlinear Dirichlet forms in the sense of Cipriani and Grillo \cite{CG03}. These are lower semicontinuous convex functionals on $L^2$-spaces whose associated nonlinear semigroups are order preserving and $L^\infty$-contractive - in the linear case they correspond to Dirichlet forms in the wide sense (i.e. Dirichlet forms without dense domain). Indeed, this class of functionals has been studied and characterized before by Benilan and Picard \cite{BP79} and, adding to the confusion, Benilan and Picard used the name nonlinear Dirichlet form for a slightly different class of functionals. 

Benilan and Picard and later Cipriani and Grillo provided  characterizations for nonlinear Dirichlet forms that are quite similar to the second Beurling-Deny criterion, which defines Dirichlet forms in the linear case. Nonlinear Dirichlet forms include energies of the $p$-Laplacian in various geometries (see e.g. \cite[Section~4]{CG03}), energies of the variable exponent $p(\cdot)$-Laplacian (see Section~\ref{section:examples}), area and perimeter functionals (see \cite[Section~5]{CG03}) and Cheeger energies on metric measure spaces (see \cite{AGS14}) - just to name a few. 

While certainly a lot of theory has been developed at an individual level for all of the mentioned examples, a general theory making some key results for Dirichlet forms also available for nonlinear Dirichlet forms is missing. Most notably, as of yet there is no way to associate a stochastic process to a suitably regular nonlinear Dirichlet form - first results in this direction have only been obtained for some nonlinear Lévy processes \cite{DKN20} and the $p$-Laplacian on $\mathbb R^d$ \cite{BRR24}.

Recently  (under some regularity conditions) a potential theory for nonlinear Dirichlet forms  without homogeneity assumptions has been developed in \cite{Cla21,Cla23} and the somewhat complicated characterization of nonlinear Dirichlet forms by  Cipriani and Grillo has been simplified \cite{BH24,Puc25}.  With the present article we seek to fill further gaps in the theory of nonlinear Dirichlet forms. More precisely, we show the existence of the extended Dirichlet space (Theorem~\ref{theorem:existence extended Dirichlet space}) and discuss the related criticality theory by introducing and characterizing criticality (recurrence) (Theorem~\ref{theorem:characterization criticality}) and subcriticality (transience) (Theorem~\ref{theorem:characterization subcriticality}). Both concepts are related to the validity of (weak) Hardy- and Poincaré inequalities (Theorem~\ref{theorem:weak hardy} and Theorem~\ref{theorem:weak poincare}), which might be of interest on their own right. Moreover, we show that there is a dichotomy between criticality and subcriticality under an irreducibility condition, see Corollary~\ref{corollary:dichotomy}. In the literature this dichotomy is sometimes called ground state alternative and our article is very much inspired by \cite{PT07,Fis23}, which show that this ground state alternative holds for certain nonlinear situations. As an application, for subcritical nonlinear Dirichlet forms we develop basics  of a potential theory (Corollary~\ref{corollary:existence of equilibrium potentials} and Theorem~\ref{theorem:characterization exceptional sets}),  which in some aspects goes beyond the results of \cite{Cla21}, and includes the situations recently discussed in \cite{BBR24,Kuw24}, but also energies of $p(\cdot)$-Laplacians as in \cite{DHHR11}. 

Extended Dirichlet spaces are an important tool for stochastic analysis of (quasi)regular Dirichlet forms. Among other things they allow the change of the measure of the underlying $L^2$-space, which for the associated process corresponds to a time change. Criticality and subcriticality can be formulated with their help and together with (weak) Hardy- and Poincaré inequalities these global properties determine the long time behavior of the associated semigroup/process (recurrence, transience, convergence rate to equilibrium). The potential theory of a (quasi)regular Dirichlet form plays a key role in the construction of the associated process. Again we refer to \cite{FOT,MR} for the construction of processes, potential theory, recurrence and transience and to \cite{RW01} for the convergence rate to equilibrium. Beyond stochastic analysis subcriticality in terms of weighted Hardy inequalities and the dichotomy between criticality and subcriticality play a role for establishing optimal Hardy inequalities, which are used in Agmon estimates for weak solutions of associated elliptic problems, see e.g. \cite{Pin,DP16}.   

 Our results do not necessarily hold for all nonlinear Dirichlet forms but require certain additional assumptions, which vary slightly from theorem to theorem. Those assumptions are reflexivity, completeness of the modular space, existence of the extended form and the (weak) $\Delta_2$-condition. At some places we do not even need to assume that we are given nonlinear Dirichlet forms, nonlinear order preserving forms are sufficient. We present the theorems with the minimal assumptions required for our proofs even though this makes the statements somewhat incoherent. The reader should note that all results in this text hold for reflexive nonlinear Dirichlet forms satisfying the $\Delta_2$-condition, a class that includes the most prominent examples, cf. Section~\ref{section:examples}.  

Claus noticed in his PhD thesis \cite{Cla21} that the Luxemburg seminorm of a nonlinear Dirichlet form is important for understanding its properties - indeed this is the key insight that made our article possible. In order to employ Luxemburg seminorms for our purposes we have to extend some aspects of the available theory for semimodulars (see e.g. \cite[Chapter~2]{DHHR11})  to our setting. More precisely, we introduce generalized semimodulars and their Luxemburg seminorms and study in quite some detail their lower semicontinuity with respect to different topologies on the underlying space (Theorem~\ref{theorem:lower semicontinuity} and Theorem~\ref{theorem:description extended space}) and  the continuity of the embedding into the underlying space (Theorem~\ref{theorem:two imply the third}). 

The paper is organized as follows: In Section~\ref{section:introduction to nonlinear Dirichlet forms} we review basic properties of nonlinear Dirichlet forms mostly following the existing literature. In Section~\ref{section:modular and extended space} we introduce the modular space and the extended Dirichlet space and give a short account on excessive functions and invariant sets. Section~\ref{section:criticality theory} can be seen as the main part of this article, it discusses characterizations of criticality and subcriticality. In Section~\ref{section:potential theory} we apply subcriticality to establish basics of a potential theory. In Section~\ref{section:examples} we discuss two classes of examples. Here, we do not aim at greatest generality, but try to illustrate how our results fit into the existing literature. We chose to put the discussion of generalized semimodulars into Appendix~\ref{appendix:generalized semimodulars}. It is more or less independent of the rest of the article and might be of use in other contexts as well.

{\bf Acknowledgments:} The authors thank Ralph Chill and Simon Puchert for sharing their knowledge on nonlinear Dirichlet forms with them and René Schilling for pointing out some references. Moreover, M.S. acknowledges financial support of the DFG.

\section{Nonlinear Dirichlet forms, resolvents and the Green operator} \label{section:introduction to nonlinear Dirichlet forms} Let $(X,\mathfrak A,\mu)$ be a $\sigma$-finite measure space. We let $L^2(\mu)$ be the  Hilbert space of real-valued square integrable functions with inner product   $\as{\cdot,\cdot}$ and  induced norm $\av{\cdot}$.

We write $L^0(\mu)$ for the space of $\mu$-a.s.\@ defined real-valued measurable functions. It is equipped with the topology of local convergence in measure. If we fix some integrable   $h \colon X \to (0,\infty)$,  this topology is induced by the $F$-norm 
$$p_h \colon L^0(\mu) \to [0,\infty], \quad p_h(f) = \int_X \min\{|f|,h\} d\mu.$$
If $(f_n)$ converges to $f$ in this topology, we write $f_n \tom f$. This convergence holds if every subsequence of $(f_n)$ has a subsequence converging to $f$ $\mu$-a.s. The cone of $[0,\infty]$-valued $\mu$-a.s.\@ defined measurable functions is denoted by $L^+(\mu)$.

\begin{definition}[The Beurling-Deny criteria and nonlinear Dirichlet forms]
Let $\cE \colon L^p(\mu) \to (-\infty,\infty]$ with $p \in \{0,2\}$. We say that a normal contraction $C \colon \R \to \R$ {\em operates on $\cE$} if for all $f,g \in L^p(\mu)$ the inequality 
$$\cE(f + Cg) + \cE(f - Cg) \leq \cE(f + g) + \cE(f - g)$$
holds. The functional $\cE$ satisfies the {\em first Beurling-Deny criterion} if the normal contraction $\R \to \R, \, x \mapsto |x|$ operates on $\cE$ and it satisfies the {\em second Beurling-Deny criterion} if all normal contractions operate on $\cE$. 

A {\em nonlinear order preserving form} is a proper lower semicontinuous convex functional $\cE \colon L^2(\mu) \to (-\infty,\infty]$ satisfying the first Beurling-Deny criterion and a {\em nonlinear Dirichlet form} is  a nonlinear order preserving form also satisfying the second Beurling-Deny criterion.

\end{definition}
\begin{remark}
\begin{enumerate}[(a)]
 \item If $\cE \colon L^2(\mu) \to [0,\infty]$ is a Dirichlet form in the sense of \cite{FOT} (i.e. a densely defined lower semicontinuous quadratic form such that $\cE( Cf ) \leq \cE(f)$ for all normal contractions $C$ and all $f \in L^2(\mu)$), then it satisfies the second Beurling-Deny criterion in the sense of the previous definition due to the parallelogram identity. To better distinguish them from the more general functionals considered in this text we call them {\em classical Dirichlet forms} or {\em bilinear Dirichlet forms} (because they are related to bilinear forms on their domain via polarization).

 \item Based on results from \cite{BH24}, it was recently established in \cite{Puc25} that our Definition of nonlinear Dirichlet forms yields the same as in \cite{CG03} - see also Proposition~\ref{proposition:contraction properties resolvent} below.  Indeed, the same class of functionals was already introduced and characterized earlier in \cite[Théorème~2.1]{BP79} without giving them a special name. We also note that in \cite[Section~2b]{BP79}  a different class of functionals is called nonlinear Dirichlet form - this seems to have caused some confusion in earlier publications.

What we call nonlinear order preserving form is called nonlinear semi-Dirichlet form in \cite{CG03}. We do not use this terminology to avoid confusion with the concept of semi-Dirichlet forms in the linear case, which is something different, see e.g. \cite{Osh13}. 

\item Below we will also study extensions of nonlinear Dirichlet forms to $L^0(\mu)$, which is why we included functionals on this space in the previous definition. 

\end{enumerate}
\end{remark}

For a convex functional $\cE \colon L^2(\mu) \to (-\infty,\infty]$ its {\em subgradient} is the multi-valued operator
$$\partial \cE = \{(f,g)  \in L^2(\mu) \times L^2(\mu) \mid \as{g,h-f} + \cE(f) \leq \cE(h) \text{ all } h \in L^2(\mu)\}.$$
This is in line with the discussion in Subsection~\ref{subsection:dual spaces and resolvents} after identifying $L^2(\mu)$ and its dual space. If $\cE$ is proper, lower semicontinuous and convex, then by general principles $\partial \cE$ is $m$-accretive. This implies that the resolvent $G_\alpha = (\alpha  + \partial \cE)^{-1}$, $\alpha > 0$, is a well-defined single-valued $\alpha^{-1}$-Lipschitz continuous operator on $L^2(\mu)$. We refer to  \cite[Section~IV.1]{Sho97}, where all of these claims and related things are discussed. The element $G_\alpha f$ is the unique minimizer of the functional 
$$L^2(\mu) \to [0,\infty], \quad g \mapsto \cE(g) + \frac{\alpha}{2} \av{g - \alpha^{-1} f}^2.$$
Hence, it is related to the resolvent family $J_\alpha$, $\alpha > 0$, introduced in Subsection~\ref{subsection:dual spaces and resolvents}, through $G_\alpha f = J_{\alpha^{-1}} (\alpha^{-1} f)$, $\alpha > 0$. It satisfies the resolvent identity 
$$G_\alpha f = G_\beta (f + (\beta - \alpha) G_\alpha f), \quad \alpha,\beta > 0, f \in L^2(\mu).$$
Moreover, for $f \in \overline{D(\cE)}$, we have $\alpha G_\alpha f \to f$, as $\alpha \to \infty$.

The following fundamental observation explains the names 'nonlinear Dirichlet form' and 'order preserving form'.

\begin{proposition}\label{proposition:contraction properties resolvent}
 Let $\cE \colon L^2(\mu) \to (-\infty,\infty]$ be convex, lower semicontinuous and proper. 
 \begin{enumerate}[(a)]
  \item The following assertions are equivalent. \begin{enumerate}[(i)]
                                                    \item $\cE$ is a nonlinear order preserving form. 
                                                    \item For all $f,g \in L^2(\mu)$ we have 
                                                    $$\cE(f \wedge g) + \cE(f \vee g) \leq \cE(f) + \cE(g).$$
                                                    \item The resolvent $G_\alpha$, $\alpha > 0$, is order preserving (i.e. $f \geq g$ implies $G_\alpha f \geq G_\alpha g$ for all $\alpha > 0$). 
                                                 \end{enumerate}
  \item $\cE$ is a nonlinear Dirichlet form if and only if $\alpha G_\alpha$, $\alpha > 0$, is order preserving and $L^\infty$-contractive, i.e.   $$\av{\alpha G_\alpha f  - \alpha G_\alpha g}_\infty \leq \av{f - g}_\infty \text{ for all } f,g \in L^2(\mu) \cap L^\infty(\mu).$$
  In this case, $\alpha G_\alpha$, $\alpha > 0$, is also $L^1$-contractive.
 \end{enumerate}
\end{proposition}
\begin{proof}
(a): The equivalence  (i) $\Leftrightarrow$ (ii) is readily verified by inserting $(f+g)/2$ and $(f-g)/2$ into the defining inequality of the first Beurling-Deny criterion. The equivalence of (ii) and (iii) is the content of   \cite[Theorem~3.8]{CG03}.

(b): This is the content of  \cite[Theorem~3.6]{CG03} together with the characterization of nonlinear Dirichlet forms in \cite{Puc25}, see also \cite[Théorème~2.1]{BP79}.
\end{proof}

We extend the resolvent of a nonlinear order preserving form $\cE$ to an operator $L^+(\mu) \to L^+(\mu)$ as follows: For $f \in L^+(\mu)$ we choose an increasing sequence $(f_n)$ of nonnegative functions in $L^2(\mu)$ with $f_n \nearrow f$ a.s.\@ and set $ G_\alpha f = \lim_{n \to  \infty} G_\alpha f_n$. This limit exists a.s.\@ due to $G_\alpha$ being order preserving and it is independent of the approximating sequence due to the continuity of $G_\alpha$ on $L^2(\mu)$. For $\beta \geq \alpha$ the resolvent identity $G_\alpha = G_\beta({\rm id} + (\beta - \alpha) G_\alpha )$ extends to $L^+(\mu)$. Together with the order preservation, it implies that for $f \in L^+(\mu)$ the map $(0,\infty) \to L^+(\mu)$, $\alpha \mapsto G_\alpha f$ is decreasing. Hence, the operator
$$G \colon L^+(\mu) \to L^+(\mu), \quad Gf = \lim_{\alpha \to 0+} G_\alpha f$$
is well-defined. It is called {\em Green operator} or {\em potential operator} of $\cE$.

If $\cE$ is a nonlinear Dirichlet form, for $\alpha > 0$ we extend $\alpha G_\alpha$ to a contraction on $L^1(\mu)$ by using the density of $L^1(\mu) \cap L^2(\mu)$ in  $L^1(\mu)$ and the $L^1$-contractivity of $\alpha G_\alpha$. By the monotone convergence theorem, this extension to $L^1(\mu)$ by continuity is compatible with the extension to $L^+(\mu)$ discussed above. Moreover, we obtain an extension of $G_\alpha$, $\alpha > 0$, to $L_+^\infty(\mu)$ by simply restricting its extension to $L^+(\mu)$ to the subset $L_+^\infty(\mu) = L^\infty(\mu)\cap L^+(\mu)$. Using the $L^\infty$-contractivity, it is readily verified that $\alpha G_\alpha$, $\alpha > 0$, are contractions on $L^\infty_+(\mu)$.

%
%

\section{The modular space and the extended Dirichlet space} \label{section:modular and extended space}

In this section we apply the abstract theory outlined in Appendix~\ref{appendix:generalized semimodulars} to convex functionals on the spaces $L^2(\mu)$ and $L^0(\mu)$. While we repeat the most important definitions in the main text, we often refer to the appendix for details. 

\subsection{The modular space} We say that a functional $\cE \colon L^2(\mu) \to (-\infty,\infty]$ is  {\em symmetric} if $\cE(f) = \cE(-f)$ for all $f \in L^2(\mu)$. In this case, properness and convexity imply that $0 \in D(\cE)$ is a global minimizer of $\cE$. Moreover, it is readily verified that if a normal contraction operates on $\cE$, it also operates on the functional $\cE - \cE(0)$.  Based on these observations we make the following standing assumptions.

\begin{assumption}
 From now on all  convex functionals $\cE$ in this text are assumed to be symmetric and satisfy $\cE(0) = 0$. In particular, they are proper and map to $[0,\infty]$.
\end{assumption}

\begin{remark}
 \begin{enumerate}[(a)]
  \item Even if $\cE$ were not symmetric, many of our results would still hold true. In this case, one has to use versions of the abstract results from Appendix~\ref{appendix:generalized semimodulars} on cones instead of vector spaces. We refrain from giving details.
 
 Alternatively, if $\cE$ is not symmetric, one can obtain a symmetric functional by considering $\cE^{sym} \colon L^2(\mu) \to [0,\infty],\, f\mapsto  \frac{1}{2} \left(\cE(f) + \cE(-f)\right)$. If a normal contraction operates on $\cE$, it also operates on $\cE^{sym}$. 
\item If a symmetric convex functional $\cE$ with $\cE(0) = 0$ is compatible with the normal contraction $C \colon \R \to \R$, then it satisfies $\cE(Cf) \leq \cE(f)$ for all $f \in L^2(\mu)$.   
 \end{enumerate}
\end{remark}

A lower semicontinuous proper convex functional $\cE \colon L^2(\mu) \to [0,\infty]$ satisfying our standing assumption is a generalized semimodular on $L^2(\mu)$ in the sense of Definition~\ref{definition:generalized semimodular}. We write 
$$M(\cE) = \{f \in L^2(\mu) \mid \lim_{\lambda \to 0+} \cE(\lambda f) = 0\} = {\rm lin}\, D(\cE)$$
for the corresponding {\em modular space} and 
$$\av{\cdot}_L = \av{\cdot}_{L,\cE} \colon M(\cE) \to [0,\infty), \quad \av{f}_L = \inf\{\lambda > 0 \mid \cE(\lambda^{-1} f) \leq 1\}$$
for the induced {\em Luxemburg seminorm}.

\begin{lemma}[Luxemburg seminorms are compatible with contractions]\label{lemma:contraction luxemburg seminorms}
Let $\cE$ be a nonlinear order preserving form. Then $\av{|f|}_L \leq \av{f}_L$ for all $f \in M(\cE)$. If $\cE$ is a nonlinear Dirichlet form, then $\av{Cf}_L \leq \av{f}_L$ for all normal contractions $C \colon \R \to \R$ and all $f \in M(\cE)$. 
\end{lemma}
\begin{proof}
The symmetry and the first Beurling-Deny criterion imply $\cE(|f|) \leq \cE(f)$. Hence, for $\lambda > 0$ we infer  
$$\cE(\lambda^{-1} |f|) = \cE(|\lambda^{-1}f|) \leq \cE(\lambda^{-1} f).$$
This implies $\av{|f|}_L \leq \av{f}_L$. 

Assume that $\cE$ is a nonlinear Dirichlet form and let $C \colon \R \to \R$ be a normal contraction. For $\lambda > 0$ let $C_\lambda \colon \R \to \R$, $C_\lambda(t)  = \lambda^{-1}C(\lambda x)$, which is also a normal contraction. Using that $\cE$ satisfies the second Beurling-Deny criterion and is symmetric, we obtain 
$$\cE(\lambda^{-1} Cf) = \cE(C_\lambda( \lambda^{-1}f)) \leq \cE(\lambda^{-1}f),$$
showing $\av{Cf}_L \leq \av{f}_L$.
\end{proof}

We say that $\cE$ is {\em reflexive} if the image of $M(\cE)$ in the bidual $(M(\cE),\av{\cdot}_L)''$  (under the canonical isometry) is norm dense. For more context on this property we refer to Appendix~\ref{appendix:generalized semimodulars}.

\begin{lemma}[Continuity of certain contractions]\label{lemma:continuity of normal contractions}
Let $\cE$ be a reflexive nonlinear Dirichlet form and for $r > 0$ let $C_r \colon \R \to \R$, $C_r(x) = (x \wedge r) \vee(-r)$. For all $f \in M(\cE)$ we have 
$$\lim_{\varepsilon \to 0+} C_\varepsilon(f)  = 0  \text{ and } \lim_{R \to \infty} C_R(f)  = f $$
with respect to $\av{\cdot}_L$. In particular, $L^1(\mu) \cap L^\infty(\mu) \cap M(\cE)$ is dense in $(M(\cE),\av{\cdot}_L)$. 
\end{lemma}
\begin{proof}
Since $C_r$ is a normal contraction, it follows from Lemma~\ref{lemma:contraction luxemburg seminorms} that  $ \av{C_r (f)}_L \leq \av{f}_L$. Since $ \lim_{R \to \infty} C_R(f)  = f$ and  $\lim_{\varepsilon \to 0+} C_\varepsilon(f)  = 0$ in $L^2(\mu)$, Lemma~\ref{lemma:weak convergence} shows that the claimed convergence holds weakly in $(M(\cE),\av{\cdot}_L)$. By Mazur's lemma we can upgrade it to $\av{\cdot}_L$-convergence by taking suitable convex combinations.

For a convex combination $\sum_{k = 1}^n \lambda_k C_{\varepsilon_k}(f)$ and $\varepsilon \leq \min\{\varepsilon_1,\ldots,\varepsilon_n\}$ we have 
$$C_\varepsilon(f) = C_\varepsilon(\sum_{k = 1}^n \lambda_k C_{\varepsilon_k}(f)).$$
With this equality Lemma~\ref{lemma:contraction luxemburg seminorms} implies
$$\av{C_\varepsilon(f)}_L \leq \av{\sum_{k = 1}^n \lambda_k C_{\varepsilon_k}(f)}_L.$$
Since suitable convex combinations converge to $0$, this yields the statement for the first limit. 

Let $D_r \colon \R \to \R, \quad D_r(x) = (x - r)_+ - (x + r)_-$. Then $f - C_r(f) = D_r(f)$ and for a convex combination we have  
$$f - \sum_{k=1}^n \lambda_k C_{R_k}(f) = \sum_{k = 1}^n \lambda_k D_{R_k}(f).$$
Let $R \geq \max\{R_1,\ldots,R_n\}$. It is readily verified that there exists a normal contraction $D$ such that 
$$D\left(\sum_{k = 1}^n \lambda_k D_{R_k}(f)\right) = D_R(f) = f - C_R(f).$$
Indeed, $D = D_{R'}$ with $R'  = R - \sum_{k=1}^n \lambda_k R_k$ is a possible choice. This equality and Lemma~\ref{lemma:contraction luxemburg seminorms} yield 
$$\av{f - C_R(f)} \leq \av{f - \sum_{k=1}^n \lambda_k C_{R_k}(f)}_L.$$
Since suitable convex combinations converge to $0$, this yields the statement for the second limit. 

The density statement follows from these convergences and that for $f \in L^2(\mu)$ we have $f - C_\varepsilon(f) \in L^1(\mu)$ and $C_R(f) \in L^\infty(\mu)$.
\end{proof}

\begin{remark}
 This lemma was established in \cite[Theorem~3.55]{Cla21} for the nonlinear Dirichlet form $\cE + \av{\cdot}^2$ under the quasilinearity assumption $D(\cE) = M(\cE)$. We can drop quasilinearity and treat $\cE$ directly but have to assume reflexivity. 
\end{remark}

%
%
%
%
%
%
%
%
%
%

\subsection{The extended Dirichlet space} We extend every convex functional  $\cE \colon L^2(\mu) \to [0,\infty]$ to a convex functional on $L^0(\mu)$ by simply letting $\cE(f) =  \infty$ for $f \in L^0(\mu) \setminus L^2(\mu)$. If this extended functional is lower semicontinuous with respect to local convergence in measure on $D(\cE)$ (i.e. for $(f_n)$, $f$ in $D(\cE)$ the convergence $f_n \tom f$ implies $\cE(f) \leq \liminf_{n \to \infty} \cE(f_n)$), then the lower semicontinuous relaxation of $\cE$ on $L^0(\mu)$ is an extension of $\cE$ (cf. Proposition~\ref{proposition:lsc relaxation}), which we call $\cE_e$. First we address the existence of this extension.

\begin{theorem}[Existence of the extended Dirichlet space]\label{theorem:existence extended Dirichlet space}
Let $\cE \colon L^2(\mu) \to [0,\infty]$ be convex (with our standing assumptions). Assume one of the following two conditions: 
\begin{enumerate}[(A)]
 \item $(M(\cE),\av{\cdot}_L)$ is a Banach space. 
 \item $\cE$ is a nonlinear Dirichlet form. 
 \end{enumerate}
Then $\cE_e$ exists. If (A) holds, then $\cE = \cE_e$ (when $\cE$ is interpreted as a functional on $L^0(\mu)$). 

\end{theorem}
\begin{proof}
 (A): According to Theorem~\ref{theorem:lower semicontinuity} we need to show that for each $r > 0$ the Luxemburg seminorm 
 $$\av{f}_{L,r} = \inf \{\lambda > 0 \mid \cE(\lambda^{-1} f) \leq r\}$$
 is lower semicontinuous on $L^0(\mu)$ with respect to local convergence in measure. Since Luxemburg seminorms with different parameters are equivalent (Lemma~\ref{lemma:equivalent luxemburg norms}),  $(M(\cE),\av{\cdot}_{L,r})$ is a Banach space. By Theorem~\ref{theorem:two imply the third} and Remark~\ref{remark:two imply the third} (applied to the generalized semimodular $r^{-1} \cE$) the embedding 
 $$(M(\cE),\av{\cdot}_{L,r}) \to L^2(\mu),\quad  f \mapsto f$$
 is continuous. Now assume that $f \in L^0(\mu)$ and $(f_n)$ is a sequence in $M(\cE)$ with $f_n \tom f$. We assume that $(\av{f_n}_{L,r})$ is bounded (else there is nothing to show). Without loss of generality we further assume $f_n \to f$ almost surely  and 
 $$\lim_{n \to \infty} \av{f_n}_{L,r}  = \liminf_{n \to \infty} \av{f_n}_{L,r} $$
 (else pass to a suitable subsequence). By the continuity of the embedding the sequence $(f_n)$ is bounded in $L^2(\mu)$. By the Banach-Saks theorem it has a subsequence $(f_{n_k})$ and there exists $g \in L^2(\mu)$ such that 
 $$\frac{1}{N} \sum_{k=1}^N f_{n_k} \to g \text{ in }L^2(\mu) \text{, as } N \to \infty.  $$
 From the almost sure convergence, which passes to the Césaro averages, we infer $f = g$. Hence, the $L^2$-lower semicontinuity of $\av{\cdot}_{L,r}$ implies
 $$\av{f}_{L,r} \leq \liminf_{N \to \infty} \lVert \frac{1}{N}\sum_{k=1}^N f_{n_k}\rVert_{L,r}  \leq  \liminf_{N \to \infty} \frac{1}{N}\sum_{k=1}^N\av{ f_{n_k}}_{L,r}  = \lim_{n \to \infty} \av{f_n}_{L,r}.$$

 (B): Let $f \in D(\cE)$ and let $(f_n)$ in $L^2(\mu)$ with $f_n \tom f$. It suffices to show $\cE(f) \leq \liminf_{n \to \infty} \cE(f_n)$.

 1. Step: Assume that $f \in L^1(\mu) \cap D(\cE)$ and that there exists $M \geq 0$ such that $|f_n| \leq M$, $n \in \N$. We use the resolvent $J_\lambda = G_{1/\lambda}(\lambda^{-1} \cdot)$, $\lambda > 0$, discussed in Subsection~\ref{subsection:dual spaces and resolvents}. 
 
 For $\lambda > 0$ we have $J_\lambda f   \in L^1(\mu) \cap L^2(\mu)$, see Proposition~\ref{proposition:contraction properties resolvent} and hence also $g_\lambda := \lambda^{-1}(f - J_\lambda f) \in L^1(\mu) \cap L^2(\mu)$.
 Lebesgue's dominated convergence theorem and the Fenchel-Moreau theorem yield 
 $$ \as{g_\lambda , f} - \cE^*(g_\lambda) = \lim_{n \to \infty} \left(\as{g_\lambda, f_n} - \cE^*(g_\lambda) \right) \leq \liminf_{n \to \infty} \cE(f_n). $$
 Moreover, Proposition~\ref{lemma:alternative formula for vr} (and $\varphi_\lambda = \as{g_\lambda,\cdot}$ in the notation used there) yields 
 $$\cE(f) = \lim_{\lambda \to 0+} \left(\as{g_\lambda , f} - \cE^*(g_\lambda)\right).  $$
 Combining both (in)equalities shows $\cE(f) \leq \liminf_{n \to \infty} \cE(f_n)$.

%

 2. Step:  For $\varepsilon,R > 0$ we consider the normal contractions $C_R\colon \R \to \R$, $C_R(x) = (x \wedge R) \vee (-R)$ and $D_\varepsilon \colon \R \to \R$, $D_\varepsilon(x) = x - C_\varepsilon(x)$. Then for any $g \in L^2(\mu)$ and $R,\varepsilon > 0$ we have $|C_R g| \leq R$ and  $D_\varepsilon g \in L^1(\mu)$.
 %
 %
 For $M > 0$ we now consider 
 \[ f^{(M)} = C_M(D_{1/M} f) \quad \text{and} \quad f^{(M)}_n = C_M(D_{1/M} f_n) . \]
 Then $f^{(M)} \in L^1(\mu) \cap D(\cE)$, $|f^{(M)}_n| \leq M$,  $f^{(M)}_n\tom f^{(M)}$, as $n \to \infty$,  and $f^{(M)} \to f$ in $L^2(\mu)$, as $M \to \infty$. Moreover, the compatibility with contractions implies $\cE(f^{(M)}_n) \leq \cE(f_n)$ for all $M > 0$. Using Step~1 and the $L^2$-lower semicontinuity of $\cE$, we infer 
 \begin{align*}
  \cE(f) &\leq \liminf_{M \to \infty} \cE(f^{(M)}) \leq \liminf_{M \to \infty}\liminf_{n \to \infty}  \cE(f^{(M)}_n) \leq \liminf_{n \to \infty} \cE(f_n). \hfill \qedhere
 \end{align*}
\end{proof}

If $\cE$ is proper and $\cE_e$ exists, it is a generalized semimodular on $L^0(\mu)$ and we write $M(\cE_e) = {\rm lin} D(\cE_e) \subseteq L^0(\mu)$ for the modular space of $\cE_e$ and denote by $\av{\cdot}_{L_e} = \av{\cdot}_{L,\cE_e}$ the induced Luxemburg seminorm. As discussed in Proposition~\ref{lemma:extended luxemburg norm}, $\av{\cdot}_{L_e}$ is the lower semicontinuous relaxation of $\av{\cdot}_L$ on $L^0(\mu)$ (when both seminorms are extended by $\infty$ outside the corresponding modular spaces). In particular, $\av{\cdot}_{L_e}$ is also lower semicontinuous with respect to local convergence in measure.  

\begin{definition}[Extended Dirichlet form]
 If $\cE$ is a nonlinear Dirichlet form (resp. a nonlinear order preserving form and $\cE_e$ exists), we call $\cE_e$ the {\em extended nonlinear Dirichlet form} (resp. the {\em extended nonlinear order preserving form}) and $(M(\cE_e),\av{\cdot}_{L_e})$  the {\em extended Dirichlet space} (resp. the {\em extended modular space}) of $\cE$. 
\end{definition}

\begin{remark}\label{remark:existence extended form}
\begin{enumerate}[(a)]
 \item The existence of the extended Dirichlet space for regular bilinear Dirichlet forms was first established in \cite{Sil}, without the regularity assumption it is contained in \cite{Schmu2}. In both cases the proofs depend on some representation theorem for the forms and cannot be transferred to the nonlinear case. The proof given here is based on a simplified proof contained in \cite{Sch22}, which in some disguise already uses the formula for $\vr$ contained in  Proposition~\ref{lemma:alternative formula for vr}. 
 
 \item Instead of assuming that $\cE$ is a nonlinear Dirichlet form, the following is sufficient: $J_\lambda f \in L^1(\mu)$ for all $f \in L^1(\mu) \cap L^2(\mu)$, $\lambda > 0$, and $\cE(C_\varepsilon D_R f) \leq \cE(f)$ for all $\varepsilon > 0$ and $R > 0$. Indeed,  by \cite[Théorème~2.2]{BP79} both are satisfied if $\cE(Cf) \leq \cE(f)$ for all normal contractions and all $f \in L^2(\mu)$. 
 
 \item Using excessive functions and $h$-transforms, in the bilinear case it is possible to show the existence of $\cE_e$ also for order preserving forms, see \cite{Schmu2,Sch22}. In view of the previous remark, we would need a measurable function $h \colon X \to (0,\infty)$ such that the functional 
 $$\cE^h \colon L^2(h^2 \mu) \to [0,\infty], \quad \cE^h(f) = \cE(h f)$$ 
 satisfies $\cE^h(Cf) \leq \cE^h(f)$ for all normal contractions and all $f \in L^2(h^2 \mu)$. In the bilinear case it is easy to always construct such a function (for a slightly perturbed form - to be precise). We were not able to establish a similar result for general nonlinear order preserving forms, see also the discussion in Subsection~\ref{subsection:excessive}. Hence, we formulated (B) of the theorem as it is for nonlinear Dirichlet forms only.

 \item For linear Markovian semigroups on $L^p$-spaces associated nonlinear (i.e. non-quadratic) Dirichlet forms on $L^p$ were introduced in \cite{HJ04} and  extended Dirichlet spaces were constructed in \cite{JS10}. While some of the theory developed in this text also applies to functionals on $L^p$, to us the connection of the nonlinear Dirichlet forms on $L^2$ considered here and the ones constructed in \cite{HJ04} is unclear.  The same holds true for the different extended Dirichlet spaces. It may well be that the extended Dirichlet spaces constructed in \cite{JS10} also appear as extended Dirichlet spaces in our setting.

\end{enumerate}

\end{remark}


The following proposition provides basic properties of the extended functional $\cE_e$. To better state it we say that $(f_n)$ in $D(\cE)$ is an {\em approximating sequence} for $f \in L^0(\mu)$ (cf. Definition~\ref{definition:approximating sequence}) if $(f_n)$ is $\av{\cdot}_L$-Cauchy and satisfies $f_n \tom f$. Moreover, the functional $\cE$ is said to satisfy the {\em $\Delta_2$-condition} (cf. Definition~\ref{definition:delta 2 condition}) if there exists a constant $K \geq 0$ such that $\cE(2f) \leq K\cE(f)$ for all $f \in L^2(\mu)$.

\begin{proposition}[Basic properties of extended functionals]\label{proposition:properties of extended forms}
 Let $\cE \colon L^2(\mu) \to [0,\infty]$ be a proper lower semicontinuous convex functional such that $\cE_e$ exists. If $\cE$ is reflexive, then for $f \in L^0(\mu)$ the following holds. 
 \begin{enumerate}[(a)]
  \item There exists an approximating sequence $(f_n)$ for $f$ if and only if $f \in M(\cE_e)$. In this case, $\av{f}_{L_e} = \lim_{n\to \infty} \av{f_n}_L$. 
  \item If $\cE$ satisfies the $\Delta_2$-condition, then $M(\cE_e) = D(\cE_e)$ and $(f_n)$ is an approximating sequence for $f$ if and only if $(f_n)$ is $\cE$-Cauchy and $f_n \tom f$. In this case, 
 $$\cE_e(f) = \lim_{n \to \infty} \cE(f_n).$$
  \item If $w \colon X \to (0,\infty)$ is measurable and $(f_n)$ in $L^2(w\mu)$ such that $f_n \to  f$ weakly in $L^2(w \mu)$, then 
  $$\cE_e(f) \leq \liminf_{n \to \infty} \cE_e(f_n) \text{ and }\av{f}_{L_e} \leq \liminf_{n \to \infty} \av{f_n}_{L_e}.$$
  %
 
%
%
 \end{enumerate}
\end{proposition}
\begin{proof}
 (a) + (b): This is contained in Theorem~\ref{theorem:description extended space}.
 
 (c): We only show the statement for $\cE_e$, the proof for $\av{\cdot}_{L_e}$ follows along the same lines. Without loss of generality we assume $\liminf_{n \to \infty} \cE_e(f_n) = \lim_{n \to \infty} \cE_e(f_n) < \infty$ (else pass to a suitable subsequence). By the Banach-Saks theorem there exists a subsequence $(f_{n_k})$ such that the Césaro-means $g_N = \frac 1 N \sum_{k=1}^N f_{n_k} \to f$ in $L^2(w\mu)$. Since $L^2(w\mu)$ continuously embeds into $L^0(\mu)$, we infer from the lower semicontinuity of $\cE_e$ and its convexity
 \begin{align*}
 \cE_e(f) &\leq \liminf_{N \to \infty} \cE_e(g_N) \leq \liminf_{N \to \infty} \frac 1 N \sum_{k=1}^N \cE_e(f_{n_k}) = \lim_{n \to \infty} \cE_e(f_n). \hfill \qedhere  
 \end{align*}
%
%
%
 %
\end{proof}

\begin{proposition}[Extended forms are compatible with contractions]\label{proposition:compatibility with contractions extended forms}
 Let $\cE \colon L^2(\mu) \to [0,\infty]$ be a proper lower semicontinuous convex functional such that $\cE_e$ exists. 
 \begin{enumerate}[(a)]
%
  \item If $\cE$ is a nonlinear order preserving form, then $D(\cE_e)  \cap L^2(\mu)  = D(\cE)$ and $M(\cE_e) \cap L^2(\mu) = M(\cE)$.
 
  \item If $\cE$ is a nonlinear order preserving form,  then $\cE_e$ satisfies the first Beurling-Deny criterion and $\av{|f|}_{L_e} \leq \av{f}_{L_e}$ for all $f \in M(\cE_e)$.
  \item If $\cE$ is a nonlinear Dirichlet form, then $\cE_e$ satisfies the second Beurling-Deny criterion and $\av{Cf}_{L_e} \leq \av{f}_{L_e}$ for all $f \in M(\cE_e)$ and all normal contractions $C \colon \R \to \R$.
 \end{enumerate}
\end{proposition}
\begin{proof}
 (a): It suffices to show  $D(\cE) \supseteq D(\cE_e) \cap L^2(\mu)$. For $f \in D(\cE_e) \cap L^2(\mu)$ we choose a sequence $(f_n)$ in $D(\cE)$ with $f_n \tom f$ and $\cE(f_n) \to \cE_e(f)$. By the dominated convergence theorem the limits
 $$f = \lim_{n \to \infty} (f_n \wedge |f|) \vee(-|f|)  =\lim_{n \to \infty} \lim_{m \to \infty} (f_n \wedge |f_m|) \vee(-|f_m|) $$
 hold in $L^2(\mu)$. Hence, the $L^2$-lower semicontinuity of $\cE$ and Proposition~\ref{proposition:contraction properties resolvent} yield 
 \begin{align*}
 \cE(f) &\leq \liminf_{n\to \infty}\liminf_{m\to \infty}\cE((f_n \wedge |f_m|) \vee(-|f_m|))\\
  &\leq \liminf_{n\to \infty}\liminf_{m\to \infty} (\cE(f_n) + 2\cE(f_m))\\
  &= 3 \cE_e(f) < \infty,
 \end{align*}
 showing $f \in D(\cE)$. 
 
 (b) + (c): We first show the statements for $\cE_e$. Let $C$ be a normal contraction such that $\cE(f + Cg) + \cE(f - Cg) \leq \cE(f+g) + \cE(f-g)$ for all $f,g \in L^2(\mu)$. If $f + g,f-g \in D(\cE_e)$ (else there is nothing to show), there exist sequences $(h_n), (\tilde h_n)$ in $D(\cE)$ with $h_n \tom f + g$, $\tilde h_n \tom f - g$, $\cE(h_n) \to \cE_e(f + g)$ and $\cE(\tilde h_n) \to \cE_e(f - g)$, see Proposition~\ref{proposition:lsc relaxation}. Then $f_n = (h_n + \tilde h_n)/2 \tom f$ and $g_n = (h_n - \tilde h_n)/2 \tom g$.  Using $Cg_n \tom Cg$, the lower semicontinuity of $\cE_e$ and $\cE = \cE_e$ on $L^2(\mu)$, we infer 
 \begin{align*}
  \cE_e(f + Cg) + \cE_e(f - Cg) &\leq \liminf_{n \to \infty} (\cE_e(f_n + C g_n)  + \cE_e(f_n - C g_n))\\
  &=    \liminf_{n \to \infty} (\cE(f_n + C g_n)  + \cE(f_n - C g_n))\\
  &\leq \liminf_{n \to \infty} (\cE(f_n + g_n) + \cE(f_n - g_n))\\
  &= \cE_e(f + g) + \cE_e(f- g).
 \end{align*}
  Since $\av{\cdot}_{L_e}$ is the Luxemburg seminorm of $\cE_e$, the statement on $\av{\cdot}_{L_e}$ can be inferred from the Beurling-Deny criteria as in the proof of Lemma~\ref{lemma:contraction luxemburg seminorms}.
\end{proof}

%

\subsection{Excessive functions}\label{subsection:excessive}
In this subsection we briefly discuss excessive functions, which will feature at various places throughout the text.

Let $\cE \colon L^2(\mu) \to [0,\infty]$ be a nonlinear order preserving form. As in the bilinear case we say that a measurable function $h \colon X \to [0,\infty)$ is {\em $\cE$-excessive}, if $G_\alpha (\alpha h) \leq h$ for all $\alpha > 0$. In the following proposition we use the lower derivative $d^+ \cE_e$, see Lemma~\ref{lemma:lower derivative}.

\begin{proposition}[Characterization of excessive functions]\label{proposition:characterization of excessive functions}
Let $\cE \colon L^2(\mu) \to [0,\infty]$ be a nonlinear order preserving form and let $h \colon X \to [0,\infty)$ measurable. 
\begin{enumerate}[(a)]
 \item $h$ is $\cE$-excessive if and only if $\cE(f \wedge h) \leq \cE(f)$ for all $f \in L^2(\mu)$. 
 \item If $\cE_e$ exists, then $h$ is $\cE$-excessive if and only if $\cE_e(f \wedge h) \leq \cE_e(f)$ for all $f \in L^0(\mu)$. 
 \item If $\cE_e$ exists and $h \in D(\cE_e)$, then the following assertions are equivalent: 
 \begin{enumerate}[(i)]
  \item $h$ is $\cE$-excessive.
  \item $\cE_e(h) = \min \{\cE_e(f) \mid f \in L^0(\mu) \text{ with } f \geq h\}$
  \item $d^+\cE_e(h,\varphi) \geq 0$ for all nonnegative $\varphi \in M(\cE_e)$. 
 \end{enumerate}
  If $h \in D(\cE)$, then in (ii) and (iii) $\cE_e$ can be replaced by $\cE$. 
%
%
\end{enumerate} 
\end{proposition}
\begin{proof}
 (a): Since $G_\alpha$ is order preserving, $  G_\alpha (\alpha h) \leq h$ is equivalent to the inclusion $G_\alpha (\alpha  K_h) \subseteq K_h$, where $K_h$ is the closed convex cone 
 $$K_h = \{f \in L^2(\mu) \mid f \leq h\}.$$
 Moreover, the Hilbert space projection $P_h$ onto $K_h$ is given by $P_h f = f \wedge h$. With this at hand the statement follows from the standard characterization of the invariance of closed convex sets under the resolvent,  see e.g. \cite[Theorem~3.4]{CG03}. It states that $K_h$ is invariant under the resolvent, i.e., $G_\alpha (\alpha K_h) \subseteq K_h$ for all $\alpha > 0$, if and only if $\cE(P_h f) \leq \cE(f)$, $f \in L^2(\mu)$. This is precisely the statement of (a). 
 
 (b): If $h$ is $\cE$-excessive, the inequality $\cE_e(f \wedge h) \leq \cE_e(f)$ follows from (a) and the lower semicontinuity of $\cE_e$, cf. proof of Proposition~\ref{proposition:compatibility with contractions extended forms}. Now assume that $\cE_e(f \wedge h) \leq \cE_e(f)$ for $f \in L^0(\mu)$. Let $f \in L^2(\mu)$. Without loss of generality we assume $f \in D(\cE)$ (else there is nothing to show). Using $\cE_e(f \wedge h) \leq \cE_e(f) = \cE(f) < \infty$ and $h \geq 0$, we infer $f \wedge h \in D(\cE_e) \cap L^2(\mu)$, which by Proposition~\ref{proposition:compatibility with contractions extended forms} implies $f \wedge h \in D(\cE)$. With this at hand the desired inequality follows from the inequality for $\cE_e$ and the fact that $\cE_e$ and $\cE$ agree on $D(\cE)$. 
 
 (c): (i) $\Rightarrow$ (ii):  Assume that $h$ is excessive and $f \geq h$. Then $h = f \wedge h$ and from (b) we infer $\cE_e(h) = \cE_e(f \wedge h) \leq \cE_e(f)$, showing that $h$ is a minimizer.
 
 (ii) $\Rightarrow$ (i): Assume that $\cE_e(h) = \min \{\cE_e(f) \mid f \in L^0(\mu) \text{ with } f \geq h\}$ holds. The first Beurling-Deny criterion implies 
 $$\cE_e(f \wedge h) + \cE_e(f \vee h) \leq \cE_e(f) + \cE_e(h).$$
 For $f \in D(\cE_e)$ we have $f \vee h \in D(\cE_e)$ and $f \vee h \geq h$. Hence, $\cE_e(f \vee h) \geq \cE_e(h)$ and substracting $\cE_e(f \vee h)$ on both sides of the previous inequality yields $\cE_e(f \wedge h) \leq \cE_e(f)$. 
 
 (ii) $\Rightarrow$ (iii): This is obvious. 
 
 (iii) $\Rightarrow$ (ii): Let $f \in D(\cE_e)$ with $f \geq h$. Then $f-h \geq 0$ with $f-h \in M(\cE_e)$ and (iii) imply $d^+ \cE_e(f,f-h) \geq 0$. According to Lemma~\ref{lemma:lower derivative} we have 
 $$\cE_e(h) + d^+ \cE_e(h,f-h) \leq \cE_e(f)$$
 and arrive at $\cE_e(h) \leq \cE_e(f)$.  
\end{proof}

\begin{remark}\label{remark:excessive functions 1}
 \begin{enumerate}[(a)]
  \item If $\cE$ (and hence also $\cE_e$) is a quadratic form, then $D(\cE_e)$ is a vector space and for $f,g \in D(\cE_e)$ we have
  $$d^+ \cE_e(f,g) = \frac{1}{2} \left(\cE_e(f + g) - \cE_e(f-g) \right).$$
  Hence, $d^+ \cE_e$ equals two times the bilinear form induced by $\cE_e$ on $D(\cE_e) \times D(\cE_e)$ via polarization. In particular, our characterization of excessive functions in $D(\cE_e)$  equals the known one via weakly superharmonic functions, see e.g. \cite[Theorem 2.52]{Schmi3}.
  \item If $\cE$ is a nonlinear Dirichlet form, the compatibility with normal contractions implies that nonnegative constant functions are excessive. The existence of further excessive functions, in particular of excessive functions in $D(\cE_e)$, is discussed in Section~\ref{section:potential theory}.
  \item We consider the functional $\cE_1 \colon L^2(\mu) \to [0,\infty]$, $\cE_1(f) = \cE(f) + \frac{1}{2}\av{f}^2$. If $G_\alpha$, $\alpha > 0$, denotes the resolvent of $\cE$, then it is readily verified that $G_{1 + \alpha}$, $\alpha > 0$, is the resolvent of $\cE_1$, (cf. Lemma~\ref{lemma:perturbed form} below). If $\psi \in L^2(\mu)$ is nonnegative, then $h = G_1 \psi \geq 0$ and the resolvent identity implies $G_{1 + \alpha} (\alpha h) \leq h$, $\alpha > 0$. Hence, $h$ is $\cE_1$-excessive. 
 \end{enumerate}
\end{remark}

Lemma~\ref{lemma:contraction luxemburg seminorms} shows that contractions operate on the Luxemburg seminorm if they operate on $\cE$. However, due to the nonlinearity, a stronger property than excessivity has to be assumed to transfer the inequality $\cE(f \wedge h) \leq \cE(f)$ to the Luxemburg seminorm.

\begin{proposition}
 Let $\cE \colon L^2(\mu) \to [0,\infty]$ be a nonlinear order preserving form and let $h \colon X \to [0,\infty)$ measurable such that for all $\lambda > 0$ the function $\lambda h$ is $\cE$-excessive. Then $\av{f \wedge h}_L \leq \av{f}_L$ for all $f \in M(\cE)$. 
\end{proposition}
\begin{proof}
 Let $\lambda > 0$ such that $\cE(\lambda^{-1} f) \leq 1$. Using Proposition~\ref{proposition:characterization of excessive functions} and that $\lambda^{-1}h$ is $\cE$-excessive, we arrive at 
 $$\cE(\lambda^{-1}(f \wedge h)) = \cE((\lambda^{-1}f) \wedge (\lambda^{-1}h)) \leq \cE(\lambda^{-1} f) \leq 1.$$
 This shows $\av{f \wedge h}_L \leq \av{f}_L$.
\end{proof}

\begin{remark}\label{remark:excessive functions 2}
\begin{enumerate}[(a)]
  \item Assume that for some $1 \leq p < \infty$ the functional $\cE$ is $p$-homogeneous, i.e., $\cE(\lambda f) = |\lambda|^p \cE(f)$ for all $\lambda  \in \R$ and $f \in L^2(\mu)$.  In this case, $h$ is $\cE$-excessive if and only if $\lambda h$ is $\cE$-excessive for one/any $\lambda > 0$ and the statement of the previous proposition holds for all $\cE$-excessive functions. Indeed, in the $p$-homogeneous case we have $\av{f}_L = \cE(f)^{1/p}$ and so the inequality for the Luxemburg seminorm is trivial. 
 \item Functions $h \colon X \to [0,\infty)$ with $\av{f \wedge h}_L \leq \av{f}_L$, $f \in M(\cE)$, play a role below, see Remark~\ref{remark:weak hardy and poincare}, but we do not have any good existence result for them (except in the homogeneous case).   For proving existence of excessive functions in Section~\ref{section:potential theory} we use their characterization as minimizers given in Proposition~\ref{proposition:characterization of excessive functions}~(c). We do not know whether a similar characterization holds for functions $h$ with $\av{f \wedge h}_L \leq \av{f}_L$, $f \in M(\cE)$. The reason is that the crucial inequality $\cE(f\wedge h) + \cE(f \vee h) \leq \cE(f)  + \cE(h)$ that is used in the proof of  Proposition~\ref{proposition:characterization of excessive functions}~(c) need not hold for the Luxemburg seminorm - indeed it even fails for most quadratic forms.
 
 \item For bilinear order preserving forms excessive functions play a major role. As discussed already in Remark~\ref{remark:existence extended form}, if $\cE$ is an order preserving quadratic form and $h$ is $\cE$-excessive, then the $h$-transform $\cE^h \colon L^2(h^2\mu) \to [0,\infty]$, $\cE^h(f) = \cE(hf)$ is a bilinear Dirichlet form. Hence, with this transformation, many results for bilinear Dirichlet forms can be transferred to bilinear order preserving forms. In the nonlinear case this does not seem to hold. Excessivity of $h$ is not enough to ensure that $\cE^h$ is a nonlinear Dirichlet form and,  similar to the previous remark, general existence results for functions $h$ for which $\cE^h$ is a nonlinear Dirichlet form remain unclear. This is the main reason why below we state some results for nonlinear order preserving forms, while others are only stated for nonlinear Dirichlet forms.  
\end{enumerate}

\end{remark}

\subsection{Invariant sets and irreducibility} In this subsection we study the kernel of the Luxemburg seminorm of extended Dirichlet forms with the help of invariant sets. 

Let $\cE \colon L^2(\mu) \to [0,\infty]$ be proper, lower semicontinuous and convex. A measurable set $A \subseteq X$ is called {\em $\cE$-invariant} if 
$$\cE(1_A f) \leq \cE(f), \quad f \in  L^2(\mu).$$
We call $\cE$ {\em irreducible} if every $\cE$-invariant set is null or co-null.

The collection of all $\cE$-invariant sets is denoted by $\mathfrak B_{\rm inv}$ and we call 
$$\fA_{\rm inv} = \{A \subseteq X \mid A \in \mathfrak B_{\rm inv} \text{ or } X \setminus A \in \mathfrak B_{\rm inv}\}$$
the {\em invariant $\sigma$-algebra} of $\cE$, which is indeed a $\sigma$-algebra by the following lemma.

\begin{lemma}\label{lemma:invariant sigma algebra}
 $\fA_{\rm inv}$ is the $\sigma$-algebra generated by $\mathfrak B_{\rm inv}$.  
\end{lemma}
\begin{proof}
 Clearly, $\emptyset, X \in \fA_{\rm inv}$ and $\fA_{\rm inv}$ is stable under taking complements. Hence, it remains to prove that $\fA_{\rm inv}$ is stable under taking countable intersections. This however follows from $1_{\bigcap_{n =1}^\infty A_n} f = \lim_{N \to \infty} \prod_{n = 1}^N 1_{A_n} f$ and the lower semicontinuity of $\cE$. 
\end{proof}

The operator of multiplication by $1_A$ is the orthogonal projection to the closed subspace 
$$I_A = \{f \in L^2(\mu) \mid f = 0 \text{ a.s.\@ on } X \setminus A\}.$$
It is known that the projection onto $I_A$ reduces the value of $\cE$ if and only if $G_\alpha (\alpha \cdot)$, $\alpha > 0$, leaves $I_A$ invariant, see e.g. \cite[Theorem~3.4]{CG03}. Hence, $A$ is $\cE$-invariant if and only if $G_\alpha (\alpha I_A) \subseteq I_A$, $\alpha >0$, which is equivalent to 
$$G_\alpha(1_A f) = 1_A G_\alpha (1_A f), \quad\alpha > 0,\, f \in L^2(\mu).$$

\begin{theorem}\label{theorem:kernel irreducibility}
 Let $\cE$ be a nonlinear Dirichlet form. 
 \begin{enumerate}[(a)]
  \item For every $h \in \ker \av{\cdot}_{L_e}$ and $\alpha \geq 0$ the set $\{h > \alpha\}$ is $\cE$-invariant.
  \item If $\cE$ is irreducible, then $\ker \av{\cdot}_{L_e} \subseteq \R \cdot 1$. 
 \end{enumerate}
\end{theorem}
\begin{proof}
(a): The assumption $\av{h}_{L_e} = 0$ implies $\cE_e(nh) = 0$ for all $n \geq 1$, see Lemma~\ref{lemma:equivalent luxemburg norms}. Now consider the functions $h_n = (n(h-\alpha)_+) \wedge 1$. These are normal contractions of $nh$ and hence satisfy $\cE_e(h_n) \leq \cE_e(nh)=0$. For $f \in L^2(\mu)$ we have $f 1_{\{h > \alpha\}} = \lim_{n \to \infty} (f \wedge h_n) \vee (-h_n)$ in $L^2(\mu)$. Moreover, 
since  $D(\cE_e) \cap L^2(\mu) = D(\cE)$, see Proposition~\ref{proposition:compatibility with contractions extended forms}, we have $(f \wedge h_n)\vee(-h_n) \in D(\cE)$. Using the $L^2$-lower semicontinuity of $\cE$ and that $\cE_e$ satisfies the first Beurling-Deny criterion, we infer
\begin{align*}
 \cE(f 1_{\{h>\alpha\}}) &\leq \liminf_{n\to \infty} \cE((f \wedge h_n)\vee(-h_n)) \\
 &=  \liminf_{n\to \infty} \cE_e((f \wedge h_n)\vee(-h_n)) \\
 &\leq \liminf_{n\to \infty} (\cE_e(f) + 2 \cE_e(h_n))\\
 &= \cE(f). 
\end{align*}
For the last equality we again used $\cE = \cE_e$ on $L^2(\mu)$. 

(b): Let $\cE$ be irreducible and let $\av{h}_{L_e} = 0$. The compatibility of $\av{\cdot}_{L_e}$ with normal contractions discussed in Proposition~\ref{proposition:compatibility with contractions extended forms} shows $\av{h_+}_{L_e} = \av{h_-}_{L_e} = 0$, and so we can assume $h \geq 0$. For any $\alpha \geq 0$ the invariance of $\{h > \alpha\}$ and the irreducibility of $\cE$ yield $\mu(\{h > \alpha\}) = 0$ or $\mu(\{h \leq \alpha\}) = 0$. This can only hold for constant functions.  
\end{proof}

\begin{lemma}\label{lemma:invariance and extended form}
 Let $\cE \colon L^2(\mu) \to [0,\infty]$ be a nonlinear order preserving form for which $\cE_e$ exists. Then $A \subseteq X$ is $\cE$-invariant if and only if $\cE_e(1_A f) \leq \cE_e(f)$ for all $f \in L^0(\mu)$.  
\end{lemma}
\begin{proof}
 First assume that $A$ is $\cE$-invariant. For $f \in D(\cE_e)$ we choose a sequence $(f_n)$ in $D(\cE)$ with $f_n  \tom f$ and $\cE(f_n) \to \cE_e(f)$. The lower semicontinuity of $\cE_e$  and $1_A f_n \in D(\cE)$, which follows from $\cE(1_A f_n) \leq \cE(f_n)$, imply 
 $$\cE_e(1_A f) \leq \liminf_{n \to \infty} \cE(1_A f_n) \leq \liminf_{n \to \infty} \cE(f_n) = \cE_e(f).$$

 Now assume $\cE_e(1_A f) \leq \cE_e(f)$ for all $f \in L^0(\mu)$.
 It suffices to show $\cE(1_A f) \leq \cE(f)$ for $f \in D(\cE)$.
 Since $D(\cE) = D(\cE_e) \cap L^2(\mu)$ by Proposition~\ref{proposition:compatibility with contractions extended forms}, the inequality $\cE_e(1_A f) \leq \cE_e(f)$ implies $1_A f \in D(\cE)$ for such $f$.
 Since $\cE$ and $\cE_e$ agree on $D(\cE)$, we infer $\cE(1_A f) \leq \cE(f)$.
\end{proof}

\begin{corollary}
Let $\cE$ be a nonlinear Dirichlet form. If $\av{1}_{L_e} = 0$, then $\ker \av{\cdot}_{L_e} = L^0(X,\fA_{\rm inv},\mu)$. In particular, $\ker \av{\cdot}_{L_e} = \R \cdot 1$ implies that $\cE$ is irreducible.  
\end{corollary}

\begin{proof}
By definition the identity $\av{1}_{L_e} = 0$ implies $\cE_e(n) = 0$ for all $n \in \N$.  Assume that $A \subseteq X$ is $\cE$-invariant. From Lemma~\ref{lemma:invariance and extended form} we infer $\cE_e(n 1_A) \leq \cE_e(n) = 0$ for $n \in \N$, showing $\av{1_A}_{L_e} = 0$. Since the kernel of the extended Luxemburg seminorm is a vector space and $1_A, 1 \in \ker \av{\cdot}_{L_e}$, we infer $1_{X \setminus A}\in \ker \av{\cdot}_{L_e}$. This implies that simple functions of the form $\sum_{k = 1}^N \alpha_k 1_{A_k}$ with $\alpha_k \in \R$, $A_k \in \fA_{\rm inv}$, belong to $\ker \av{\cdot}_{L_e}$. Such functions are dense in $L^0(X,\fA_{\rm inv},\mu)$ with respect to local convergence in measure. Hence, the lower semicontinuity of $\av{\cdot}_{L_e}$ implies 
$\ker \av{\cdot}_{L_e} \supseteq L^0(X,\fA_{\rm inv},\mu).$
 
Now assume that $\av{h}_{L_e} = 0$. The compatibility with normal contractions implies $\av{h_+}_{L_e}  = \av{h_-}_{L_e} = 0$ and so we can assume $h \geq 0$. Theorem~\ref{theorem:kernel irreducibility} implies that for $\alpha \geq 0$ the set $\{h > \alpha\}$ is $\cE$-invariant. This implies $h \in L^0(X,\fA_{\rm inv},\mu)$.

Now assume that $\ker \av{\cdot}_{L_e} = \R \cdot 1$. What we already proved yields the identity $L^0(X,\fA_{\rm inv},\mu) = \R \cdot 1$. This implies that each $A \in \fA_{\rm inv}$ is null or co-null, showing the irreducibility of $\cE$.
\end{proof}

%


\section{Criticality theory}   \label{section:criticality theory}

In this section we prove the main results of this paper, namely characterizations of criticality and subcriticality. Roughly speaking, criticality means that the potential operator $G$ is infinite on a large class of functions and subcriticality is the property that the potential operator is finite for one strictly positive function. We show how these concepts are related to the validity of weighted Hardy inequalities and  completeness of the extended Dirichlet space. Some results of this section do not only hold for nonlinear Dirichlet forms but also for nonlinear order preserving forms. Since this might be of interest, we note where this is the case and formulate the theorems accordingly.

\subsection{The Green operator and Hardy inequalities}\label{subsection:green and hardy}

In this subsection we characterize finiteness of the Green operator in terms of a Hardy inequality. Only order preservation is required here. 

Let $\cE$ be a nonlinear order preserving form and let $w \in L^\infty_+(\mu)$. We define the functional $\cE_w \colon L^2(\mu) \to [0,\infty]$ by
$$\cE_w(f) = \cE(f) + \frac{1}{2} \int_X |f|^2 w d\mu.$$
It is readily verified that $\cE_w$ is a nonlinear order preserving form with $D(\cE_w) = D(\cE)$. The resolvent of $\cE_w$ is denoted by $(G^w_\alpha)$, $\alpha > 0$, and we write $G^w$ for its Green operator. 

\begin{lemma} \label{lemma:perturbed form}
 Assume the situation just described. 
 \begin{enumerate}[(a)]
  \item $\partial \cE_w = \partial \cE + w$ (here $w$ denotes the operator of multiplication by $w$).
  \item $G^w_\alpha f = G_\alpha (f - w G^w_\alpha f)$, $f \in L^2(\mu)$, $\alpha > 0$.
  \item If $\tilde w \in L^\infty_+(\mu)$, then $G^w_\alpha f =  G^{\tilde w}_\alpha (f + (\tilde w - w) G^w_\alpha f)$,  $f \in L^2(\mu)$, $\alpha > 0$.  
 \end{enumerate}
\end{lemma}
\begin{proof}
(a): The subgradient of the continuous functional $f \mapsto \frac{1}{2} \int_X |f|^2 w d\mu$ is the operator of multiplication by $w$. With this observation the statement is straightforward. 

(b) + (c): Using (a) these can be proven exactly as the resolvent formula. 
\end{proof}

\begin{proposition}\label{proposition:hardy}
 Let $\cE$ be a nonlinear order preserving form. For  $w \in L^+(\mu)$ let $K(w) = \int_X w Gw d\mu \in [0,\infty]$ (in the integral we use the convention $0 \cdot  \infty = 0$).
 \begin{enumerate}[(a)]
  \item  For all $f \in M(\cE)$ we have
 $$\int_X |f| w d\mu \leq (1 + K(w))\av{f}_L.$$
 \item If for some $C \geq 0$ and all $f \in M(\cE)$ the inequality
  $$\int_X |f| w d\mu \leq C \av{f}_L$$
  holds, then $K(w/C) \leq 1$.
 \end{enumerate}
\end{proposition}
\begin{proof}
 (a): It suffices to consider the case $K(w) < \infty$ and, using the monotone convergence theorem, it suffices to consider $w \in L^2_+(\mu)$. Since $w - \alpha G_\alpha w \in \partial \cE(G_\alpha w)$, Lemma~\ref{lemma:perturbed form} applied to the constant function $\alpha$ yields $w \in \partial \cE_\alpha(G_\alpha w).$ We infer 
 $$\as{w,f - G_\alpha w} \leq \cE_\alpha(f) - \cE_\alpha(G_\alpha w)$$
 for all $f \in L^2(\mu)$. Rearranging this inequality and letting $\alpha \to 0+$ leads to 
 $$\as{w,f} \leq \cE(f) + K(w)$$
 for all $f \in L^2(\mu)$. Since $\cE$ satisfies the first Beurling-Deny criterion and is symmetric, this implies $\as{w,|f|} \leq \cE(f) + K(w)$ for all $f \in L^2(\mu)$. The previous  inequality and $\cE(f) \leq 1$ if and only if $\av{f}_L \leq 1$ yield the claim. 
 %
 
%
%
%
%
 %
 (b): It suffices to treat the case $C = 1$. Without loss of generality we assume $w \in L^2_+(\mu)$. It suffices to show $\as{w,G_\alpha w} \leq 1$ for all $\alpha > 0$. Since $w - \alpha G_\alpha w \in \partial \cE(G_\alpha f)$, we have 
 $$\as{w - \alpha G_\alpha w, 0 - G_\alpha w} \leq \cE(0) - \cE(G_\alpha w),$$
 which implies $\cE(G_\alpha w) \leq \as{w,G_\alpha w} - \alpha \av{G_\alpha w}^2$.
 In particular, $\cE(G_\alpha w) < \infty$ so that $G_\alpha w \in M(\cE)$.
 Therefore, the assumed inequality yields $\as{w,G_\alpha w} \leq \av{G_\alpha w}_L$.
 We show $\av{G_\alpha w}_L \leq 1$. 
 
 Assume this were not the case, i.e., $\av{G_\alpha w}_L > 1$. Lemma~\ref{lemma:equivalent luxemburg norms} then shows $\av{G_\alpha w}_L \leq \cE(G_\alpha w)$. Using all the inequalities, we infer 
 $$\av{G_\alpha w}_L  \leq \cE(G_\alpha w) \leq \as{w,G_\alpha w} - \alpha \av{G_\alpha w}^2 \leq  \av{G_\alpha w}_L - \alpha \av{G_\alpha w}^2. $$
 Since  $\alpha > 0$ and $\av{\alpha G_\alpha}_L > 1$ implies $G_\alpha w \neq 0$, this inequality can never hold and so we obtain $\av{G_\alpha w}_L \leq 1$.
\end{proof}

An inequality of the type $\int_X |f| w d\mu \leq C \av{f}_L$, $f \in M(\cE)$, as in the previous proposition, is called {\em   Hardy inequality} with weight $w$. We can improve the constants given in the previous proposition. 

\begin{corollary}[Optimal constants]
 Assume the situation of the previous proposition. For $w \in L^+(\mu)$ we let
$$\tilde K(w) = \inf\{ C > 0 \mid K(w/C) \leq 1\}$$
and  denote by 
$$\mu(w) = \sup \{\av{f w}_1 / \av{f}_L \mid f \in M(\cE)\}$$
the optimal constant in the Hardy inequality with weight $w$. Then $\tilde K(w) \leq \mu(w) \leq 2 \tilde K(w)$. 

\end{corollary}

\begin{proof}
For $C > \tilde K(w)$ the monotonicity of $K$ in the argument implies $K(w/C) \leq 1$. Hence, (a) of the previous proposition applied to $w/C$ yields $\mu(w) \leq 2C$.

Part (b) of the previous proposition implies $K(w/ \mu(w)) \leq 1$. Hence, $\tilde K(w) \leq \mu(w)$. 
\end{proof}

\begin{remark} Due to the lack of homogeneity also the constant $\tilde K(w)$ is not optimal. In the bilinear case (i.e. when $\cE$ is a quadratic form and the resolvent is linear) it is knwon that 
$$\mu(w) =  \sqrt{2K(w)} = \sqrt{2} \tilde K(w)$$
is the best constant, see e.g. \cite[Lemma~1.5.3]{FOT} for the case of bilinear Dirichlet forms. Note that at a first glance the factor $\sqrt{2}$ is missing in the statement of \cite[Lemma~1.5.3]{FOT}. However, the Green operator we use here is $1/2$-times the Green operator considered in \cite{FOT}, because for quadratic forms $\partial \cE$ is two times the self-adjoint generator of $\cE$.

\end{remark}

\begin{example}\label{example:completeness implies subcriticality I}
 Let $\cE \colon L^2(\mu) \to [0,\infty]$ be a nonlinear order preserving form for which $(M(\cE),\av{\cdot}_L)$ is a Banach space. Then there exists $C > 0$ such that $K(|f|/(C\av{f}_2)) \leq 1$ for all $f \in L^2(\mu)$. In particular, $G(f/(C \av{f}_2)) < \infty$ a.s.\@ if $f > 0$ a.s.\@ 
\end{example}
\begin{proof}
  $(M(\cE),\av{\cdot}_L)$ being a Banach space implies that $(M(\cE),\av{\cdot}_L)$ continuously embeds into $L^2(\mu)$, see the proof of Theorem~\ref{theorem:existence extended Dirichlet space} or Theorem~\ref{theorem:two imply the third} and Remark~\ref{remark:two imply the third}. Hence, there exists $C > 0$ such that 
  $$\left(\int_X |g|^2d\mu \right)^{1/2} \leq C \av{g}_L $$
  for all $g \in M(\cE)$.
  With this at hand the statement follows from the Cauchy-Schwarz inequality and the previous proposition. 
\end{proof}

\begin{theorem}[Finite Green operator vs. Hardy inequality]\label{theorem:existence of hardy weight}
 Let $\cE$ be a nonlinear order preserving form. The following assertions are equivalent.
 \begin{enumerate}[(i)]
  \item There exists a measurable $g \colon X \to (0,\infty)$ with $Gg < \infty$ a.s.\@  
  \item There exists a measurable $w \colon X \to (0,\infty)$ such that 
  $$\int_X |f| w d\mu \leq \av{f}_L, \quad f \in M(\cE).$$
  \end{enumerate}
   If $\cE_e$ exists, then the inequality in (ii) extends to $M(\cE_e)$.
  \end{theorem}
 
\begin{proof}
(i) $\Rightarrow$ (ii): Let $g \colon X \to (0,\infty)$ with $Gg < \infty$. Without loss of generality we assume $g \in L^1(\mu)$ and consider $w = g /(Gg \vee 1)$. Using that $G$ is order preserving, we obtain  
$$ \int_{X} w G w d\mu \leq \int_X \frac{g}{Gg \vee 1} Gg d\mu \leq \int_X g d\mu < \infty.  $$
Hence, Proposition~\ref{proposition:hardy} shows the claim. 

(ii) $\Rightarrow$ (i): Proposition~\ref{proposition:hardy} shows $\int_X w Gw d\mu \leq 1$. Since $w > 0$ a.s.\@, this implies $Gw < \infty$ a.s.\@

That the inequality in (ii) extends to $M(\cE_e)$ follows from Fatou's lemma. 
\end{proof}

\begin{remark}
 The novelty here is that finiteness of $G$ is related to weighted Hardy inequalities with respect to the Luxemburg seminorm. Once this is quantified in Proposition~\ref{proposition:hardy}, the statement of Theorem~\ref{theorem:existence of hardy weight} follows as in the bilinear case (cf. the discussion preceding \cite[Theorem~1.5.1]{FOT}).
\end{remark}

\subsection{Weak Hardy and Poincaré inequalities}\label{subsection:weak hardy and poincare}

In this subsection we study the completeness of the extended Dirichlet space $(M(\cE_e),\av{\cdot}_{L_e})$ and relate it to weak Hardy and Poincaré inequalities.

\begin{theorem}\label{theorem:weak hardy}
Let  $\cE$ be a nonlinear Dirichlet form. The following assertions are equivalent. 
\begin{enumerate}[(i)]
 \item $\ker \av{\cdot}_{L_e} = \{0\}$.
 \item $(M(\cE_e),\av{\cdot}_{L_e})$ is a Banach space. 
 \item The embedding $ (M(\cE_e),\av{\cdot}_{L_e}) \to L^0(\mu)$, $f \mapsto f$, is continuous.
 \item For one/all integrable $w \colon X \to (0,\infty)$ and one/all $1 \leq p < \infty$ there exists a decreasing function $\alpha \colon (0,\infty) \to (0,\infty)$ such that for all $r >0$
 $$\left(\int_X |f|^p w d\mu\right)^{1/p} \leq \alpha(r) \av{f}_{L_e} + r \av{f}_\infty, \quad   f \in M(\cE_e) \cap L^\infty(\mu).$$
\end{enumerate}
\end{theorem}
\begin{proof}
(iii) $\Rightarrow$ (ii): This follows from Theorem~\ref{theorem:two imply the third} and Remark~\ref{remark:two imply the third}, which states that the required implication from Theorem~\ref{theorem:two imply the third}  does not need reflexivity of $\cE$. 
 
 (ii) $\Rightarrow$ (i): $ \av{\cdot}_{L_e}$ being a norm implies $\ker \av{\cdot}_{L_e} = \{0\}$.

 (iv) $\Rightarrow$ (iii): Assume the inequality in (iv) is satisfied for $p \geq 1$, the integrable function $w \colon X \to (0,\infty)$ and $\alpha \colon (0,\infty) \to (0,\infty)$.  Let $(f_n)$ in $M(\cE_e)$ with $\av{f_n}_{L_e} \to 0$. We need to show $f_n \tom 0$.  
 
 Using (iv) and the compatibility of $\av{\cdot}_{L_e}$ with normal contractions, see Proposition~\ref{proposition:compatibility with contractions extended forms}, we obtain 
 \begin{align*}
  \left(\int_X (|f_n| \wedge 1)^p w d\mu\right)^{1/p} &\leq \alpha(r) \av{|f_n|\wedge 1}_{L_e} + r \av{|f_n| \wedge 1}_\infty\\
  &\leq \alpha(r) \av{f_n}_{L_e} + r
 \end{align*}
 for all $r > 0$. This shows $\int_X (|f_n| \wedge 1)^p w d\mu \to 0$, which implies $f_n \tom 0$.  
 
 (i) $\Rightarrow$ (iv): Let $w \colon X \to (0,\infty)$ integrable and let $p \geq 1$. Assume that (iv) does not hold for $w$ and $p$. Then there exists $R > 0$ and a sequence $(f_n)$ such that
 $$\left(\int_X |f_n|^p w d\mu\right)^{1/p} \geq n \av{f_n}_{L_e} + R \av{f_n}_\infty, \quad n \geq 1.$$
 By Proposition~\ref{proposition:compatibility with contractions extended forms} we can assume $f_n \geq 0$ (else consider $|f_n|$) and by rescaling we can assume  $\int_X |f_n|^p w d\mu = 1$. This implies $\av{f_n}_{L_e} \to 0$ and $\av{f_n}_\infty \leq 1/R$. Since $w\mu$ is a finite measure, by the uniform boundedness we can assume $f_{n} \to f$ weakly in $L^2(w\mu)$  (after passing to a subsequence if necessary). Using Proposition~\ref{proposition:properties of extended forms}, we infer 
$$\av{f}_{L_e} \leq \liminf_{n \to \infty}\av{f_n}_{L_e} =  0.$$
Our assumption (i) implies $f = 0$. Since $1 \in L^2(w\mu)$ and $f_n \geq 0$, the weak convergence in $L^2(w\mu)$ yields 
$\int_X |f_n| w d\mu = \int_X f_n w d\mu \to 0.$
Without loss of generality we can therefore assume $f_n \to  0$ $\mu$-a.s.\@ (else pass to a suitable subsequence). Using this convergence, the uniform boundedness of $(f_n)$ and the finiteness of $w \mu$,  Lebesgue's dominated convergence theorem implies 
$\int_X |f_n|^p w d\mu \to 0,$
which contradicts our assumption $\int_X |f_n|^p w d\mu = 1$.
 \end{proof}

\begin{theorem}\label{theorem:weak poincare}
Let  $\cE$ be a nonlinear Dirichlet form. If $\ker \av{\cdot}_{L_e} = \R \cdot 1$, then the following assertions are equivalent.
\begin{enumerate}[(i)]
 \item $(M(\cE_e)/\R\cdot 1,\av{\cdot}_{L_e})$ is a Banach space. 
 \item The embedding
 $$ (M(\cE_e)/\R \cdot 1,\av{\cdot}_{L_e}) \to L^0(\mu)/\R \cdot 1, \quad f + \R \cdot 1 \mapsto f + \R \cdot 1,$$
 is continuous, where $L^0(\mu)/\R \cdot 1$ is equipped with the quotient topology. 
 \item For one/all integrable $w \colon X \to (0,\infty)$ and one/all $1 \leq p < \infty$ there exists a decreasing function $\alpha \colon (0,\infty) \to (0,\infty)$ such that for all $r >0$
 $$\left(\int_X |f - \overline f|^p w d\mu\right)^{1/p} \leq \alpha(r) \av{f}_{L_e} + r \delta(f), \quad   f \in M(\cE_e) \cap L^\infty(\mu).$$
 Here, $\overline f = \frac{1}{\av{w}_1} \int_X f w d\mu$ and $\delta(f)  = \esssup f - \essinf f$. 
\end{enumerate}
\end{theorem}

\begin{proof}
 (i) $\Leftrightarrow$ (ii): Since $\cE_e$ is lower semicontinuous on $L^0(\mu)$, this equivalence is contained in Theorem~\ref{theorem:two imply the third} and Remark~\ref{remark:two imply the third}.
 
 (ii) $\Rightarrow$ (iii): Let $w \colon X \to (0,\infty)$ be integrable and $p \geq 1$ and assume the inequality in (iii) does not hold for any $\alpha$. Using $\ker \av{\cdot}_{L_e} = \ker \delta = \R \cdot 1$,  there exists $R > 0$ and a sequence $(f_n)$ in $M(\cE_e)$ with $\int_X |f_n|^p w d\mu = 1$, $\int_X f_n w d\mu = 0$ and 
 $$1 \geq n \av{f_n}_{L_e} + R \delta(f_n).$$
 This implies $\delta(f_n) \leq 1/R$ and $\av{f_n}_{L_e} \to 0$. Moreover, $\int_X f_n w d\mu = 0$ yields $\essinf f \leq 0 \leq \esssup f$, which implies $\av{f}_\infty \leq \delta(f) \leq 1/R$.

 Using $\av{f_n}_{L_e} \to 0$ and (ii) we find $C_n \in \R$, $n \in \R$, such that $f_n - C_n \tom 0$.  Below we show that $C_n \to 0$.  This implies $f_n \tom 0$ and Lebesgue's dominated convergence yields 
 $$0 = \lim_{n \to \infty} \int_X |f_n|^p w d\mu = 1,$$
 a contradiction. 
 
 $C_n \to 0$: The inequality $|C_n| \leq |f_n - C_n| + |f_n| \leq |f_n - C_n| + 1/R  \tom 1/R$ yields that $(C_n)$ is bounded so that also $(f_n - C_n)$ is uniformly bounded. With this at hand, $\int_X f_n w d\mu = 0$, $f_n - C_n \tom 0$  and Lebesgue's dominated convergence theorem yield
 $$-C_n \int_X w d\mu = \int_X (f_n - C_n) w d\mu \to 0,\text{ as } n \to \infty. $$
 Since  $\int_X w d\mu > 0$, this shows $C_n \to 0$. 
 
 (iii) $\Rightarrow$ (ii): It suffices to show that for $(f_n)$ in $M(\cE_e)$ with $\av{f_n}_{L_e} \to 0$ there exist constants $C_n \in \R$, $n \in \N$, such that $f_n - C_n \tom 0$. 
 
 Let $w \colon X \to (0,\infty)$ be integrable such that the  inequality in (iii) holds. We define $T \colon L^0(\mu) \to L^0(\mu)$, $Tf = (f \wedge 1) \vee (-1)$. Using the intermediate value theorem, for $n \in \N$ we choose $C_n \in \R$ such that 
 $$\int_X T(f_n - C_n) w d\mu = 0.$$
 The inequality in (iii) and $|T(f_n - C_n)| \leq 1$ imply  that for $r  > 0$ we have
 \begin{align*}
  \left(\int_X |T(f_n - C_n)|^p w d\mu\right)^{1/p} &\leq \alpha(r) \av{T(f_n - C_n)}_{L_e} + 2r \\
  &\leq \alpha(r) \av{f_n}_{L_e} + 2r,  
 \end{align*}
where for the last inequality we used the compatibility of $\av{\cdot}_{L_e}$ with normal contractions, see Proposition~\ref{proposition:compatibility with contractions extended forms}, and $\ker \av{\cdot}_{L_e} = \R \cdot 1$. This implies $T(f_n - C_n) \to 0$ in $L^1(w \mu)$, which in turn yields $f_n - C_n \tom 0$.  
\end{proof}

\begin{remark}
An inequality as in Theorem~\ref{theorem:weak hardy}~(iv) is called weak Hardy inequality and an inequality as in Theorem~\ref{theorem:weak poincare}~(iii) is called weak Poincaré inequality. Inequalities of this kind were introduced in \cite{RW01} for bilinear Dirichlet forms to study non-exponential convergence rates to equilibrium of the associated semigroups. While it may not be possible to extend these results to the fully nonlinear case, for $p$-homogeneous differentiable $\cE$ the arguments given in \cite{RW01} should give convergence rates for the corresponding nonlinear semigroups.  We refrain from giving details.

That weak Hardy/Poincaré inequalities can be used to characterize the completeness of the extended Dirichlet space was observed in \cite{Sch22}. The proofs given here for the nonlinear case are quite similar to the ones in \cite{Sch22} but heavily rely on the new Theorem~\ref{theorem:two imply the third}.
\end{remark}

\begin{remark}[Order preserving forms]\label{remark:weak hardy and poincare}
Both theorems remain valid for nonlinear order preserving forms under some conditions: Theorem~\ref{theorem:weak hardy} requires the existence of a function $h \colon X \to (0,\infty)$ such that $\av{f \wedge h}_L \leq \av{f}_L$ for all $f \in M(\cE)$ and Theorem~\ref{theorem:weak poincare} requires $\ker \av{\cdot}_{L_e} = \R \cdot h$ for some $h \colon X \to (0,\infty)$, see also Subsection~\ref{subsection:excessive} for more background on the existence of such functions. For the first theorem $\av{f}_\infty$ has to be replaced by $\av{f/h}_\infty$ and for the second one $\delta(f)$ has to be replaced by $\delta(f/h)$.
\end{remark}

\subsection{Criticality theory} 

In this subsection we bring together the results of the previous two subsections. Before proving the main results we return to properties of the functional $\cE_w$ defined for $w\in L^\infty_+(\mu)$ in Subsection~\ref{subsection:green and hardy}. 

\begin{lemma}\label{lemma:perturbed form II}
 Let $\cE$ be a nonlinear Dirichlet form and let $w \in L^\infty_+(\mu)$. 
 \begin{enumerate}[(a)]
  \item $G^w_\alpha (\alpha + w) \leq 1$, $\alpha > 0$. In particular, $G^w w \leq 1$. 
  \item $\cE  (G^w_\alpha f) \leq \as{f - (w+\alpha) G^w_\alpha f, G^w_\alpha f}$, for $\alpha > 0$ and $f \in L^2(\mu)$.
  \item If $w \in L^1(\mu) \cap L^\infty_+(\mu)$, then 
  $$\cE_e(G^w w) \leq \int_X w(1-G^ww) G^w w d\mu \leq \av{w}_1.$$
 \end{enumerate}
\end{lemma}
\begin{proof}
 (a): Let $\tilde \cE \colon L^2((w+\alpha)\mu) \to [0,\infty],\, \tilde \cE(f) =  \cE(f)$. It is a nonlinear Dirichlet form. The Hilbert spaces $L^2((w+\alpha)\mu)$ and $L^2(\mu)$ are equal as sets but carry  different inner products, which we denote by $\as{\cdot,\cdot}_{(w+\alpha)\mu}$ and $\as{\cdot,\cdot}_\mu$.  We prove (a) by showing that the resolvent $\tilde G_1$  of $\tilde \cE$  (which is defined with respect to the inner product $\as{\cdot,\cdot}_{(w+\alpha)\mu}$) satisfies the identity  $G^w_\alpha ((w+\alpha)f) = \tilde G_1 f$, $f \in L^2(\mu)$. The $L^\infty$-contractivity of $\tilde G_1$ then implies the claim. 
 
 By definition we have $g = \tilde G_1 f$ if and only if $f - g \in \partial \tilde \cE (g)$. The latter is equivalent to 
 $$\as{f-g,h - g}_{(w+\alpha)\mu} \leq \tilde \cE(h) - \tilde \cE(g)$$
 for all $h \in L^2((w+\alpha)\mu)$. Using the definition of $\tilde \cE$, this inequality is equivalent to 
 $$\as{(w + \alpha)(f - g), h - g} \leq \cE(h) - \cE(g)$$
 for all $h \in L^2(\mu)$. This statement is equivalent to $(w + \alpha)(f-g) \in \partial \cE(g)$, which, using $\partial \cE_w = \partial \cE + w$, is equivalent to 
 $(w + \alpha) f - \alpha g \in \partial \cE_w(g)$. The latter is equivalent to $g = G_\alpha^w ((w+\alpha)f)$.
 
 (b): We saw in Lemma~\ref{lemma:perturbed form} that $G_\alpha^w f = G_\alpha (f - w G_\alpha^w f)$. Using this and the equality $g - \alpha G_\alpha g \in \partial \cE(G_\alpha g)$ for $g = f - w G_\alpha^w f$, we infer 
 $$ \as{f - w G_\alpha^w f - \alpha  G_\alpha^w f, h - G_\alpha^w f} \leq  \cE(h) - \cE(G_\alpha^w f).$$
 Applied to $h = 0$, this yields 
 $$\cE(G_\alpha^w f) \leq  \as{f - (w + \alpha) G_\alpha^w f,  G_\alpha^w f}.$$

 (c): By definition we have $G_\alpha^w w \tom G^w w$, as $\alpha \to 0+$, and, since $w \in L^2(\mu)$, $G_\alpha^w w \in D(\cE)$.  The lower semicontinuity of $\cE_e$ and (b)  imply 
 $$\cE_e(G^w w) \leq \liminf_{\alpha \to 0+} \cE(G_\alpha^w w) \leq \liminf_{\alpha \to 0+} \as{w - (w + \alpha) G_\alpha^w w,  G_\alpha^w w}.  $$
 With this at hand the statement follows from Lebesgue's dominated convergence theorem and $G_\alpha^w w, G^w w \leq 1$.
%
%
%
%
%
%
\end{proof}

\begin{definition}[Criticality and subcriticality]
 Let $\cE \colon L^2(\mu) \to [0,\infty]$ be a nonlinear order preserving form. We say that $\cE$ is {\em critical} if for every $f \in L^1_+(\mu)$ we have $\mu(\{f > 0\} \cap \{G f < \infty\}) = 0$. We say that $\cE$ is {\em subcritical}  if there exists a measurable $f \colon X \to (0,\infty)$ such that $Gf < \infty$ $\mu$-a.s.\@ 
\end{definition}

\begin{remark}
 For bilinear Dirichlet forms criticality is also called {\em recurrence} and subcriticality is also called {\em transience}.
\end{remark}

\begin{example}\label{example:completeness implies subcriticality II}
 Example~\ref{example:completeness implies subcriticality I} shows that $\cE$ is subcritical if $(M(\cE),\av{\cdot}_L)$ is a Banach space. 
\end{example}

We say that $\cE$ satisfies the {\em weak $\Delta_2$-condition} if for all $(f_n)$ in $L^2(\mu)$ the convergence $\cE(f_n) \to 0$ implies $\cE(2f_n) \to 0$; see the discussion after Definition~\ref{definition:delta 2 condition} for more background. As discussed in Lemma~\ref{lemma:delta2 conditions}, the weak $\Delta_2$-condition is equivalent to  $\av{f_n}_L \to 0$ if and only if $\cE(f_n) \to 0$.

\begin{theorem}[Characterization subcriticality]\label{theorem:characterization subcriticality}
 Let $\cE$ be a nonlinear Dirichlet form satisfying the weak $\Delta_2$-condition. Then the following assertions are equivalent. 

 \begin{enumerate}[(i)]
  \item $\cE$ is subcritical. 
  \item The equivalent conditions of Theorem~\ref{theorem:existence of hardy weight} hold for $\cE$. 
  \item The equivalent conditions of Theorem~\ref{theorem:weak hardy} hold for $\cE$. 
 \end{enumerate}
\end{theorem}

\begin{proof}
 (i) $\Leftrightarrow$ (ii): Assertion (i) of Theorem~\ref{theorem:existence of hardy weight} is the definition of subcriticality. 
 
 (ii) $\Rightarrow$ (iii): By (ii) there exists $w \colon X \to (0,\infty)$ with $\int_X |f| w d\mu \leq \av{f}_{L_e}$ for all $f \in M(\cE_e)$. This implies that the kernel of $\av{\cdot}_{L_e}$ is trivial, which is one of the equivalent conditions of Theorem~\ref{theorem:weak hardy}.
 
 (iii) $\Rightarrow$ (ii): Let $w \in L^1(\mu) \cap L^\infty(\mu)$ such that $w > 0$ a.s.\@ and $\av{w}_1 \leq 1$. We set $w_n = n^{-1}w$.  Then $w_n(1 - G^{w_n}w_n) \geq 0$  by Lemma~\ref{lemma:perturbed form II}. We show that 
 $$W = \sum_{n = 1}^\infty \frac{1}{2^n} w_n(1 - G^{w_n} w_n) $$
 satisfies $W  > 0$ a.s.\@ and 
 $$\int_X |f| W d\mu \leq 2\av{f}_L, \quad f \in M(\cE).$$

 According to Lemma~\ref{lemma:perturbed form} and Lemma~\ref{lemma:perturbed form II}, we have $ G_\alpha (w_n(1 - G_\alpha^{w_n}w_n)) = G^{w_n}_\alpha w_n \leq 1$. Using $G^{w_n}_\alpha w_n \leq G^{w_n} w_n$, this implies  $G_\alpha (w_n(1-G^{w_n}w_n)) \leq 1$ and hence $G(w_n(1 - G^{w_n}w_n)) \leq 1$. This leads to 
 $$\int_X w_n(1 - G^{w_n}w_n) G(w_n(1 - G^{w_n}w_n)) d\mu \leq \av{w_n}_1 \leq \frac{1}{n}, $$
 which by Proposition~\ref{proposition:hardy} implies 
 $$\int_X |f| w_n(1 - G^{w_n}w_n) d\mu \leq (1 + 1/n)\av{f}_L, \quad f \in M(\cE).$$
 Hence, it remains to show   $W > 0$.
 
 Lemma~\ref{lemma:perturbed form II} yields $\cE_e(G^{w_n}w_n) \leq \av{w_n}_1 \leq n^{-1}$ and hence the weak $\Delta_2$-condition, which by Lemma~\ref{lemma:extended luxemburg norm} extends to $\cE_e$,  implies $\av{G^{w_n} w_n}_{L_e} \to 0$. This convergence and  (iii)  imply $G^{w_n}w_n \tom 0$, see Theorem~\ref{theorem:weak hardy}. Moreover, by Lemma~\ref{lemma:perturbed form}~(c), $n \mapsto G^{w_n}w_n$ is decreasing.  Combined, we obtain that the sets  $X_n = \{w_n (1 - G^{w_n} w_n) > 0\}$ satisfy $X_n \nearrow X$ a.s.\@, showing $W > 0$ a.s.\@  
\end{proof}

\begin{theorem}[Characterization criticality]\label{theorem:characterization criticality}
 Let $\cE$ be a nonlinear Dirichlet form  satisfying the weak $\Delta_2$-condition. Then the following assertions are equivalent. 
 \begin{enumerate}[(i)]
  \item $\cE$ is critical. 
  \item $1 \in M(\cE_e)$ and  $\av{1}_{L_e} = 0$. 
  \item There exists a sequence $(e_n)$ in $M(\cE)$ with $0 \leq e_n \leq 1$, $e_n \tom 1$ and $\av{e_n}_L \to 0$.
 \end{enumerate}

\end{theorem}
\begin{proof}
(i) $\Rightarrow$ (ii): Let $w \in L^1(\mu) \cap L^\infty(\mu)$ with $w > 0$. As in the proof of Theorem~\ref{theorem:characterization subcriticality} we infer $G(w(1-G^ww)) \leq 1$. Hence, criticality implies $w(1-G^w w) = 0$ and we obtain $G^w w = 1$. 
%
Using Lemma~\ref{lemma:perturbed form II}, we infer 
$$\cE_e(1) = \cE_e(G^w w) \leq \int_X w(1-G^w w) G^w w d\mu = 0.$$

The weak $\Delta_2$-condition for $\cE_e$ implies $\ker \cE_e = \ker \av{\cdot}_{L_e}$ and we arrive at $1 \in M(\cE_e)$ and $\av{1}_{L_e} = 0$. 

(ii) $\Leftrightarrow$ (iii): By our characterization of the extended Luxemburg seminorm $\av{1}_{L_e} = 0$ is equivalent to the existence of $(e_n)$ in $M(\cE)$ such that $e_n \tom 1$ and $\av{e_n}_L \to 0$. Since $\av{\cdot}_L$ is compatible with normal contractions, see Lemma~\ref{lemma:contraction luxemburg seminorms}, these $(e_n)$ can be chosen to satisfy $0 \leq e_n \leq 1$. 

(ii) $\Rightarrow$ (i): For $w \in L^1_+(\mu)$ let $K(w) = \int_X w Gw d\mu$. If $w \neq 0$, then (ii) implies that for any $n \geq 0$ we have
$$\int_X |1| w d\mu > 0 = n \av{1}_{L_e}.$$
Together with Proposition~\ref{proposition:hardy}, this implies  $K(w) = \infty$ for all $0 \neq w \in L^1_+(\mu)$. For $g \in L^1_+(\mu)$ we consider the function $w = g /(Gg \vee 1)$ (here division by $\infty$ is interpreted as $0$). As in the proof Theorem~\ref{theorem:existence of hardy weight}, it is readily verified that $K(w) \leq \av{g}_1$. Hence, $w = 0$, which implies $Gg = \infty$ on $\{g > 0\}$. This shows criticality.  
%
%
%
%
\end{proof}

\begin{corollary}\label{corollary:dichotomy}
   Let $\cE$ be a nonlinear Dirichlet form  satisfying the weak $\Delta_2$-condition. If $\cE$ is irreducible, then $\cE$ is either critical or subcritical. 
\end{corollary}
\begin{proof}
 Since $\cE$ is irreducible, we have $\ker \av{\cdot}_{L_e} \subseteq \R \cdot 1$, see Theorem~\ref{theorem:kernel irreducibility}.
 Because the kernel of a seminorm is a linear space, it is either $\{ 0 \}$ or $\R \cdot 1$.
 Now, Theorem~\ref{theorem:characterization subcriticality} and Theorem~\ref{theorem:characterization criticality} yield the claim.
\end{proof}

\begin{remark}
  For energies of $p$-Laplace type operators the previous corollary is known as ground state alternative (with $1$ being interpreted as a generalized ground state - the so-called Agmon ground state - in the critical case). It  is contained in \cite{PT07} for $p$-Laplace type operators on manifolds and in \cite{Fis23} for $p$-Laplacians on weighted graphs with slight variations (e.g. due to $p$-homogeneity the weighted Hardy inequalities are formulated in terms of a weighted $L^p$-norm instead of a weighted $L^1$-norm).  

\end{remark}

\begin{remark}[Comparison to the bilinear case] \label{remark:comparison to bilinear}
In principle we adapt the proofs given in \cite{FOT} to the nonlinear setting to establish our characterizations of criticality and subcriticality. There are however two key differences: 
\begin{enumerate}[(a)]
 \item If $\cE$ is a nonlinear subcritical Dirichlet form, the finiteness of $Gf$ holds for some particular measurable function $f \colon X \to (0,\infty)$, which is more or less explicitly given. In contrast, for bilinear Dirichlet forms subcriticality implies $Gf < \infty$ a.s.\@ for all $f \in L^1_+(\mu)$. 
 \item For bilinear Dirichlet forms criticality is equivalent to $\mu(\{0 < Gf < \infty\}) = 0$ for all $f \in L^1_+(\mu)$.
%
 \end{enumerate}
In summary, for bilinear Dirichlet forms both subcriticality and criticality have slightly different conditions characterizing them. The reason for this is the validity of Hopf’s maximal ergodic inequality, see \cite[Lemma~1.5.2]{FOT}, the linearity of the resolvent and $D(\cE)$ being dense in $L^2(\mu)$ (which is always assumed in the bilinear case). As of yet we do not have a variant of Hopf's maximal ergodic inequality for the nonlinear case. However, even if it holds in some form, the nonlinearity of the resolvent prevents the proofs given in \cite{FOT} to immediately carry over. 
\end{remark}

\section{Equilibrium potentials and a glimpse at potential theory} \label{section:potential theory}

In this section we briefly discuss how subcriticality can be used to show the existence of equilibrium potentials, which are important for establishing a potential theory for nonlinear order preserving forms via capacities. While we are not going to discuss all the  details, we do provide some basic properties.

Here we assume that $X$ is a Hausdorff topological space and that $\mu$ is a $\sigma$-finite Borel measure on $X$ (i.e. a locally finite measure on the Borel-$\sigma$-algebra of $X$). Moreover, $\cE \colon L^2(\mu) \to [0,\infty]$ is assumed to be a nonlinear order preserving form for which $\cE_e$ exists and $h \colon X \to [0,\infty)$ is an $\cE$-excessive function.   

For a measurable set $A \subseteq X$ let 
$$\mathcal L_A = \{f \in L^0(\mu) \mid f \geq h \text{ a.s.\@ on some open neighborhood of }A\} $$
and 
$${\rm cap}_h(A) = {\rm cap}_{\cE,h}(A) =  \inf \{\cE_e(f) \mid f \in \mathcal L_A\}.$$
It follows directly from the first Beurling-Deny criterion for $\cE$ that for all $A,B \subseteq X$ we have 
 $${\rm cap}_h(A \cap B) + {\rm cap}_h(A \cup B) \leq {\rm cap}_h(A) + {\rm cap}_h(B).$$
 Two typical situations where this construction is used are the following: 
 
 \begin{example}[0-Capacity for extended Dirichlet forms]
  Let $\cE$ be a nonlinear Dirichlet form. In this case, $1$ is an excessive function (not the only one, see below). The induced capacity ${\rm cap}_{\cE,1}$ is an extension of the capacity ${\rm Cap}_{(0)}$ introduced at the end of \cite[Section~2.1]{FOT} to the nonlinear case. 
 \end{example}

\begin{example}[1-Capacity for nonlinear order preserving forms] Let $\cE$ be a nonlinear order preserving form but instead of $\cE$ we consider $\cE_1 = \cE + \frac{1}{2}\av{\cdot}^2$ (cf. the definition of $\cE_w$ at the beginning of Subsection~\ref{subsection:weak hardy and poincare}). It is a nonlinear order preserving form for which $(M(\cE_1),\av{\cdot}_{L,\cE_1})$ is a Banach space. The completeness can be inferred from Theorem~\ref{theorem:two imply the third} and Remark~\ref{remark:two imply the third},  see also \cite[Theorem~3.7]{Cla21}. Hence, by Theorem~\ref{theorem:existence extended Dirichlet space} the extended form $(\cE_1)_e$ exists and $D((\cE_1)_e) = D(\cE)$. We saw in Remark~\ref{remark:excessive functions 1} that $h = G_1 \psi$ with $0 \leq \psi \in L^2(\mu)$ is an $\cE_1$-excessive function.  

  Hence, $\cE_1$ with $h$ satisfies the above assumptions and  
  $${\rm cap}_{\cE_1,h}(A)  =   \inf \{\cE_1(f) \mid f \in \mathcal L_A\}.$$
  This capacity is an extension of the capacity constructed in \cite[Section~III.2]{MR} to the nonlinear case. 
\end{example}

One important feature of the previous example is that the excessive function $h$ belongs to the domain of the functional, as this implies that the whole space has finite capacity. The existence of such excessive functions under relatively mild assumptions is discussed next. For certain uniqueness aspects we need the strict convexity of $\cE_e$, where  $\cE_e$ is called {\em strictly convex} if $\cE_e((f + g)/2) < (\cE_e(f) + \cE_e(g))/2$ for all $f,g \in D(\cE_e)$.

\begin{theorem}[Existence of excessive functions]\label{theorem:existence of excesive functions}
 Let $\cE \colon L^2(\mu) \to [0,\infty]$ be a nonlinear order preserving form. Assume that either $(M(\cE),\av{\cdot}_L)$ is a Banach space or that $\cE_e$ exists and $\cE$ is reflexive and subcritical. Then the following holds: 
 \begin{enumerate}[(a)]
  \item Let $g \in L^0(\mu)$ and let $A\subseteq X$ be measurable such that there exists $f \in D(\cE_e)$ with $f \geq g$ a.s.\@ on $A$. Then there exists an $\cE$-excessive function $0 \leq h \in D(\cE_e)$ with $h \geq g$ a.s.\@ on $A$ and 
  $$ \cE_e(h) =  \min \{\cE_e(f) \mid f \in L^0(\mu) \text{ with } f \geq g \text{ on }A\}.$$
   If  $\cE_e$ is {\em strictly convex}, then $h$ is unique.  Moreover, if $(M(\cE),\av{\cdot}_L)$ is a Banach space, then $\cE = \cE_e$ and $h \in D(\cE)$.
  \item If there exist $g \in D(\cE_e)$ with $g > 0$ a.s.\@,  then there exists an $\cE$-excessive function $h \colon X \to (0,\infty)$. 
 \end{enumerate}
\end{theorem}
\begin{proof}
(a): Let $g$ be given as in the statement. Our assumption implies 
 $$I = \inf \{\cE_e(f) \mid f \in L^0(\mu) \text{ with } f \geq g \text{ on }A\}< \infty.$$
By Proposition~\ref{proposition:characterization of excessive functions}~(c) (more precisely by the same argument as in its proof) any function attaining this infimum is $\cE$-excessive. Hence, we need to show the existence of a minimizer.

 Let $(h_n)$ be a sequence in $D(\cE_e)$ with $h_n \geq g$ a.s.\@ on $A$ such that $\cE_e(h_n) \searrow I$. Since $|h_n| \geq h_n \geq g$ on $A$  and $\cE_e(|h_n|) \leq \cE_e(h_n)$, we can assume $h_n \geq 0$.  Since $(h_n)$ is $\cE_e$-bounded, it is also $\av{\cdot}_{L_e}$-bounded.
 
 Case~1: $(M(\cE),\av{\cdot}_L)$ is a Banach space. In this case $\cE_e$ exists and $\cE = \cE_e$, see Theorem~\ref{theorem:existence extended Dirichlet space}.   Since $(M(\cE),\av{\cdot}_L)$ is a Banach space, as discussed in the proof of Example~\ref{example:completeness implies subcriticality I}, $(M(\cE),\av{\cdot}_L)$ continuously embeds into $L^2(\mu)$. Hence $(h_n)$ is bounded in $L^2(\mu)$ and by the Banach-Saks theorem we find a subsequence $(h_{n_k})$ such that the Césaro means  $f_N = (1/N) \sum_{k=1}^N h_{n_k}$ converge to some $h \in L^2(\mu)$. These means satisfy $0\leq f_N$ and $g \leq f_N$ on $A$ and hence $0\leq h$ and $g \leq h$ on $A$. The $L^2$-lower semicontinuity and the convexity of $\cE$ imply 
 $$\cE(h) \leq \liminf_{N\to \infty} \cE(f_N) \leq \liminf_{k\to \infty} \cE(h_{n_k}) = I.$$
 
 Case 2: $\cE$ is reflexive and subcritical: Using reflexivity, by Mazur's theorem we find a sequence $(f_n)$ such that $f_n \in {\rm conv}\{h_k \mid k \geq n\}$ and  $(f_n)$ is $\av{\cdot}_{L_e}$-Cauchy. These convex combinations satisfy $f_n \geq g$ on $A$  and the convexity of $\cE$ implies $\cE_e(f_n) \to I$. Using Theorem~\ref{theorem:characterization subcriticality} we choose $w \colon X \to (0,\infty)$ such that the embedding of $(M(\cE_e),\av{\cdot}_{L_e})$ into $L^1(w\mu)$ is continuous. Hence, $(f_n)$ is Cauchy in $L^1(w\mu)$ and therefore $f_n \to h$ in $L^1(w\mu)$ for some $h \in L^1(w\mu)$. The inequalities $h \geq g$ on $A$ and  $h \geq 0$ hold and the lower semicontinuity of $\cE_e$ implies 
 $$\cE_e(h) \leq \liminf_{n\to \infty} \cE_e(f_n) = I.$$
 %
 
 (b): This follows directly from (a).
\end{proof}


\begin{corollary}[Existence of equilibrium potentials] \label{corollary:existence of equilibrium potentials}
 Assume the standing assumptions of this section and that either $(M(\cE),\av{\cdot}_L)$ is a Banach space or that $\cE$ is reflexive and subcritical.  
 \begin{enumerate}[(a)]
  \item For each open $O \subseteq X$ with ${\rm cap}_h(O) < \infty$ there exists an $\cE$-excessive function $e_O \in D(\cE_e)$ with $0 \leq e_O \leq h$ and $e_O = h$ a.s.\@ on $O$ such that 
 $${\rm cap}_h(O) = \cE_e(e_O).$$
  If $\cE_e$ is strictly convex, then $e_O$ is unique, and if $(M(\cE),\av{\cdot}_L)$ is a Banach space, then $e_O \in D(\cE)$. 
  \item ${\rm cap}_h$ is a Choquet capacity. 
 \end{enumerate}
\end{corollary}

\begin{remark}
 The functions $e_O$ constructed in this corollary are called {\em equilibrium potential} for $O$. Note that the uniqueness of $e_O$ is not guaranteed without strict convexity. 
\end{remark}

 \begin{proof}
(a):  Using Theorem~\ref{theorem:existence of excesive functions} we find an $\cE$-excessive function $\tilde e_O \in D(\cE_e)$ with $\tilde e_O \geq h$ on $O$ and $\tilde e_O \geq 0$ such that 
 $$\cE_e(\tilde e_O) =  \min \{\cE_e(f) \mid f \in L^0(\mu) \text{ with } f \geq h \text{ on }O\}.$$
 Let $e_O = \tilde e_O \wedge h$ such that $e_O = h$ on $O$.  Since $h$ is $\cE$-excessive, we infer $\cE_e(e_O) \leq \cE_e(\tilde e_O)$ and by our characterization of excessive functions in Proposition~\ref{proposition:characterization of excessive functions} $e_O$ is also $\cE$-excessive.  If $(M(\cE),\av{\cdot}_L)$ is a Banach space, then $\tilde e_O \in D(\cE)$ follows from Theorem~\ref{theorem:existence of excesive functions}. Moreover, $\cE(e_O) \leq \cE(\tilde e_O)$ (and hence $e_O \in D(\cE)$) follows from Proposition~\ref{proposition:characterization of excessive functions}. 

(b): It follows directly from the definition of $\mathcal L_A$ that ${\rm cap}_h(A) = \inf \{ {\rm cap}_h(O) \mid O \text{ open with } A\subseteq O\}$. Hence, according to \cite[Theorem~A.1.2]{FOT}, it suffices to show the following for open subsets $O_n$, $n \in \N$, of $X$:

\begin{enumerate}
 \item The inclusion $O_1 \subseteq O_2$ implies ${\rm cap}_h(O_1) \leq {\rm cap}_h(O_2$). This follows directly from the definition. 
  \item ${\rm cap}_h(O_1 \cap O_2) + {\rm cap}_h(O_1 \cup O_2) \leq {\rm cap}_h(O_1) + {\rm cap}_h(O_2)$. As mentioned above, this is a direct consequence of the first Beurling-Deny criterion. 
  \item If $(O_n)$ is increasing and $O = \cup_n O_n$, then $\sup_n {\rm cap}_h(O_n) = {\rm cap}_h(O)$. This is a consequence of the existence of equilibrium potentials: Let $(e_n)$ in $D(\cE_e)$ with $\cE_e(e_n) = {\rm cap}_h(O_n)$. We let $f_n = \inf_{k \geq n} e_k$. The $L^0$-lower semicontinuity of $\cE_e$ (note that $\cE_e = \cE$  if $(M(\cE),\av{\cdot}_L)$ is a Banach space) and the $\cE$-excessivity of the $(e_n)$ yield
  $$\cE_e(f_n) \leq \liminf_{N\to \infty} \cE_e(\inf_{n \leq k \leq N} e_k) \leq \cE_e(e_n) = {\rm cap}_h(O_n).$$
  A similar computation shows that $f_n$ is $\cE$-excessive. Since $(O_n)$ is increasing, we also have $f_n \geq h$ on $O_n$ a.s.\@ Hence, the sequence of equilibrium potentials for $(O_n)$ can be chosen to be increasing. In particular, the limit $f = \lim_{n \to \infty} f_n$ exists a.s.\@ It satisfies $f \geq h$ on $O$ a.s.\@ From the lower semicontinuity of $\cE_e$ with respect to local convergence in measure we infer 
  $${\rm cap}_h(O) \leq \cE_e(f) \leq \liminf_{n \to \infty} \cE_e(f_n) \leq \liminf_{n \to \infty} {\rm cap}_h(O_n). $$
  Since ${\rm cap}_h(O_n)\leq {\rm cap}_h(O)$ by (1), we infer $\sup_n {\rm cap}_h(O_n) = {\rm cap}_h(O)$.
\end{enumerate}
\end{proof}

With the help of Proposition~\ref{proposition:characterization of excessive functions} we can characterize equilibrium potentials in terms of directional derivatives of $\cE_e$.

\begin{corollary} \label{corollary:alternative description equilibrium potential}
We use the assumptions of Corollary~\ref{corollary:existence of equilibrium potentials} and additionally assume that $\cE_e$ is strictly convex. For $O \subseteq X$ open with ${\rm cap}_h(O) < \infty$ and $f \in D(\cE_e)$ the following assertions are equivalent:  
\begin{enumerate}[(i)]
 \item $f = e_O$, where $e_O$ is the unique equilibrium potential of $O$. 
 \item $f \geq 0$,  $f = h$ a.s.\@ on $O$ and $d^+ \cE_e(f,\varphi) \geq 0$ for all $\varphi \in M(\cE_e)$ with $\varphi \geq 0$ a.s.\@ on $O$.
\end{enumerate}
%
%
\end{corollary}
\begin{proof}
(i) $\Rightarrow$ (ii):  Let $e_O$ be the equilibrium potential of $O$. Then $0 \leq e_O$ and $e_O = h$ on $O$.
In particular, we have
 \[ \cL_O = \{ g \in L^0(\mu) \mid g \geq e_O \text{ a.s.\@ on } O \} . \]
Since $\cE_e(e_O) = {\rm cap}_h(O)$, the definition of the directional derivative yields $d^+ \cE_e(e_O,\varphi) \geq 0$ for all $\varphi \in M(\cE_e)$ with $\varphi \geq 0$ a.s.\@ on $O$.
 
 (ii) $\Rightarrow$ (i):
 Arguing just as in the step (iii) $\Rightarrow$ (ii) in the proof of Proposition~\ref{proposition:characterization of excessive functions}, we obtain
  \[ \cE_e(f) = \min \{ \cE_e(g) \mid g \in L^0(\mu) \text{ with $g \geq f$ a.s.\@ on $O$} \} . \]
 Since $e_O = h = f$ a.s.\@ on $O$, this implies $\cE_e(f) \leq \cE_e(e_O) = {\rm cap}_h(O)$.
 Because $\cE_e$ is strictly convex and $f \in \cL_O$, this yields $f = e_O$.
 \end{proof}

 \begin{remark}
 \begin{enumerate}[(a)]
 
  \item  For general nonlinear Dirichlet forms a capacity ${\rm Cap}$ was introduced in \cite[Section~3.4]{Cla21} (see also \cite{Cla23}) by letting
 $${\rm Cap}(A) = \inf \{ \av{f}_{L,\cE_2} \mid f \geq 1 \text{ a.s.\@ on an open neighborhood of } A\},$$
where $\av{\cdot}_{L,\cE_2}$ is the Luxemburg norm of the nonlinear Dirichlet form $\cE_2 = \cE + \av{\cdot}^2$. Potential theory with respect to this capacity is developed there in quite some detail. In particular, the existence of equilibrium potentials is established. However, it is not shown that ${\rm Cap}$ is a Choquet capacity because the Luxemburg norm  $\av{\cdot}_{L,\cE_2}$ in general does not satisfy the first Beurling-Deny criterion. Hence, the inequality ${\rm Cap}(O_1 \cap O_2) + {\rm Cap}(O_1 \cup O_2) \leq {\rm Cap}(O_1) + {\rm Cap}(O_2)$, which is crucial for proving that ${\rm  Cap}$ is a Choquet capacity, need not hold.  
   \item  Capacities for $p$-homogeneous nonlinear Dirichlet forms induced by quasi-regular strongly local Dirichlet forms were recently studied in \cite{BBR24} and \cite{Kuw24}. In Subsection~\ref{subsection:quasi-regular} we discuss in more detail how our results extend the ones obtained there. For the case of variable exponent Sobolev spaces we refer to \cite[Chapter~10]{DHHR11}.
 \end{enumerate}
\end{remark}

  Of particular importance are sets of capacity zero. For certain nonlinear order preserving forms we characterize them with the help of nests. For $G \subseteq X$ and any subset $\cF \subseteq L^0(\mu)$ we let 
  $$\cF_G = \{f \in \cF \mid f  = 0 \text{ on an open neighborhood of } X \setminus G\}.$$
We call an increasing sequence of closed subsets $(F_n)$ of $X$ an {\em $\cE$-nest} if 
$$\bigcup_{n = 1}^\infty M(\cE_e)_{F_n}$$
is $\av{\cdot}_{L}$-dense in $M(\cE_e)$. 


%
%
%
%
%

A set $A \subseteq X$ is called {\em $\cE$-exceptional} if there exists an $\cE$-nest $(F_n)$ with $A \subseteq \bigcap_{n = 1}^\infty (X \setminus F_n)$.

\begin{theorem}\label{theorem:characterization exceptional sets}
Let $\cE$ be a nonlinear order preserving form for which $\cE_e$ exists.    Assume further that $\cE$ is reflexive and subcritical and that it satisfies the $\Delta_2$-condition. Further assume that there exists an $\cE$-excessive function $h \in D(\cE_e)$ with $h > 0$ a.s.\@ For $A \subseteq X$ the following assertions are equivalent: 
\begin{enumerate}[(i)]
 \item $A$ is $\cE$-exceptional. 
 \item For all $\cE$-excessive $\tilde h \in D(\cE_e)$ with $\tilde h > 0$ a.s.\@ we have ${\rm cap}_{\tilde h}(A) =  0$.
\end{enumerate}
If $Nh$ is $\cE$-excessive for each $N \in \N$, then these are equivalent to 
\begin{enumerate}[(i)] \setcounter{enumi}{2}
 \item ${\rm cap}_{h}(A) =  0$.
\end{enumerate}

\end{theorem}

\begin{remark}

\begin{enumerate}[(a)]
 \item The characterization of $\cE$-exceptional sets through nests is quite important for potential theory of Dirichlet forms. In the bilinear case it was first observed in \cite{AM91a,Am91b} and has since seen various extensions to more and more general situations, see e.g. \cite{MR,MOR95}. To the best of our knowledge our theorem is the first version in the nonlinear case.
 
 \item  In view of the examples given in the next section we do not see reflexivity and the $\Delta_2$-condition as strong assumptions. However, the existence of $h \in D(\cE_e)$ with $Nh$ being $\cE$-excessive for all $N \in \N$ is somewhat restrictive. It is satisfied if $\cE$ is $p$-homogeneous for some $p \geq 1$ (use Theorem~\ref{theorem:existence of excesive functions} and Remark~\ref{remark:excessive functions 2}) or if $\cE$ is a nonlinear Dirichlet form with $1 \in D(\cE_e)$. 
\end{enumerate}
\end{remark}

We prove this theorem through two lemmas. 

\begin{lemma}\label{lemma:alternative formula capacity}
 Assume the standing assumptions of this section.
 For all $A \subseteq X$ we have 
 $${\rm cap}_h(A)  = \inf\{\cE_e (h-g) \mid g \in L^0(\mu)_{X\setminus A}\}. $$
\end{lemma}
\begin{proof}
Let $I  = \inf\{\cE_e (h-g) \mid g \in L^0(\mu)_{X \setminus A}\}$ and note that $g \in L^0(\mu)_{X \setminus A}$ if and only if $g = 0$ on an open neighborhood of $A$.

If $g = 0$ on some open neighborhood of $A$, then $h-g \geq h$ on some open neighborhood of $A$. This implies $h - g \in \mathcal L_A$ so that ${\rm cap}_h(A) \leq I$. 

For the opposite inequality let $\varepsilon > 0$ and choose $f \in \mathcal L_A$ with $\cE_e(f) \leq {\rm cap}_h(A) + \varepsilon$. Then $h - f \wedge h = (h - f)_+ = 0$ on an open neighborhood of $A$. Since $\cE_e(f \wedge h) \leq \cE_e(f)$, we infer 
\begin{align*}
\varepsilon + {\rm cap}_h(A) &\geq \cE_e(f \wedge h) = \cE(h - (h - f\wedge h)) \geq I. \hfill \qedhere
\end{align*}
\end{proof}
\begin{remark}
 The proof of this lemma is basically the same as the proof of \cite[Lemma~2.72]{Schmi3}. There the result is stated for quadratic forms under the additional assumption $h \in D(\cE_e)$, which is equivalent to ${\rm cap}_h(X) < \infty$. But this finiteness is not necessary in the proof.  
\end{remark}


\begin{lemma}\label{lemma:characterization of nests}
Let $\cE$ and $h$ be as in Theorem~\ref{theorem:characterization exceptional sets}. Let $(F_n)$ be an increasing sequence of closed subsets of $X$. The following assertions are equivalent. 
 \begin{enumerate}[(i)]
  \item $(F_n)$ is an $\cE$-nest. 
  \item For all $\cE$-excessive $\tilde h \in D(\cE_e)$ with $\tilde h > 0$ we have $\inf_{n \geq 1} {\rm cap}_{\tilde h}(X \setminus F_n) = 0$. 
 \end{enumerate}
 If $Nh$ is $\cE$-excessive for each $N \in \N$, then these are equivalent to 
 \begin{enumerate}[(i)]\setcounter{enumi}{2}
  \item  $\inf_{n \geq 1} {\rm cap}_{h}(X \setminus F_n) = 0.$
 \end{enumerate}
 
\end{lemma}
\begin{proof}
(i) $\Rightarrow$ (ii): Since $(F_n)$ is an $\cE_e$-nest and $\tilde h \in D(\cE_e)$, there exist $h_n \in M(\cE_e)_{F_n}$ with $\inf_{n \geq 1}\lVert\tilde h - h_n\rVert_{L_e} = 0$. The (weak) $\Delta_2$-condition implies $\inf_{n \geq 1} \cE_e(\tilde h - h_n) = 0$, see Lemma~\ref{lemma:delta2 conditions}. From the alternative formula for the capacity of Lemma~\ref{lemma:alternative formula capacity} we infer
$$\inf_{n \geq 1}{\rm cap}_{\tilde h}(X \setminus F_n) \leq \inf_{n \geq 1} \cE_e(\tilde h - h_n) = 0.$$

We show (ii) $\Rightarrow$ (i) and (iii) $\Rightarrow$ (i) at the same time: The $\Delta_2$-condition implies $D(\cE_e) = M(\cE_e)$ and that $\cE_e$-convergence and $\av{\cdot}_{L_e}$-convergence agree. Let $\widetilde \cE_e$ be the lower semicontinuous relaxation of the restriction of $\cE_e$ to the set $\bigcup_{n=1}^\infty D(\cE_e)_{F_n}$ on $L^0(\mu)$ - it also satisfies the $\Delta_2$-condition.

Claim:  $D(\cE_e) \subseteq D(\widetilde \cE_e)$ implies that  $\bigcup_{n=1}^\infty D(\cE_e)_{F_n}$ is $\av{\cdot}_{L_e}$-dense in $D(\cE_e)$. 

Proof of the claim: Let $f \in D(\cE_e)$. Since $f \in D(\widetilde \cE_e)$, there exist $(f_n)$ in $\bigcup_{n=1}^\infty D(\cE_e)_{F_n}$ such that $f_n \tom f$ and $\widetilde \cE_e(f) = \lim_{n\to \infty} \cE_e(f_n)$. In particular, this implies that $(f_n)$ is $\cE_e$-bounded and the $\Delta_2$-condition implies that $(f_n)$ is even $\av{\cdot}_{L_e}$-bounded. With the help of Lemma~\ref{lemma:weak convergence} we infer $f_n \to f$ weakly in $(M(\cE_e),\av{\cdot}_{L_e})$. Since in semi-normed spaces weak and strong closures of subspaces coincide,  $f$ belongs to the closure of $\bigcup_{n=1}^\infty D(\cE_e)_{F_n}$, which proves the claim.

In order to show the inclusion $D(\cE_e) \subseteq D(\widetilde \cE_e)$ we need some preparations.

%

Since by the $\Delta_2$-condition $D(\cE_e)$ is a vector space, we have $Nh \in D(\cE_e)$ for $N \in \N$. Assuming (ii), we use Theorem~\ref{theorem:existence of excesive functions} to choose an $\cE$-excessive function $h_N \in D(\cE_e)$ such that $h_N \geq Nh$, and assuming (iii), we simply let $h_N = Nh$. Assuming (ii) we have $\inf_{n \geq 1} {\rm cap}_{h_N}(X \setminus F_n) = 0$. In case (iii) holds, we deduce $\inf_{n \geq 1} {\rm cap}_{h_N}(X \setminus F_n) = 0$ from the $\Delta_2$-condition, which implies the existence of $C > 0$ such that 
$${\rm cap}_{h_N}(X\setminus F_n) = \inf \{ \cE_e(Nf) \mid f \geq h \text{ on } X \setminus F_n \} \leq C {\rm cap}_{h}(X\setminus F_n) . $$

With  $\inf_{n \geq 1} {\rm cap}_{h_N}(X \setminus F_n) = 0$ at hand,  Lemma~\ref{lemma:alternative formula capacity} yields the existence of $(h_{n,N})$ with $h_{n,N} \in D(\cE_e)_{F_n}$  and $\cE_e(h_N - h_{n,N}) \to 0$, as $n \to \infty$. From the $\Delta_2$-condition we infer $\av{h_N - h_{n,N}}_{L_e} \to 0$.  Since $\cE$ is subcritical,  Theorem~\ref{theorem:characterization subcriticality} (ii) implies  $h_{n,N} \tom h_N$.

 As discussed above, it suffices to show $D(\cE_e) \subseteq D(\widetilde \cE_e)$. Let $f \in D(\cE_e)$ with $f \geq 0$ be given. Then $h_{n,N} \in D(\cE_e)_{F_n}$ implies $f \wedge h_{n,N} \in D(\cE_e)_{F_n}$ and we have $f \wedge h_{n,N} \tom f \wedge h_N$. The lower semicontinuity of $\widetilde \cE_e$ and $\widetilde \cE_e = \cE_e$ on $\bigcup_{n=1}^\infty D(\cE_e)_{F_n}$ yield
\begin{align*}
 \widetilde \cE_e(f \wedge h_N)  &\leq \liminf_{n \to \infty}\widetilde\cE_e(f \wedge h_{n,N}) = \liminf_{n \to \infty}\cE_e(f \wedge h_{n,N}) \\
 &\leq \cE_e(f) +  \liminf_{n \to \infty} \cE_e(h_{n,N}),
\end{align*}
where we use the first Beurling-Deny criterion for the last inequality. Since $(h_{n,N})$ is $\av{\cdot}_{L_e}$-bounded, the $\Delta_2$-condition implies that it is $\cE_e$-bounded, and we arrive at $f \wedge h_N \in D(\widetilde \cE_e)$. The inequality $h_N \geq N h$ and $h > 0$ a.s.\@ yield $f \wedge h_N \tom f$, as $N \to \infty$. Hence, using the fact that $\widetilde \cE_e = \cE_e$ on $D(\widetilde \cE_e)$ and the $\cE$-excessivity of $h_N$, we obtain 
$$\widetilde \cE_e(f) \leq \liminf_{N \to \infty}\widetilde \cE_e(f \wedge h_N) = \liminf_{N \to \infty}\cE_e(f \wedge h_N) \leq \cE_e(f), $$
showing $f \in D(\widetilde \cE_e)$. Now let $f \in D(\cE_e)$ arbitrary. Since $\cE_e$ satisfies the first Beurling-Deny criterion, we have $f_+,f_- \in D(\cE_e)$. What we already proved yields $f_+,f_- \in D(\widetilde \cE_e)$. Since the latter set is a vector space (use that $\widetilde \cE_e$ satisfies the $\Delta_2$-condition), we infer $f \in D(\widetilde \cE_e)$. 
\end{proof}

\begin{proof}[Proof of Theorem~\ref{theorem:characterization exceptional sets}]
 (i) $\Rightarrow$ (ii): The $\cE$-exceptionality of $A$ implies the existence of an $\cE$-nest $(F_n)$ such that $A \subseteq \bigcap_{n=1}^\infty (X \setminus F_n)$. Using Lemma~\ref{lemma:characterization of nests}, for each $\cE$-excessive function $\tilde h \in D(\cE_e)$ with $\tilde h > 0$ a.s.\@ we infer
 $${\rm cap}_{\tilde h}(A) \leq {\rm cap}_{\tilde h}(X \setminus F_n) \to 0, \quad n \to \infty.  $$

 (ii) $\Rightarrow$ (i): Let $\tilde h \in D(\cE_e)$ with $\tilde h > 0$ a.s.\@ It follows from the definition of ${\rm cap}_{\tilde h}$ that ${\rm cap}_{\tilde h}(A) = 0$ yields the existence of open sets $O_n \supseteq A$ with $\inf_{n\geq 1}{\rm cap}_{\tilde h} (O_n) = 0$. Without loss of generality we can assume $O_n \supseteq O_{n+1}$. With this at hand, Lemma~\ref{lemma:characterization of nests} yields that $F_n = X \setminus O_n$ is an $\cE$-nest with $A \subseteq \bigcap_{n=1}^\infty X \setminus F_n$. 
 
 With the same arguments the equivalence with (iii) can be established under the stronger condition that $Nh$ is $\cE$-excessive for all $N \in \N$. 
\end{proof}

\section{Examples}\label{section:examples}

In this section we provide two examples concerning variable exponent Dirichlet spaces that show that our results extend previous works. 

\subsection{Variable exponent Dirichlet forms on Riemannian manifolds}
 
 Let $(M,g,\mu)$ be a weighted Riemannian manifold (i.e. $\mu = e^{-\Phi} {\rm vol}$, where $\Phi \colon M \to \R$ is smooth and ${\rm vol}$ is the Riemannian volume measure). By $d$ we denote the path metric on $M$ induced by $g$ and by ${\rm Lip}_c(M)$ we denote the  Lipschitz functions of compact support with respect to $d$. In this text we will use some properties of variable exponent Lebesgue- and Sobolev spaces. We refer to \cite{DHHR11} for background on them.

 Let $p \colon M \to [1,\infty)$ be measurable with $p_- = \essinf p$ and $p_+ = \esssup p$.   We define the functional $\cE_{p} \colon L^2(\mu) \to [0,\infty]$ through
 $$\cE_{p} (f) =\begin{cases}
                                               \int_M \frac{1}{p}|\nabla f|^{p} d  \mu &\text{if } f \in W^{1,1}_{\rm loc}(M)\\
                                             \infty &\text{else}
                                             \end{cases}.$$
 Moreover, we let $\cE_p^0$ be the largest $L^2(\mu)$-lower semicontinuous minorant of the restriction of $\cE_{p}$ to $D(\cE_{p})_c$, where the latter denotes the set of all $f \in D(\cE_{p})$ vanishing outside a compact set. 
 Here, restriction is meant in the sense that $\cE_{p}$ is set to $\infty$ outside of $D(\cE_{p})_c$.
 
 \begin{proposition}\label{proposition:basic properties variable exponent sobolev}
  Assume  $1 < p_-  \leq p_+ < \infty$.
  
  \begin{enumerate}[(a)]
   \item $\cE_p^0$ and $\cE_{p}$ are reflexive nonlinear Dirichlet forms satisfying the $\Delta_2$-condition.
   \item $D(\cE_p^0) = \overline{D(\cE_{p})_c}$, where the closure is taken in $D(\cE_{p}) = M(\cE_{p})$ with respect to the norm $\av{\cdot}_{L,\cE_{p}}  + \av{\cdot}_2$.
   \item  $(\cE_p^0)_e$ is a restriction of $(\cE_{p})_e$ and $D((\cE_p^0)_e)$ is the closure of $D(\cE_p^0)$ in $D((\cE_{p})_e) = M((\cE_{p})_e)$ with respect to $\av{\cdot}_{L_e,\cE_{p}} + q$ (here $q$ is an $F$-norm inducing the topology of local convergence in measure).
   \item For all $f \in D((\cE_{p})_e)$ we have $f \in W^{1,1}_{\rm loc}(M)$ and
  $$(\cE_{p})_e(f) = \int_M \frac{1}{p}|\nabla f|^{p} d  \mu.$$
  \end{enumerate}
 \end{proposition}
\begin{proof}
%
(a): Lower semicontinuity: For $\cE _{p}$ this can be established as in \cite[Example 3.58]{Cla21}. $\cE_p^0$ is lower semicontinuous by definition.  

$\cE _{p}$ is compatible with normal contractions: This can also be inferred as in \cite{Cla21}. We give a simplified proof based on our characterization of nonlinear Dirichlet forms.  Since $\cE _{p}$ is lower semicontinuous, it suffices to show compatibility with continuously differentiable normal contractions. For such a normal contraction $C$ and $f \in W^{1,1}_{\rm loc}(M)$ we have $\nabla (Cf)= C'(f) \nabla f$ with $|C'(f)| \leq 1$.  Lemma~\ref{lemma:basice inequality} yields 
 \begin{align*}
|\nabla (f + Cg)|^p + |\nabla(f - Cg)|^p &= |\nabla f + C'(g)\nabla g|^p + |\nabla f - C'(g)\nabla g|^p\\
&\leq |\nabla f + \nabla g|^p + |\nabla f - \nabla g|^p.
 \end{align*}
 Integrating this inequality shows that $C$ operates on $\cE_{p}$.

 $\cE_p^0$ is compatible with normal contractions: This can be proven using lower semicontinuity and the contraction property for $\cE _{p}$ on $D(\cE_{p})_c$. The arguments are the same as in the proof of Proposition~\ref{proposition:compatibility with contractions extended forms}~(c).  
  
$\Delta_2$-condition: This is a straightforward consequence of $p_+ < \infty$.  
Concerning $\cE_p^0$, see Lemma~\ref{lemma:extended luxemburg norm}~(b).

Reflexivity: The variable exponent Lebesgue spaces $(L^{p(\cdot)}(\mu),\av{\cdot}_{p(\cdot)})$ are uniformly convex under the assumption $1<p_- \leq p_+ < \infty$, see \cite[Theorem~3.4.9]{DHHR11}. Using $p_+ < \infty$, the seminorm  $\av{\cdot}_{L,\cE_{p}}$ is equivalent to the seminorm
$$f \mapsto \av{|\nabla f|}_{p(\cdot)}$$
(in the definition of the variable exponent spaces one uses the semimodular $f \mapsto \int_M |f|^p d\mu$ and does not divide by $p$ in the integral).  

This shows that the Luxemburg seminorm of $\cE_{p}$ is equivalent to a uniformly convex seminorm, which in turn yields reflexivity. Since the Luxemburg seminorm of $\cE_p^0$  is a restriction of the Luxemburg seminorm of $\cE_{p}$  to a smaller domain,  $\cE_p^0$ is also reflexive.

(b): $D(\cE_p^0) = \overline{D(\cE_{p})_c}$: The definitions and the $\Delta_2$-condition yield that the effective domain of $\cE_p^0$, the largest lower semicontinuous minorant of the restriction of $\cE_{p}$ to $D(\cE_{p})_c$, must contain the $\av{\cdot}_{L,\cE_{p}}  + \av{\cdot}_2$-closure of $D(\cE_{p})_c$ in $D(\cE_{p})$. Conversely, using reflexivity, it follows from Theorem~\ref{theorem:lower semicontinuity}~(iv) that the restriction of $\cE_{p}$ to this closure, which has this closure as its effective domain, is lower semicontinuous. In particular, it is a minorant of $\cE_p^0$.
 %
%


(c): Our description of the extended Dirichlet space with the help of approximating sequences in Proposition~\ref{proposition:properties of extended forms} implies that $(\cE_{p})_e$ is an extension of $(\cE_p^0)_e$ and that $D((\cE_p^0)_e)$ is the closure of $D(\cE_p^0)$ in $D((\cE_{p}))_e$ with respect to $\av{\cdot}_{L_e,\cE_{p}} + q$.

(d): Let $f \in D((\cE_{p})_e)$ and let $(f_n)$ in $D(\cE_{p})$ be an approximating sequence for $f$ (i.e.\@ $(f_n)$ is $\av{\cdot}_{L,\cE_{p}}$-Cauchy and $f_n \tom f$). According to Proposition~\ref{proposition:properties of extended forms} we have 
 \[ (\cE_{p})_e(f) = \lim_{n \to \infty}\cE_{p}(f_n) \]
and we only have to show that this limit equals the desired integral.

Since $\av{\cdot}_{L,\cE_{p}}$ is equivalent to the seminorm $f \mapsto \av{|\nabla f|}_{p(\cdot)}$, the properties of $(f_n)$ yield that $(\nabla f_n)$ is Cauchy in $\vec L^{p(\cdot)}(M,\mu)$ and hence converges to some limit $X \in \vec L^{p(\cdot)}(M,\mu)$. Using a local Poincaré inequality for variable exponent spaces, e.g. a slightly modified version of \cite[Lemma 8.2.14]{DHHR11}, we obtain (with usual localization techniques) for each compact $K \subseteq M$ a constant $C_K > 0$ such that 
$$\int_K |g - \as{g}_K| d\mu  \leq C_K \av{|\nabla g|}_{p(\cdot)}$$
with $\as{g}_K = \mu(K)^{-1} \int_K g d\mu$, as long as $\nabla g \in L^{p(\cdot)}(M,\mu)$. Hence, $(f_n - \as{f_n}_K)$ is Cauchy in $L^1(K,\mu)$ and converges to $\tilde f_K \in L^1(K,\mu)$. Since also $f_n \tom f$, we infer that $\as{f_n}_K$ converges to some constant $C \in \R$ and $\tilde f_K + C = f$ on $K$. Since pointwise convergence of constant functions implies convergence in $L^1(K,\mu)$ (use $\mu(K) < \infty$), we infer $f_n = f_n - \as{f_n}_K + \as{f_n}_K \to \tilde f_K + C = f$ in  $L^1(K,\mu)$. Altogether, this implies $f \in L^1_{\rm loc}(M,\mu)$ and $f_n \to f$ in $L^1_{\rm loc}(M,\mu)$. We obtain $\nabla f_n \to \nabla f$ in $\mathcal D'(M)$ (in the usual weak-*-sense) and, using Hölder's inequality for variable exponent spaces,  we also have $\nabla f_n  \to X$ in $\vec L^1_{\rm loc}(M,\mu)$. Combining both results yields $\nabla f = X \in L^{p(\cdot)}(M,\mu)$  and 
\begin{align*}
 \int_M \frac{1}{p}|\nabla f|^{p} d\mu &= \lim_{n \to \infty} \int_M \frac{1}{p}|\nabla f_n|^{p} d\mu =  \lim_{n \to \infty} \cE _{p}(f_n).
\end{align*}
For the first equality we used that by the $\Delta_2$-condition $\av{|\nabla f_n - \nabla f|}_{p(\cdot)} \to 0$ implies convergence of the integrals (see Lemma~\ref{lemma:Delta_2 and Luxemburg seminorm convergence yield convergence of gsm}). 
 \end{proof}

We provide some basic criteria for criticality, which extend well-known results from the bilinear case.

\begin{proposition} 
Let $1 < p_- \leq p_+ < \infty$. 
\begin{enumerate}[(a)]
 \item If $\mu(M) < \infty$, then $\cE_{p}$ is critical. 
 \item Assume that $M$ is complete and for $o \in M$ and $r \geq 0$ let $B_r = \{x \in M \mid d(x,o) \leq r\}$. If 
 $\liminf_{r \to \infty} \mu(B_r)r^{-p_-} < \infty,  $
 then $\cE_p^0$ is critical. 
\end{enumerate}
\end{proposition}
\begin{proof}
 (a): If $\mu(M) < \infty$, then $1 \in L^2(\mu)$ and $\cE_{p}(1) = 0$. Theorem~\ref{theorem:characterization criticality} yields the criticality of $\cE _{p}$.
 
 (b): We choose an increasing sequence $R_k \to \infty$ with 
 $$\lim_{k \to \infty}  \mu(B_{R_k}){R_k}^{-p_-} = \liminf_{r \to \infty} \mu(B_r)r^{-p_-} < \infty.$$
 We consider the function $f_{k} = (1 - d(\cdot,o)/R_k)_+$.  It is $(R_k)^{-1}$-Lipschitz and vanishes on $M \setminus B_{R_k}$. This implies $\nabla f_k = 0$ on $M \setminus B_{R_k}$ and    $|\nabla f_k|^p \leq (R_k)^{-p_-}$, where we use Rademacher's theorem for the second inequality.  Since $M$ is complete, closed distance balls are compact and we obtain $f_k \in D(\cE _{p})_c \subseteq D(\cE_p^0)$. Moreover, these estimates show 
 $$\cE_p^0(f_k) = \int_M \frac{1}{p} |\nabla f_k|^{p} d\mu\leq \frac{1}{p_-}  \mu(B_{R_k}) (R_k)^{-p_-}.$$
 We obtain that $(f_k)$ is $\cE_{p,0}$-bounded and we also have $f_k \tom 1$. Using lower semicontinuity,  we obtain $1 \in D((\cE_p^0)_e)$.  With this at hand $(\cE_p^0)_e(1) = 0$ follows from Proposition~\ref{proposition:basic properties variable exponent sobolev}, where we determined the action of $(\cE_p^0)_e$ on $D((\cE_p^0)_e)$.
\end{proof}

\begin{remark}
If $p = 2$, then the criticality of $\cE_{2}^0$ is known as recurrence or parabolicity of the weighted manifold $(M,g,\mu)$. It is well known that this holds if $M$ is complete and $\liminf_{r \to \infty} \mu(B_r)r^{-2} < \infty$, see e.g. \cite{Gri1} (indeed in this case much more sophisticated tests are available). Hence, the previous result can be seen as a generalization to the nonlinear case. 
\end{remark}

Next, we consider the functionals
$$\cE_{p,1} \colon L^2(\mu) \to [0,\infty], \quad \cE_{p,1}(f) = \cE_{p}(f) + \int_X \frac{1}{p} |f|^{p} d\mu.  $$
and 
$$\cE^0_{p,1} \colon L^2(\mu) \to [0,\infty], \quad \cE^0_{p,1}(f) = \cE^0_{p}(f) + \int_X \frac{1}{p} |f|^{p} d\mu.$$

The assertions of Proposition~\ref{proposition:basic properties variable exponent sobolev} also hold for $\cE^{0}_{p,1}$ and $\cE_{p,1}$ (with almost the same proofs) if one simply replaces $\cE_{p}$ by $\cE_{p,1}$ and $\cE_p^0$ by $\cE_{p,1}^0$. The only difference is the action of the extended form. In this case, for $f \in D((\cE_{p,1})_e)$ we have 
$$(\cE_{p,1})_e(f) = \int_M \frac{1}{p}|\nabla f|^{p} d  \mu +  \int_M \frac{1}{p}|f|^{p} d  \mu.$$
We refrain from spelling out all details again but note the following result concerning the domain of the extended Dirichlet form of $\cE_{p,1}$.

\begin{proposition}
Assume $1 <  p_- \leq p_+ < \infty$. Then  
$$D((\cE_{p,1})_e) = \{f \in L^{p(\cdot)}(M,\mu) \mid \nabla f \in L^{p(\cdot)}(M,\mu)\} = W^{1,p(\cdot)}(M,\mu)$$
and 
$$(\cE_{p,1})_e(f) =  \int_M \frac{1}{p}|\nabla f|^{p} d  \mu +  \int_M \frac{1}{p}|f|^{p} d  \mu.$$
In particular, the Luxemburg norm of  $(\cE_{p,1})_e$ is equivalent to the $W^{1,p(\cdot)}(M,\mu)$-norm. 

Moreover, $(\cE^{0}_{p,1})_e$ is the restriction of $(\cE_{p,1})_e$ to the closure of $D(\cE_{p})_c$ with respect to the $W^{1,p(\cdot)}(M,\mu)$-norm.  
\end{proposition}
\begin{proof}
 It suffices to show $D((\cE_{p,1})_e) \supseteq  W^{1,p(\cdot)}(M)$, the rest was sketched above. Let $f \in W^{1,p(\cdot)}(M,\mu)$.  For $r > 0$ let $C_r(x) = (x \wedge r) \vee (-r)$ and consider the functions $f_{R,\varepsilon} = C_R(f - C_\varepsilon(f))$, $0 < \varepsilon < R$. Then $f_{R,\varepsilon} \in L^1(M,\mu) \cap L^\infty(M,\mu) \subseteq L^2(M,\mu)$.

 For $g,h \in W^{1,1}_{\rm loc}(M)$ we have $f \wedge g,f \vee g \in W^{1,1}_{\rm loc}(M)$ and
 $$\nabla (g \wedge h) = 1_{\{g <h\}} \nabla g + 1_{\{h \leq g\}} \nabla h \text{ and } \nabla (g \vee h) = 1_{\{g > h\}} \nabla g + 1_{\{h \geq g\}} \nabla h.$$
 This implies $f_{R,\varepsilon} \in W^{1,1}_{\rm loc}(M)$ and $|\nabla f_{R,\varepsilon}| \leq |\nabla f|$ as well as $|f_{R,\varepsilon}| \leq |f|$. Since $f_{R,\varepsilon} \in L^2(M,\mu)$ and $f \in W^{1,p(\cdot)}(M,\mu)$, this implies $f_{R,\varepsilon} \in D(\cE_{p,1})$. The convergence $f_{R,\varepsilon} \tom f$, as $\varepsilon \to 0+, R \to \infty$, and the lower semicontinuity of $(\cE_{p,1})_e$ yield 
 \begin{align*}
  (\cE_{p,1})_e(f) &\leq \liminf_{R \to \infty,\varepsilon \to 0+}  (\cE_{p,1})_e(f_{R,\varepsilon}).
  \end{align*}
 Since $(\cE_{p,1})_e =  \cE_{p,1}$ on $D(\cE_{p,1})$ and $f_{R,\varepsilon} \in D(\cE_{p,1})$, we infer
 \begin{align*}
  (\cE_{p,1})_e(f_{R,\varepsilon}) &= \int_M \frac{1}{p}|\nabla f_{R,\varepsilon}|^{p} d  \mu +  \int_M \frac{1}{p}|f_{R,\varepsilon}|^{p} d  \mu\\
  &\leq \int_M \frac{1}{p}|\nabla f|^{p} d  \mu +  \int_M \frac{1}{p}|f|^{p} d  \mu < \infty.
 \end{align*}
Put together, this shows $f \in  D((\cE_{p,1})_e)$.  
\end{proof}

\begin{remark}
This example shows that the capacities we introduce in Section~\ref{section:potential theory} with the help of extended Dirichlet spaces include the capacities often used in the context of $p$-energies or variable exponent Sobolev spaces, cf. \cite[Chapter~10]{DHHR11}.  
\end{remark}

\subsection{$p$-energies induced by strongly local regular Dirichlet forms} \label{subsection:quasi-regular}

In this subsection we use (almost) the same setting as in \cite{BBR24}. We refer to it (in particular to its appendix) for precise definitions and proofs (or references to proofs) for the claimed properties. Let $\cE$ be a strongly local regular bilinear Dirichlet form on $L^2(\mu)$ admitting a (positive definite symmetric bilinear) carré du champ operator $\Gamma \colon D(\cE) \times D(\cE) \to L^1(\mu)$ such that 
$$\cE(f) = \frac{1}{2} \int_X \Gamma(f) d\mu,$$
where we use the notation $\Gamma(f) = \Gamma(f,f)$. 
Important for us is that if $C \colon \R \to \R$ is a Lipschitz function, then for $f \in D(\cE)$ we have $C(f) \in D(\cE)$ and 
$$\Gamma(C(f),g) = C'(f) \Gamma(f,g) \text{ for all } g \in D(\cE).$$
Moreover, the Cauchy-Schwarz inequality $|\Gamma(f,g)| \leq \Gamma(f)^{1/2}   \Gamma(g)^{1/2}$ holds for $f,g \in D(\cE)$.

For a measurable  $p \colon X \to [1,\infty)$ and $\alpha \geq 0$ we define $\cE^{(0)}_{p,\alpha} \colon L^2(\mu) \to [0,\infty]$ by 
$$\cE^{(0)}_{p,\alpha}(f) = \begin{cases}
                                                                   \int_X  \frac{1}{p} \Gamma(f)^\frac{p}{2}d\mu + \alpha \int_X \frac{1}{p}|f|^p d\mu  &\text{if } f \in D(\cE)\\
                                                                    \infty &\text{else}
                                                                  \end{cases}.
$$
\begin{lemma}
$\cE^{(0)}_{p,\alpha}$ satisfies the second Beurling-Deny criterion.  If $1 < p_- \leq p_+ < \infty$, it is reflexive and satisfies the $\Delta_2$-condition.  
\end{lemma}
\begin{proof}
Second Beurling-Deny criterion:  Let $C \colon \R \to \R$ be a normal contraction. Using the chain rule, the bilinearity and the symmetry of $\Gamma$, for $f,g \in D(\cE)$ we obtain  
 $$\Gamma(f + C g)^\frac{p}{2}  = (\Gamma(f) + 2 C'(g) \Gamma(f,g) + C'(g)^2 \Gamma(g))^\frac{p}{2} $$
 and
 $$ \Gamma(f - C g)^\frac{p}{2} = (\Gamma(f) - 2 C'(g) \Gamma(f,g) + C'(g)^2 \Gamma(g))^\frac{p}{2}.$$
 Using these identities, $|C'(g)| \leq 1$ and $|\Gamma(f,g)| \leq \Gamma(f)^{1/2}   \Gamma(g)^{1/2}$,  Corollary~\ref{corllary:basic inequality} yields
 $$\Gamma(f + C g)^\frac{p}{2} + \Gamma(f - C g)^\frac{p}{2} \leq \Gamma(f +  g)^\frac{p}{2} + \Gamma(f -  g)^\frac{p}{2}. $$
 Moreover, Lemma~\ref{lemma:basice inequality} shows 
 $$|f + Cg|^p + |f - Cg|^p \leq |f + g|^p + |f - g|^p.$$
 Integrating these inequalities yields the claim. 
 
 Reflexivity:  As in the proof of Proposition~\ref{proposition:basic properties variable exponent sobolev} we obtain that the modular space $(M(\cE^{(0)}_{p,\alpha}),\av{\cdot}_{L})$ has a bi-Lipschitz embedding into a uniformly convex variable exponent Lebesgue space (more precisely $L^{p(\cdot)}(\mu)$ if $\alpha = 0$ and $L^{p(\cdot)}(\mu) \times L^{p(\cdot)}(\mu)$ if $\alpha > 0$). This implies the reflexivity of $\cE^{(0)}_{p,\alpha}$. 
 
  $\Delta_2$-condition: This directly follows from $p_+ < \infty$.
\end{proof}

 If $\cE^{(0)}_{p,\alpha}$ possesses a lower semicontinuous extension on $L^2(\mu)$, we denote its lower semicontinuous relaxation (the largest lower semicontinuous minorant) by $\cE_{p,\alpha}$. If $1 < p_- \leq p_+ < \infty$ (and hence $\cE^{(0)}_{p,\alpha}$ is reflexive and satisfies the $\Delta_2$-condition), then this is the case if and only if for any $\av{\cdot}_{L,\cE^{(0)}_{p,\alpha}}$-Cauchy sequence $(f_n)$ in $M(\cE^{(0)}_{p,\alpha})$ with $f_n \to  0$ in $L^2(\mu)$ we have $\av{f_n}_{L,\cE^{(0)}_{p,\alpha}} \to 0$, see Theorem~\ref{theorem:description extended space}.

 %
%
\begin{proposition}
 If $\cE_{p,\alpha}$ exists, it is a nonlinear Dirichlet form. If $1 < p_- \leq p_+ < \infty$, then $\cE_{p,\alpha}$ is reflexive and satisfies the $\Delta_2$-condition. 
\end{proposition}
\begin{proof}
 Reflexivity and $\Delta_2$-condition: This follows from the corresponding results for $\cE^{(0)}_{p,\alpha}$, see Lemma~\ref{lemma:extended luxemburg norm}.
 
 Beurling-Deny criterion: This can be proven as in Proposition~\ref{proposition:compatibility with contractions extended forms}, which shows that the second Beurling-Deny criterion also holds for extended forms. 
\end{proof}
\begin{proposition}
 If $2 \leq p_- \leq p_+ < \infty$ and $\mu(X) < \infty$, then $\cE^{(0)}_{p,\alpha}$ is lower semicontinuous  (in other words $\cE^{(0)}_{p,\alpha} = \cE_{p,\alpha}$). Moreover, the extended Dirichlet space satisfies $D((\cE_{p,\alpha})_e) \cap L^\infty(\mu) = D(\cE_{p,\alpha}) \cap L^\infty(\mu)$ and if $\alpha > 0$, then  $D((\cE_{p,\alpha})_e) = D(\cE_{p,\alpha})$.
\end{proposition}
\begin{proof}
 
 Lower semicontinuity: According to Theorem~\ref{theorem:lower semicontinuity}~(iv), it suffices to show that the modular space $(M(\cE^{(0)}_{p,\alpha}),\av{\cdot}_{L} + \av{\cdot}_2)$ is complete. This however can be proven with the same arguments as in the proof of \cite[Proposition~3.2]{BBR24} noting that $p_- \geq 2$ and $\mu(X) < \infty$ imply the continuity of the embedding $L^{p(\cdot)}(\mu) \to L^2(\mu)$.

 Extended Dirichlet spaces: By Proposition~\ref{proposition:properties of extended forms} we have $D((\cE_{p,\alpha})_e) \cap L^2(\mu) = D(\cE_{p,\alpha})$. Since $L^\infty(\mu) \subseteq L^2(\mu)$, this implies the first statement. For the second statement we note that if $\alpha > 0$, then $(M(\cE_{p,\alpha}),\av{\cdot}_{L})$ continuously embeds into $L^{p(\cdot)}(\mu)$, which in turn continuously embeds into $L^2(\mu)$ (use $p_- \geq 2$ and $\mu(X) < \infty$). Using $L^0$-lower semicontinuity of the $L^2$-norm, this embedding extends to $(M((\cE_{p,\alpha})_e),\av{\cdot}_{L_e})$ and we arrive at $D((\cE_{p,\alpha})_e) \subseteq L^2(\mu)$. 
\end{proof}

\begin{remark}
 \begin{enumerate}[(a)]
  \item If $\mu(X) = \infty$, the functional $\cE^{(0)}_{p,\alpha}$ need not be lower semicontinuous, its domain is too small.  By extending $\Gamma$ to the local form domain $D(\cE)_{\rm loc}$ (or rather a slightly larger space) and by using appropriate localizations (e.g. by employing the forms $\cE_\varphi$ constructed in \cite{Schmi2}) one can show the existence of $\cE_{p,\alpha}$ as long as  $2 \leq p_- \leq p_+ < \infty$. The regularity of $\cE$ and the existence of the carré du champ operator can also be omitted - quasiregularity and strict locality are sufficient (in this case $\cE_{p,\alpha}^{(0)}$ has to be defined with respect to an energy dominant measure). For the case of constant $p$ the details are discussed in \cite[Proposition~1.1]{Kuw24}, whose proof uses some probabilistic techniques. Since such results are somewhat technical, we do not discuss details here. In contrast, replacing the assumption $p_- \geq 2$ by $p_- > 1$ seems to require new ideas. 
  
  \item In this remark we always assume that $p \geq 2$ is constant and $\mu(X) < \infty$. Under these assumptions the previous proposition shows that $\cE_{p,0}$  (up to constants) equals the functional $\cE^p$ considered in \cite{BBR24,Kuw24} (even though the underlying space is $L^2(\mu)$ not $L^p(\mu)$, which by $\mu(X)< \infty$ does not make a big difference). 
  
  In \cite{BBR24} the authors construct a Choquet capacity in terms of $\cE_{p,0}$ with respect to the excessive function $h = 1$ under a coercivity assumption, see \cite[Proposition~3.10 and Theorem~3.11]{BBR24}. In our terminology this coercivity implies subcriticality (indeed subcriticality of $\cE_{p,0}$ seems to be weaker than coercivity). Since by the previous proposition  $\cE_{p,0} = (\cE_{p,0})_e$ on $L^\infty(\mu)$ and all equilibrium potentials are bounded by $1$, the capacity constructed there and our capacity coincide.
  
  In \cite{Kuw24} a capacity is constructed in terms of $\cE_{p,1}$ with respect to the excessive function $h = 1$. In this case $(M(\cE_{p,1}),\av{\cdot}_L)$ is a Banach space and $\cE_{p,1} = (\cE_{p,1})_e$, see the previous proposition. Hence, our capacity for $\cE_{p,1}$ and $h = 1$ coincides with the one constructed in \cite{Kuw24} (under the  assumptions $p \geq 2$ and $\mu(X) < \infty$).

  In both cases, with Corollary~\ref{corollary:existence of equilibrium potentials} and Theorem~\ref{theorem:characterization exceptional sets} we recover and extend results on potential theory  obtained in  \cite{BBR24,Kuw24}.  Note that  for $f,g \in D(\cE_{p,\alpha})$ we have
  $$d^+ \cE_{p,\alpha}(f,g) = \int_X \Gamma(f)^{\frac{p-2}{2}} \Gamma(f,g) + \alpha\int_X |f|^{\frac{p-2}{2}}  fg d\mu.$$
  Therefore, our characterization of the equilibrium potential in Corollary~\ref{corollary:alternative description equilibrium potential} yields the same one as given in \cite[Theorem~1.12~$(3)$]{Kuw24}.
  
  We also note that the analysis of \cite{BBR24} is based on the study of $d^+ \cE_{p,0}$ rather than $\cE_{p,0}$. Restricted to $D(\cE_{p,0}) \times D(\cE_{p,0})$ the lower derivative $d^+ \cE_{p,0}$ is a nonlinear Dirichlet form in the sense of van Beusekom \cite{Beu94}. This notion of nonlinear Dirichlet form does not coincide with the one of Cipriani and Grillo \cite{CG03}. In spirit, as in the given example, the nonlinear Dirichlet forms of van Beusekom seem to correspond to lower derivatives of extended nonlinear Dirichlet forms of Cipriani and Grillo under some additional conditions. It might be worthwhile to make this relation precise. 
  
  \item If $p \geq 2$ and $\mu(X) = \infty$, it can be shown that $(\cE_{p})_e$ equals the functional $\cE^p$ considered in \cite{Kuw24}. Hence, also in this case we can recover and extend results from \cite{Kuw24}.  Again,  this is somewhat technical and we do not provide details here.
 \end{enumerate}

\end{remark}

%

\appendix

\section{Generalized semimodulars} \label{appendix:generalized semimodulars}

\subsection{Definitions and basic properties} In this subsection we discuss basic properties of generalized semimodulars and their induced Luxemburg seminorms on vector spaces. Most of the results are well-known to experts. Since they are somewhat scattered in the literature, we include them for the convenience of the reader.

Let $E$ be a $\mathbb K$-vectorspace with $\mathbb K \in \{\R,\IC\}$. A function $\vr \colon E \to[0,\infty]$ is called {\em proper} if its {\em effective domain} $D(\vr) = \{x \in E \mid \vr(x) < \infty\}$ is nonempty. It is called {\em left-continuous} if 
$$\lim_{\lambda \to 1-} \vr(\lambda x) = \vr(x)$$
for all $x \in E$. We say that $\vr$ is {\em absolutely convex}, if for all $\lambda,\mu \in \mathbb K$ with $|\lambda| + |\mu| \leq 1$ and all $x,y \in E$ it satisfies 
$$\vr(\lambda x  + \mu y) \leq |\lambda| \vr(x) + |\mu|\vr(y).$$
It is readily verified that a convex functional $\vr$ is absolutely convex if and only if $\vr(0) = 0$ and $\vr(\lambda x)  =\vr(x)$ for all $x \in E$ and all $\lambda \in \IK$ with $|\lambda| = 1$.

\begin{definition}[Generalized semimodular] \label{definition:generalized semimodular}
A function $\varrho \colon E \to [0,\infty]$ is called {\em generalized semimodular} (gsm for short) if it is proper, absolutely convex and left-continuous.  

\begin{remark}
 In the literature, see e.g. \cite{DHHR11}, a gsm $\vr \colon E \to [0,\infty]$ is called {\em semimodular} if for $x \in E$ the equality $\vr(\lambda x) =  0$ for all $\lambda > 0$ implies $x = 0$. Moreover, it is called {\em modular} if $\vr(x) = 0$ implies $x = 0$. For the examples that we want to treat it is quite important  to not assume these stronger properties.  
\end{remark}

\end{definition}

Given a gsm $\vr$ we introduce the corresponding {\em modular space} 
$$M(\vr) = \{x \in E \mid \lim_{\lambda \to 0+} \vr(\lambda x) = 0\}.  $$
By absolute convexity of $\vr$ it coincides with the linear hull of $D(\vr)$. For $r > 0$ the set  
$$B_r^\vr = \{x \in E \mid \vr(x) \leq r\}$$
is absolutely convex and absorbing in the vector space $M(\vr)$. The  Minkowski functional induced by $B_r^\vr$ on $M(\vr)$ is denoted by $\av{\cdot}_{L,r}$. By standard theory it is a seminorm on $M(\vr)$ and is given by 
$$\av{x}_{L,r} = \inf \{\lambda > 0 \mid \lambda^{-1} x \in B_r^\vr\} = \inf \{\lambda > 0 \mid \vr(\lambda^{-1}x) \leq r\}.$$
We extend it to $E\setminus M(\vr)$ by letting $\av{x}_{L,r} = \infty$ for $x \in E\setminus M(\vr)$. 

The family $\av{\cdot}_{L,r}$, $r>0$, is called family of {\em Luxemburg seminorms} of $\vr$. Sometimes we simply call $\av{\cdot}_L = \av{\cdot}_{L,1}$ the {\em Luxemburg seminorm} of $\vr$.

\begin{lemma}[Basic properties of the Luxemburg seminorms]\label{lemma:equivalent luxemburg norms} Let $\vr$ be a gsm on $E$. For $x \in E$ and $r,s > 0$  the following hold. 
\begin{enumerate}[(a)]
 \item $\vr(x) \leq r$ if and only if $\av{x}_{L,r} \leq 1$. 
 \item $\vr(x) \geq r$ implies $\av{x}_{L,r} \leq \vr(x)/r$. 
 \item $\vr(x) = \inf \{r > 0 \mid \av{x}_{L,r} \leq 1\}$, with the convention $\inf \emptyset = \infty$. 
 \item The seminorms $\av{\cdot}_{L,r}$ and $\av{\cdot}_{L,s}$ are equivalent. More precisely, if $s \leq r$, then 
 $$\av{\cdot}_{L,r}\leq \av{\cdot}_{L,s} \leq \frac{r}{s} \av{\cdot}_{L,r}.$$
 \item $\av{x}_L = 0$ if and only if $\vr(\lambda x) =  0$ for all $\lambda > 0$.
 \end{enumerate}

\end{lemma}
\begin{proof}
 (a): This can be proven as in \cite[Lemma~2.1.14]{DHHR11} (which only treats semimodulars).  
 
 (b): If $\vr(x)  \geq r$, we obtain using absolute convexity
 $$\vr(\frac{r}{\vr(x)} x) \leq \frac{r}{\vr(x)} \vr(x) = r.$$
 This shows $\av{x}_{L,r} \leq \vr(x)/r$.
 
 (c): Clearly, $x \in M(\vr)$ if and only if there exists $r > 0$ such that $\av{x}_{L,r} \leq 1$. With this observation the result follows directly from (a).
 
 (d): Without loss of generality we assume $s \leq r$. Clearly, this implies $\av{\cdot}_{L,r} \leq \av{\cdot}_{L,s}$. For the converse inequality we let $\lambda = s/r \leq 1$. If $\vr(x) \leq r$, then by convexity we have
 $$\vr(\lambda x) \leq \lambda \vr(x) \leq \lambda r  = s. $$
 This shows $\av{x}_{L,s} \leq 1/\lambda = r/s$ as long as $\vr(x) \leq r$. In other words, 
 \begin{align*}
\av{\cdot}_{L,s} &\leq r/s \av{\cdot}_{L,r}.
 \end{align*}
 (e): This follows  directly from the definition of $\av{\cdot}_L$ and the absolute convexity of $\vr$. 
%
\end{proof}

\begin{corollary} \label{corollary:lower semicontinuity with respect to Luxemburg seminorm}
Every gsm $\vr$ is lower semicontinuous on $M(\vr)$ with respect to $\av{\cdot}_L$-convergence. 
\end{corollary}
\begin{proof}
 The functional $\vr$ is lower semicontinuous on the space $(M(\vr),\av{\cdot}_L)$ if and only if for each $r > 0$ the set  
 $$ \{x \in M(\vr) \mid \vr(x) \leq r\} = \{x \in M(\vr) \mid \av{x}_{L,r} \leq 1\}  $$
 is $\av{\cdot}_L$-closed, where we used the previous lemma for the equality. The set on the right side is closed because all Luxemburg seminorms are equivalent. 
\end{proof}

\begin{remark}
 This is an extension of \cite[Theorem~2.1.17]{DHHR11} with a simplified proof. 
\end{remark}

\begin{definition}[$\Delta_2$-conditions]\label{definition:delta 2 condition}
 A gsm $\vr$ on $E$ is said to satisfy the {\em $\Delta_2$-condition} if there exists $K > 0$ such that 
 $$\vr(2x) \leq K \vr (x) $$
 for all $x \in E$. The functional $\vr$ is said to satisfy the {\em weak $\Delta_2$-condition} if for all sequences $\vr(x_n) \to 0$ implies $\vr(2x_n) \to 0$. Moreover, $\vr$ is said to satisfy the {\em very weak $\Delta_2$-condition} if $\vr(x) = 0$ implies $\vr(2x) = 0$. 
\end{definition}

\begin{remark}
 Clearly, the $\Delta_2$-condition implies the weak $\Delta_2$-condition, which in turn implies the very weak $\Delta_2$-condition. The converse implications do not hold in general. 
\end{remark}

\begin{lemma}[Characterizing the $\Delta_2$-conditions] \label{lemma:delta2 conditions}
Let $\vr$ be a gsm on $E$. 
 \begin{enumerate}[(a)]
  \item The following assertions are equivalent. 
  \begin{enumerate}[(i)]
    \item $\vr$ satisfies $\Delta_2$.
  \item For all $C > 1$ there exists $\varepsilon > 0$ such that 
  $$\vr((1 + \varepsilon) x) \leq C \vr(x)$$
  for all $x \in E$. 
  \end{enumerate}
  In this case, $D(\vr) = M(\vr)$ and for $s < r$ there exists $\delta > 0$ such that $\av{\cdot}_{L,r} \leq (1-\delta) \av{\cdot}_{L,s}$.
  
  \item $\vr$ satisfies the weak $\Delta_2$-condition if and only if for all sequences $(x_n)$ in $E$ the following holds:   $\av{x_n}_L \to 0$ if and only if $\vr(x_n) \to 0$. 
  
  \item $\vr$ satisfies the very weak $\Delta_2$-condition if and only if for all $x \in E$ the following holds: $\av{x}_L = 0$ if and only if $\vr(x) = 0$. 
 \end{enumerate}

\end{lemma}
\begin{proof}
(a): (i) $\Rightarrow$ (ii):   Let $C > 1$ and let $K > 0$ such that $\vr(2x) \leq K \vr(x)$ for all $x\in E$. Without loss of generality we can assume $K > C$. Let $\varepsilon = (C-1)/(K-1) < 1$. Using the convexity of $\vr$ we infer
\begin{align*}
 \vr((1 + \varepsilon) x) &= \vr((1-\varepsilon) x + 2 \varepsilon x) \\
 &\leq (1-\varepsilon) \vr(x) + \varepsilon \vr(2x) \\
 &\leq (1-\varepsilon) \vr(x) + K\varepsilon \vr(x)\\
 &= C \vr(x).
\end{align*} 

(ii) $\Rightarrow$ (i): Let $C > 0$ and choose an appropriate $\varepsilon > 0$ such that  $\vr((1 + \varepsilon) x) \leq C \vr(x)$ for all $x \in E$. Choose $n \in \N$ such that $2 \leq (1+\varepsilon)^n$. The absolute convexity of $\vr$ implies 
$$\vr(2 x) \leq \vr((1+\varepsilon)^n x) \leq C^n \vr(x).$$

Now assume that $\Delta_2$ holds. The convexity of $\vr$ and $\Delta_2$ imply the linearity of $D(\vr)$. Since $M(\vr)$ is the linear hull of $D(\vr)$, we obtain $D(\vr) = M(\vr)$.  Let $s < r$.   Using (ii) for $C = r/s > 1$ we find $\varepsilon > 0$ such that $\vr((1+\varepsilon)x) \leq C \vr(x)$ for $x \in E$. For $x \in E$ with $\vr(x) \leq s$, this implies 
$$\vr((1+\varepsilon)x) \leq C s = r,$$
or in other words $\av{x}_{L,r} \leq 1/(1+\varepsilon)$. We infer $\av{\cdot}_{L,r} \leq (1-\delta) \av{\cdot}_{L,s}$ with $\delta = \varepsilon/(1+\varepsilon)$.

(b):  This can be proven as in \cite[Lemma~2.1.11]{DHHR11} (which only treats semimodulars). 

(c): It is straightforward from the definition of $\av{\cdot}_L$ and the absolute convexity of $\vr$ that $\av{x}_L =  0$ if and only if $\vr(2^n x) \leq 1$ for all $n \in \N$.  By absolute convexity we also have 
$$\vr(x) \leq 2^{-n} \vr(2^n x)$$
for all $x \in E$. Together, these observations imply $\av{x}_L = 0$ if and only if $\vr(2^n x) = 0$ for all $n \in \mathbb N$. This description of the kernel of $\av{\cdot}_L$ implies the claimed equivalence. 
%
%
%
\end{proof}

\begin{remark}
 The proof of (a) is taken from the proof of \cite[Lemma~2]{CS94}.
\end{remark}

\begin{lemma} \label{lemma:Delta_2 and Luxemburg seminorm convergence yield convergence of gsm}
 Let $\vr$ be a gsm on $E$ satisfying the $\Delta_2$-condition. Then $\av{x_n - x}_L \to 0$ implies $\vr(x_n) \to \vr(x)$. 
\end{lemma}

\begin{proof}
Corollary~\ref{corollary:lower semicontinuity with respect to Luxemburg seminorm} shows
 \[ \rho(x) \leq \liminf_{n \to \infty} \rho(x_n) . \]
To show the converse inequality, let $r > s > \rho(x)$.
Using the $\Delta_2$-condition and Lemma~\ref{lemma:delta2 conditions}, we obtain a $\delta \in (0,1)$ with $\|\cdot\|_{L,r} \leq (1-\delta) \|\cdot\|_{L,s}$.
Because the Luxemburg seminorms are equivalent, $\|x_n\|_{L,s} \to \|x\|_{L,s}$ and there exists $N \in \N$ such that
 \[ \|x_n\|_{L,r} \leq (1-\delta) \|x_n\|_{L,s} \leq \|x\|_{L,s} \]
for all $n \geq N$.
Since $\rho(x) < s$, we have $\|x\|_{L,s} \leq 1$ so that $\rho(x_n) \leq r$ for all $n \geq N$.
Because this $r > \rho(x)$ was arbitrary, we obtain
 \[ \limsup_{n \to \infty} \rho(x_n) \leq \rho(x) . \qedhere \]
\end{proof}

\subsection{Lower semicontinuity}

Recall that an {\em $F$-norm} on $E$ is a functional $p \colon E \to [0,\infty)$ satisfying the following properties: $p(x)  = 0$ if and only if $x = 0$ (non-degeneracy), $p(x + y) \leq p(x) + p(y)$ for all $x,y \in E$ (triangle inequality), $p(|\lambda| x) \leq p(x)$ for all $|\lambda| \leq 1$, $x \in E$ (balls around $0$ are circled) and $p(n^{-1}x) \to 0$ for all $x \in E$ (continuity of scalar multiplication at $0$). 

We now assume that $E$ is endowed with a  vector space topology induced by an $F$-norm $p$. More precisely, we consider the vector space topology induced by the metric $d \colon E \times E \to [0,\infty)$, $d(x,y)  = p(x-y)$.  Indeed,  any metrizable vector space topology is induced by an $F$-norm and  completeness of $(E,p)$ does not depend on the particular choice of $p$. In this subsection (semi)continuity is always considered with respect to the topology induced by $p$ via $d$.   

For a gsm $\vr$ the seminormed space $(M(\vr),\av{\cdot}_L)$ isometrically embeds into its bidual $(M(\vr))''$ via the map 
$$\iota \colon M(\vr) \to (M(\vr))'', \quad x \mapsto \as{\cdot,x},$$
where $\as{\cdot,\cdot}´$ denotes the dual pairing between $M(\vr)$ and $(M(\vr))'$.

\begin{definition}[Reflexive gsm]
 We call $\vr$ {\em reflexive} if $\iota(M(\vr))$ is norm dense in $(M(\vr))''$.
\end{definition}

\begin{remark}
 Reflexivity of $\vr$ is equivalent to reflexivity of the completion of the quotient $(M(\vr) / \ker \av{\cdot}_L, \av{\cdot}_L)$. Since all Luxemburg seminorms are equivalent, reflexivity of $\vr$ is independent of the choice of the parameter in the Luxemburg seminorm. This is why we attribute reflexivity to $\vr$ instead of $\av{\cdot}_L$.
\end{remark}

The next theorem is one of the main results in this subsection. It characterizes the lower semicontinuity of $\vr$ in terms of the lower semicontinuity of the Luxemburg seminorms.

\begin{theorem}\label{theorem:lower semicontinuity}
 Let $\vr$ be a gsm on $(E,p)$.
 \begin{enumerate}[(a)]
  \item The following assertions are equivalent. 
 \begin{enumerate}[(i)]
  \item $\vr$ is lower semicontinuous.
  \item For all $r > 0$ the functional $\av{\cdot}_{L,r}$ is lower semicontinuous (as a function $E \to [0,\infty]$).
 \end{enumerate}

 \item If $\vr$ is reflexive, the assertions in (a) are equivalent to: 
 \begin{enumerate}[(i)] \setcounter{enumii}{2} 
  \item $\av{\cdot}_L$ is lower semicontinuous (as a function $E \to [0,\infty]$).
 \end{enumerate}
 \item If $\vr$ is reflexive and $(E,p)$ is complete, the assertions in (a) and (b) are equivalent to: 
 \begin{enumerate}[(i)] \setcounter{enumii}{3} 
  \item  $(M(\vr),p + \av{\cdot}_L)$ is complete.
 \end{enumerate}
 \end{enumerate}
\end{theorem}

Before proving this result, we discuss two technical lemmas.   

\begin{lemma}\label{lemma:converging convex combinations}
 Let $(x_n)$, $x$ in $(E,p)$ with 
 $$\sum_{n=1}^\infty p(x-x_n) < \infty.$$
 Then for every sequence $(y_n)$ in $E$ with $y_n \in {\rm conv}\, \{x_{k} \mid k \geq n\}$, $n \in \mathbb N$, we have $y_n \to x$ with respect to $p$. 
\end{lemma}
\begin{proof}
For $y = \sum_{k=n}^\infty \lambda_k x_k \in  {\rm conv}\, \{x_{k} \mid k \geq n\}$ we obtain using the properties of $F$-norms
$$p(x - y) \leq \sum_{k = n}^\infty p(\lambda_k(x-x_k)) \leq \sum_{k = n}^\infty p(x-x_k). $$
This inequality and the assumption $\sum_{n=1}^\infty p(x-x_n) < \infty$ imply the statement. 
\end{proof}

\begin{lemma}\label{lemma:weak convergence} 
 Let $\vr$ be a reflexive lower semicontinuous gsm on $(E,p)$. If $(x_n)$ is bounded in $(M(\vr),\av{\cdot}_L)$ and converges to $x \in E$ with respect to $p$, then $x_n \to x$ weakly in $(M(\vr),\av{\cdot}_L)$. 
\end{lemma}
\begin{proof}
Using a standard subsequence argument, it suffices to show that $(x_n)$ has a subsequence converging weakly to $x$. Hence, without loss of generality, we can assume  $\sum_{n=1}^\infty p(x-x_{n}) < \infty$. By the weak sequential compactness of balls in the completion of $(M(\vr),\av{\cdot}_L)$, we obtain a subsequence $(x_{n_k})$ that converges weakly to some $\overline x$ in the completion. Mazur's lemma yields a sequence $(y_n)$ with the following properties:
 \begin{itemize}
  \item $(y_n)$  converges to $\overline{x}$ in the completion of $(M(\vr),\av{\cdot}_L)$. 
  \item $(y_n)$ is a finite convex combination of the elements of $\{x_{n_k} \mid k \geq n\}$.
  \end{itemize}
  Lemma~\ref{lemma:converging convex combinations} shows  $y_n \to x$ with respect to $p$. With this at hand, the lower semicontinuity of $\av{\cdot}_L$ implies 
$$\av{x - \overline{x}}_L = \lim_{n\to \infty} \av{x - y_n}_L \leq \liminf_{n,m \to \infty} \av{y_m  - y_n}_L = 0.$$
This implies $\varphi(x) = \varphi(\overline x)$ for each $\varphi \in M(\vr)'$ and so we obtain the weak convergence of $(x_{n_k})$ to $x$.  
\end{proof}

\begin{proof}[Proof of Theorem~\ref{theorem:lower semicontinuity}]

(a): (i) $\Leftrightarrow$ (ii): This follows from two observations: 1. The functional $\vr$ is lower semicontinuous with respect to $p$ if and only if for all $r > 0$ the set 
$$\{x \in E \mid \vr(x) \leq r\} = \{x \in E \mid \av{x}_{L,r} \leq 1\} $$
is closed with respect to $p$. 2. By a scaling argument the functional $\av{\cdot}_{L,r}$ is lower semicontinuous if and only if its unit ball $\{x \in E \mid \av{x}_{L,r} \leq 1\}$ is closed with respect to $p$. 

(b): Now we assume that $\vr$ is reflexive. (ii) $\Rightarrow$ (iii) is trivial.

(iii) $\Rightarrow$ (ii): Let $(x_n), x$ in $M(\vr)$ with $x_n \to x$ with respect to $p$. We show $\av{x}_{L,r} \leq \liminf_{n\to \infty} \av{x_n}_{L,r}$. Without loss of generality we assume that $(x_n)$ is $\av{\cdot}_{L,r}$-bounded (else pass to a subsequence, without bounded subsequence there is nothing to show). Since all Luxemburg seminorms are equivalent, $(x_n)$ is also $\av{\cdot}_L$-bounded. Lemma~\ref{lemma:weak convergence} yields $x_n \to x$ weakly in $M(\vr)$, which does not depend on the choice of the Luxemburg seminorm. Since every seminorm is lower semicontinuous with respect to weak convergence, we infer 
%
%
%
%
$\av{x}_{L,r} \leq \lim_{n\to \infty}\av{x_n}_{L,r}.$
%
 
(c): Now we assume that $\vr$ is reflexive and $(E,p)$ is complete.

(iii) $\Rightarrow$ (iv): Let $(x_n)$ be $p  + \av{\cdot}_L$-Cauchy. The completeness of $(E,p)$ yields $x \in E$ such that $x_n \to x$ with respect to $p$. The lower semicontinuity of $\av{\cdot}_L$ implies
$$\av{x - x_n}_L \leq \liminf_{m \to \infty} \av{x_m - x_n}_L. $$
Since $(x_n)$ is also $\av{\cdot}_L$-Cauchy, we infer $x_n \to x$ with respect to $p + \av{\cdot}_L$. 

(iv) $\Rightarrow $(iii): We use a similar argument as in the proof of (iii) $\Rightarrow$ (ii) and we refer there for more details. Let $x_n \to x$ with respect to $p$. Without loss of generality we assume 
$$\liminf_{n\to \infty} \av{x_n}_L = \lim_{n\to \infty}\av{x_n}_L  < \infty$$
and
$$\sum_{n = 1}^\infty p(x-x_n) < \infty. $$
Using the $\av{\cdot}_L$-boundedness of $(x_n)$ and the reflexivity of $\vr$ we find $(y_n)$ such that $(y_n)$ is $\av{\cdot}_L$-Cauchy and $y_n \in {\rm conv}\, \{x_k \mid k \geq n\}$. Lemma~\ref{lemma:converging convex combinations} yields $y_n \to x$ with respect to $p$. Overall, $(y_n)$ is $p + \av{\cdot}_L$-Cauchy. By completeness, it $p + \av{\cdot}_L$-converges to some $y \in E$. Since $p$-limits are unique, we infer $x =  y$ and 
$$\av{x}_L = \lim_{n \to \infty}\av{y_n}_L \leq \lim_{n \to \infty}\av{x_n}_L.$$
For the last inequality we used convexity of the seminorm. 
\end{proof}

\begin{remark}
 The equivalence of (i) and (ii) is more or less trivial. The equivalence with (iii) under the reflexivity condition seems to be a new observation. The equivalence with (iv) was first observed in \cite{Schmi3} for quadratic forms, see also \cite[Lemma~A.4]{Schmi2}.  
\end{remark}

Next we discuss when $\vr$ can be extended to a lower semicontinuous functional. We say that $\vr'\colon E \to [0,\infty]$ is an {\em extension} of $\vr$ (or that $\vr$ is a {\em restriction} of $\vr'$) if $D(\vr) \subseteq D(\vr')$ and $\vr(x) = \vr'(x)$ for $x \in D(\vr)$. 
Moreover, we say that $\vr$ is {\em lower semicontinuous on $C \subseteq E$} if the restriction $\vr|_C \colon C \to [0,\infty]$ is lower semicontinuous with respect to the metric induced by $p$ on $C$.

The {\em lower semicontinuous relaxation} of $\vr$ is defined by 
$${\rm sc}\, \vr \colon E \to [0,\infty], \quad {\rm sc}\, \vr (x) = \sup\{f(x) \mid f \in C(E,[0,\infty)) \text{ with } f\leq \vr\}.$$
The following proposition provides a standard description of the lower semicontinuous relaxation. We mention it for completeness but omit a proof. 

\begin{proposition}\label{proposition:lsc relaxation}
 The lower semicontinuous relaxation ${\rm sc}\, \vr$ is the largest lower semicontinuous minorant of $\vr$.
 If $\vr$ is lower semicontinuous on $D(\vr)$, then ${\rm sc}\, \vr$ is a lower semicontinuous extension of $\vr$. Among all lower semicontinuous extensions of $\vr$ it has the smallest effective domain, which is  given by
 $$D({\rm sc}\, \vr) = \{x \in E \mid \text{ex. $(x_n)$ s.t. } p(x_n - x) \to 0 \text{ and } \sup_n \vr(x_n) < \infty\}. $$
 Furthermore,  for $x \in D({\rm sc}\, \vr)$ there exists $(x_n)$ in $D(\vr)$ such that $x_n \to x$ with respect to $p$ and
 $${\rm sc}\, \vr(x) = \lim_{n\to \infty} \vr(x_n).$$
\end{proposition}

\begin{remark}
 The previous proposition will be applied below to $\vr$ and to the Luxemburg seminorms $\av{\cdot}_{L,r} \colon E \to [0,\infty]$, which are  convex functionals with effective domain  $M(\vr)$. 
\end{remark}

\begin{definition}[Extended functional and Luxemburg seminorm]
 Let $\vr$ be a gsm on $(E,p)$ and assume that $\vr$ is lower semicontinuous on $D(\vr)$. Then we write $\vr_e = {\rm sc}\, \vr$ and call $\vr_e$ the  the {\em extended functional} of $\vr$. Similarly, if $\av{\cdot}_{L,r}$ is lower semicontinuous on $M(\vr)$,  we  write $\av{\cdot}_{L_e,r} = {\rm sc}\, \av{\cdot}_{L,r}$ and call it the {\em extended Luxemburg seminorm}. 
\end{definition}


The next lemma shows that $\av{\cdot}_{L_e,r}$ is a Luxemburg seminorm of $\vr_e$.

\begin{lemma}\label{lemma:extended luxemburg norm}
Assume that $\vr$ is lower semicontinuous on $M(\vr)$.
Then for every $x \in E$ we have 
$$\av{x}_{L_e,r} = \inf \{\lambda > 0 \mid \vr_e(\lambda^{-1} x) \leq r\},$$
(with the convention $\inf \emptyset = \infty$) and  
$$M(\vr_e) = \{x \in E \mid \av{x}_{L_e,r} < \infty\}.$$
Moreover, $\vr$ satisfies the (weak) $\Delta_2$-condition if and only if $\vr_e$ does.
\end{lemma}

\begin{proof}
Applying Theorem~\ref{theorem:lower semicontinuity} to the gsm $\vr|_{M(\vr)}$ on $(M(\vr),p)$ yields that $\vr$ is lower semicontinuous on $M(\vr)$ if and only if the Luxemburg seminorms $\|\cdot\|_{L,r}$, $r > 0$, are lower semicontinuous on $M(\rho)$.

For $x \in E$ let 
$$N(x) =  \inf \{\lambda > 0 \mid \vr_e(\lambda^{-1} x) \leq r\}.$$

Assume $\av{x}_{L_e,r} < \infty$ and choose a sequence $(x_n)$ in $M(\vr)$ such that $x_n \to x$ with respect to $p$ and $\av{x}_{L_e,r} = \lim_{n \to \infty} \av{x_n}_{L,r}$. For $\varepsilon > 0$ let $\lambda_n = \av{x_n}_{L,r} + \varepsilon$ and $\lambda = \av{x}_{L_e,r} + \varepsilon$.  Then $\lambda_n^{-1} x_n \to \lambda^{-1} x$ with respect to $p$ and $\vr(\lambda_n^{-1} x_n) \leq r$. The lower semicontinuity of $\vr_e$ and $\vr_e = \vr$ on $D(\vr)$ yield 
$$\vr_e(\lambda^{-1} x) \leq \liminf_{n \to \infty} \vr(\lambda_n^{-1} x_n) \leq r, $$
showing 
$$N(x) \leq \lambda  = \av{x}_{L_e,r} + \varepsilon.$$
This implies $N \leq \av{\cdot}_{L_e,r}$.

To prove the opposite inequality, assume $N(x) < \infty$ and let $\lambda > 0$ such that $\vr_e(\lambda^{-1} x) \leq r$. We choose a sequence $(x_n)$ in $D(\vr)$ with $x_n \to x$ with respect to $p$ and $\vr(\lambda^{-1} x_n) \to \vr_e(\lambda^{-1} x)$. Let $\varepsilon > 0$ arbitrary. Without loss of generality we can assume $\vr(\lambda^{-1} x_n) \leq r + \varepsilon$ for all $n \in \N$, which implies 
$$\av{x_n}_{L,r+\varepsilon} \leq \lambda.$$
The lower semicontinuity of $\av{\cdot}_{L_e,r+\varepsilon}$ and the inequality for the Luxemburg seminorms with different parameters yield 
$$\frac{r}{r+\varepsilon} \av{x}_{L_e,r} \leq  \av{x}_{L_e,r+\varepsilon}\leq \liminf_{n \to \infty} \av{x_n}_{L,r+\varepsilon} \leq \lambda.$$
Here, the first inequality follows from the description of the lower semicontinuous relaxation $\|\cdot\|_{L_e,r+\eps}$ and the lower semicontinuity of $\|\cdot\|_{L_e,r}$.
Since $\varepsilon > 0$ was arbitrary, we obtain $\av{\cdot}_{L_e,r} \leq N$.

With the explicit description of $\vr_e  = {\rm sc}\, \vr$ it is readily verified that if $\vr$ satisfies the $\Delta_2$-condition, then $\vr_e$ satisfies the $\Delta_2$-condition.

The lower semicontinuity of $\rho$ on $M(\rho)$ implies $\rho = \rho_e$ on $M(\rho)$.
Therefore, if $\vr_e$ satisfies the $\Delta_2$-condition with constant $K$, then we have
 \[ \rho(2x) = \rho_e(2x) \leq K \rho_e(x) = K \rho(x) \]
for all $x \in M(\rho)$.
For $x \in E \setminus M(\rho)$, there is nothing to show.

The weak $\Delta_2$-condition is equivalent to the following: For each $\varepsilon >0$ there exists $\delta > 0$ such that $\vr(f) \leq \delta$ implies $\vr(2f) \leq \varepsilon$. Let now $\varepsilon > 0$ arbitrary and choose $\delta > 0$ as in the last sentence. Let $(f_n)$ such that $\vr_e(f_n) \to 0$. Fix $N \in \N$ such that $\vr_e(f_n) \leq \delta/2$ for $n \geq N$ and for $n \in \N$ choose sequences $(g_{m,n})$ with $g_{m,{n}} \to f_n$ with respect to $p$, as $m \to \infty$, and   $|\vr(g_{m,n}) - \vr_e(f_n)| \to 0$, as $m \to \infty$. For $n \geq N$ and $m$ large enough we have $|\vr(g_{m,n}) - \vr_e(f_n)| \leq \delta/2$, which implies  $\vr(g_{m,n}) \leq \delta$ and leads to $\vr(2g_{m,n}) \leq \varepsilon$ for $n \geq N$ and $m$ large enough. Using the lower semicontinuity of $\vr_e$, we arrive at 
$$\vr_e(2 f_n) \leq \liminf_{m \to \infty} \vr(2g_{m,n}) \leq \varepsilon,$$
for $n \geq N$. This implies $\vr_e(2f_n) \to 0$, as $n \to \infty$. 
%
%
%

If $\rho_e$ satisfies the weak $\Delta_2$-condition and $(x_n)$ is a sequence in $E$ with $\rho(x_n) \to 0$, then $x_n \in M(\rho)$ eventually so that $\rho_e(x_n) = \rho(x_n) \to 0$ and, hence, $\rho(2 x_n) = \rho_e(2 x_n) \to 0$.
\end{proof}

\begin{remark}
 The very weak $\Delta_2$-condition need not extend to $\vr_e$. 
\end{remark}

The next theorem gives equivalent conditions for the lower semicontinuity of $\vr$ on $M(\vr)$ and of $\av{\cdot}_L$ on $M(\vr)$. Moreover, it provides explicit formulas for $\vr_e$ and $\av{\cdot}_{L_e}$. For a better readability of its statements and for later purposes we make one more definition. 

\begin{definition}[Approximating sequence]\label{definition:approximating sequence}
Let $\vr$ be a gsm on $(E,p)$ and let $x \in E$. A sequence $(x_n)$ in $M(\vr)$ is called {\em approximating sequence for} $x$ if $(x_n)$ is $\av{\cdot}_L$-Cauchy and $x_n \to x$ with respect to $p$. 
 
\end{definition}

\begin{theorem}\label{theorem:description extended space}
  Let $\vr$ be a gsm on $(E,p)$. 
  \begin{enumerate}[(a)]
   \item  The following assertions are equivalent. 
 \begin{enumerate}[(i)]
  \item $\vr$ is lower semicontinuous on $M(\vr)$.
  \item For all $r > 0$ the functional $\av{\cdot}_{L,r}$ is lower semicontinuous on $M(\vr)$.
 \end{enumerate}
 
 \item  If $\vr$ is reflexive, the assertions in (a) are equivalent to:
 \begin{enumerate}[(i)] 
  \item[(iii)] $\av{\cdot}_L$ is lower semicontinuous on $M(\vr)$.
 \end{enumerate}
 \item If $\vr$ is reflexive and $(E,p)$ is complete, the assertions in (a) and (b) are equivalent to: 
 \begin{enumerate}[(i)] 
  \item[(iv)]  For all $\av{\cdot}_L$-Cauchy sequences $(x_n)$ in $M(\vr)$ with $p(x_n) \to 0$ we have $\av{x_n}_L \to 0$.  
 \end{enumerate}
 In the latter case, $x \in M(\vr_e)$ if and only if there exists an approximating sequence for $x$ and for every approximating sequence $(x_n)$ for  $x$ we have
 $$\av{x}_{L_e} = \lim_{n \to \infty} \av{x_n}_L.$$
 \item Assume that $\vr$ is reflexive, $(E,p)$ is complete and that the equivalent assertions (i) - (iv) hold. If $\vr$ satisfies $\Delta_2$, then $\vr_e$ satisfies $\Delta_2$,  $D(\vr_e) = M(\vr_e)$ and for every approximating sequence $(x_n)$ for $x$ we have 
 $$\vr_e(x)  = \lim_{n \to \infty} \vr(x_n).$$

  \end{enumerate}
\end{theorem}
\begin{proof}
 (a) and (b): This follows directly by applying Theorem~\ref{theorem:lower semicontinuity} to the gsm $\vr|_{M(\vr)}$ on $(M(\vr),p)$ and observing that $\rho|_{M(\rho)}$ is lower semicontinuous if and only if $\rho$ is lower semicontinuous on $M(\rho)$.
 
 (c): (iii) $\Rightarrow$ (iv): Since (iii) implies (i), $\vr_e$ is a lower semicontinuous extension of $\vr$. By (iii) and Lemma~\ref{lemma:extended luxemburg norm} the extended Luxemburg seminorm  $\av{\cdot}_{L_e}$ exists and equals the Luxemburg seminorm of $\vr_e$. Hence, Theorem~\ref{theorem:lower semicontinuity} applied to $\vr_e$ implies that $M(\vr_e)$ is complete with respect to $p + \av{\cdot}_{L_e}$. Now let $(x_n)$ be $\av{\cdot}_L$-Cauchy with $p(x_n) \to 0$. Since $\av{\cdot}_{L_e}$ is an extension of $\av{\cdot}_L$, we infer that $(x_n)$ is $p + \av{\cdot}_{L_e}$-Cauchy. By completeness it converges to $x \in M(\vr_e)$ with respect to $p + \av{\cdot}_{L_e}$. Since $p(x_n) \to 0$ and $p$-limits are unique, we infer $x =  0$, showing $\av{x_n}_L = \av{x_n}_{L_e} \to 0$. 
 
 (iv) $\Rightarrow$ (iii): We define $N \colon E \to [0,\infty]$ by 
 $$ N(x) = \begin{cases}
                                      \lim\limits_{n \to \infty} \av{x_n}_L &\text{if $(x_n)$ is an approximating sequence for $x$}\\
                                      \infty &\text{if $x$ has no approximating sequence} 
                                     \end{cases}
$$
 and show that $N$ is the lower semicontinuous extension of $\av{\cdot}_L$ with minimal effective domain. This implies (iii) and the formula for $M(\vr_e)$ and $\av{\cdot}_{L_e}$. 
 
 $N$ is well-defined and an extension of $\av{\cdot}_L$: This is a direct consequence of (iv). 
 
 $M(\vr)$ is dense in $D(N)$ with respect to $p + N$: For $x \in D(N)$ we choose an approximating sequence $(x_n)$ for $x$. Then for each $m \in \mathbb N$ the element  $x - x_m$ has the approximating sequence $(x_n - x_m)_n$. This implies 
 $$\lim_{m \to \infty} N(x-x_m) = \lim_{m \to \infty} \lim_{n\to \infty} \av{x_n - x_m}_L = 0.$$

 $D(N)$ is complete with respect to $p + N$: Let $(x_n)$ be $p + N$-Cauchy. By completeness of $(E,p)$ there exists $x \in E$ with $x_n \to x$ with respect to $p$. Moreover, by density of $M(\vr)$ in $D(N)$ with respect to $p + N$ we can choose $y_n \in M(\vr)$ with $p(x_n -y_n) + N(x_n - y_n)  \to 0$. Then $p(x - y_n) \to 0$ and since $N$ is an extension of $\av{\cdot}_L$, the sequence $(y_n)$ is $\av{\cdot}_L$-Cauchy. In other words, $(y_n)$ is an approximating sequence for $x$, which implies $x \in D(N)$ and $N(x-x_n) \to 0$. 

 $N$ is lower semicontinuous: $N$ is a seminorm on $D(N)$ and hence coincides with its own Luxemburg seminorm. With this at hand lower semicontinuity of $N$ follows from the completeness of $(D(N),p + N)$ and Theorem~\ref{theorem:lower semicontinuity}. Note that the required reflexivity of $N$ follows from the density of $M(\vr)$ in $D(N)$ and the reflexivity of $\vr$.

 $N$ is the largest lower semicontinuous minorant: By definition $x \in D(N)$ if and only if it has an approximating sequence $(x_n)$. Now, if $M$ is a lower semicontinuous minorant of $\|\cdot\|_L$, we obtain
  \[ M(x) \leq \liminf_{n \to \infty} M(x_n) \leq \lim_{n \to \infty} \|x_n\|_L = N(x) . \]
 
 (d): We have already observed in Lemma~\ref{lemma:extended luxemburg norm} that $\vr_e$ inherits the $\Delta_2$ property from $\vr$.
 Since $\av{\cdot}_{L_e,r}$ is a Luxemburg seminorm of $\vr_e$, we can apply Lemma~\ref{lemma:equivalent luxemburg norms} and Lemma~\ref{lemma:delta2 conditions} to $\vr_e$ and $\av{\cdot}_{L_e,r}$. In particular, we obtain $D(\vr_e) = M(\vr_e)$. From (iv) it then follows that $D(\vr_e)$ consists of those elements of $E$ that possess an approximating sequence.

 Let $(x_n)$ be an approximating sequence for $x$. Then $\vr_e(x) < \infty$ and from the lower semicontinuity of $\vr_e$ and $\vr_e = \vr$ on $D(\vr)$ we infer 
 $$\vr_e(x) \leq \liminf_{n \to \infty} \vr(x_n).$$
 %
%
 %
 It remains to prove $\vr_e(x) \geq \limsup_{n \to \infty} \vr(x_n)$. Lemma~\ref{lemma:equivalent luxemburg norms} shows
 $$\vr_e(x)  = \inf \{r > 0 \mid \av{x}_{L_e,r} \leq 1\}.$$
 Let  $r > s > \vr_e(x)$. We use Lemma~\ref{lemma:delta2 conditions} and $\Delta_2$ to choose $0 < \delta < 1$ with $\av{\cdot}_{L_e,r} \leq (1-\delta)\av{\cdot}_{L_e,s}$.  The lower semicontinuity of $\av{\cdot}_{L_e,s}$ yields 
 $$\av{x - x_n}_{L_e,s} \leq \liminf_{m \to \infty}\av{x_m - x_n}_{L_e,s}= \liminf_{m \to \infty}\av{x_m - x_n}_{L,s}, $$
 showing $x_n \to x$ with respect to $\av{\cdot}_{L_e,s}$ (here we use that $(x_n)$ is also $\av{\cdot}_{L,s}$-Cauchy, which follows from the equivalence of the Luxemburg seminorms of $\vr$). In particular, we find $N \geq 1$ such that 
 $$\av{x_n}_{L_e,r} \leq (1-\delta)\av{x_n}_{L_e,s} \leq \av{x}_{L_e,s}$$
 for all $n \geq N$. Since $s > \vr_e(x)$, we have $\av{x}_{L_e,s} \leq 1$. This implies $\av{x_n}_{L_e,r} \leq 1$ for each $n \geq N$, which in turn yields $\vr_e(x_n) \leq r$ for all $n \geq N$, showing 
 $$\limsup_{n \to \infty} \vr_e(x_n) \leq r.$$
 Since $r > \vr_e(x)$ was arbitrary and since $\vr = \vr_e$ on $D(\vr)$, we obtain the claim.
 \end{proof}

 \begin{remark}
  For quadratic forms this theorem has been obtained in \cite{Schmi3}. In this case, it is a rather simple consequence of Theorem~\ref{theorem:lower semicontinuity}. In the general case the statement of (c) and the precise description of $\vr_e$ in (d) under $\Delta_2$ are more difficult to obtain.  
 \end{remark}

 \begin{lemma}\label{lemma:reflexivity and relaxation}
  Assume that $(E,p)$ complete.
  Then $\vr$ is reflexive if and only if $\vr_e$ is reflexive.
 \end{lemma}

 \begin{proof}
  It follows from Theorem~\ref{theorem:description extended space}~(c) that $M(\vr)$ is dense in $M(\vr_e)$ with respect to $\av{\cdot}_{L_e,r}$ if $\vr$ is reflexive. Hence, the bidual spaces $(M(\vr))''$ and $(M(\vr_e))''$ agree and $\vr_e$ is also reflexive.
  
  Conversely, let $\vr_e$ be reflexive.
  This means that the Banach space given by the completion of $(M(\vr_e)/\ker \|\cdot\|_{L_e},\|\cdot\|_{L_e})$ is reflexive.
  Then the completion of $(M(\vr)/\ker \|\cdot\|_L, \|\cdot\|_L)$ is also reflexive - it can be interpreted as a closed subspace of that space.
 \end{proof}

 At the end of this section we discuss completeness of $(M(\vr),\av{\cdot}_L)$ itself without adding $p$ to the Luxemburg seminorm. For a subspace $S \subseteq E$ the quotient $F$-seminorm is given by 
 $$p_{S} \colon \faktor{E}{S} \to [0,\infty),\quad  p_{S}(x + S) = \inf\{p(x+y) \mid y \in S\}.$$
 It is an $F$-norm on $\faktor E S$ if and only if $S$ is closed. Moreover, if $(E,p)$ is complete and $S$ is closed, then $(\faktor {E}{S},p_{S})$ is complete.

 \begin{theorem}\label{theorem:two imply the third}
  Let $(E,p)$ be complete, let $\vr$ be a reflexive gsm on $(E,p)$ and let $N = \{x \in M(\vr) \mid \av{x}_L = 0\}$. Of the following assertions each two imply the third.  
  \begin{enumerate}[(i)]
   \item $\vr$ is lower semicontinuous.
   \item $N$ is closed in $(E,p)$ and the embedding 
   $$(\faktor{M(\vr)}{N}, \av{\cdot}_L) \to (\faktor E N,p_{N}),\quad  x + N \mapsto x + N $$
   is continuous. 
   \item $(\faktor{M(\vr)}{N},\av{\cdot}_L)$ is a Banach space.
  \end{enumerate}
  
 \end{theorem}

 \begin{proof}
  (i) \& (ii) $\Rightarrow$ (iii): Let $(x_n + N)$ be Cauchy in $(\faktor{M(\vr)}{N},\av{\cdot}_L)$. By (ii) it is Cauchy in $(\faktor E N,p_{N})$. Since $N$ is closed and $(E,p)$ is complete, the latter space is complete and we obtain $x \in E$ such that $x_n + N \to x + N$ with respect to $p_N$. We infer the existence of  a sequence $(k_n)$ in $N$ with $x_n + k_n \to x$ with respect to $p$. Since the lower semicontinuity of $\vr$ implies the lower semicontinuity of $\av{\cdot}_L$, we infer 
  $$\av{x  - x_n}_L  \leq \liminf_{m \to \infty} \av{x_m + k_m - x_n}_L = \liminf_{m \to \infty} \av{x_m  - x_n}_L.$$
  This shows $x \in M(\vr)$ and $x_n + N \to x + N$ in $(\faktor{M(\vr)}{N},\av{\cdot}_L)$.
  
   (i) \& (iii) $\Rightarrow$ (ii): The lower semicontinuity of $\av{\cdot}_L$, which follows from the lower semicontinuity of $\vr$, implies that $N$ is closed. In particular, $(\faktor E N,p_{N})$ is complete. 
   
   The closed graph theorem holds in complete metrizable topological vector spaces (so-called $F$-spaces), see e.g. \cite[Theorem~3.8]{Hus65}.  Since $(\faktor{M(\vr)}{N}, \av{\cdot}_L)$ is also complete by (iii), it suffices to show that the embedding is closed. Assume $x_n + N \to x + N$ in $(\faktor{M(\vr)}{N}, \av{\cdot}_L)$  and $x_n + N \to y + N$ in $(\faktor E N,p_{N})$. We choose a sequence $(k_n)$ in $N$ with $p(x_n + k_n - y) \to 0$. The lower semicontinuity of $\av{\cdot}_L$ yields 
   $$\av{x  - y }_L \leq \liminf_{n \to \infty} \av{x - x_n  - k_n}_L = \av{x - x_n}_L \to 0, \quad  \text{ as } n\to \infty.  $$
   This shows  $x + N =  y + N$ and the closedness is proven. 
  
  (ii) \& (iii) $\Rightarrow$ (i): Since $\vr$ is reflexive, by Theorem~\ref{theorem:lower semicontinuity} it suffices to show that $(M(\vr),p + \av{\cdot}_L)$ is complete. Let $(x_n)$ be Cauchy in this space. The completeness of $(\faktor{M(\vr)}{N},\av{\cdot}_L)$ and $(E,p)$ yield $x, y \in E$ such that $x_n \to y$ with respect to $\av{\cdot}_L$ and $x_n \to  x$ with respect to $p$. The continuity assumption in (ii) yields $x + N = y + N$ so that also $x_n \to x$ with respect to $\av{\cdot}_L$. This shows the desired completeness.  
 \end{proof}

\begin{remark}\label{remark:two imply the third}
\begin{enumerate}[(a)]
 \item The implications (i) \& (ii) $\Rightarrow$ (iii) and (i) \& (iii) $\Rightarrow$ (ii) hold without assuming the reflexivity of $\vr$.
 
 \item For quadratic forms this theorem was first established in \cite{Schmi3}.

\end{enumerate}

 \end{remark}

\subsection{Dual spaces and resolvents}\label{subsection:dual spaces and resolvents}

In this subsection we study the dual space of a given modular space. A particular focus is being laid on an alternative description of $\vr$ using duality. 

For a vector space $E$ we let $E^*$ denote its algebraic dual space. If $E$ carries a vector space topology (e.g. induced by an $F$-norm), we write $E'$ for its continuous dual space. 

Throughout this subsection we assume that $E$ is an $\R$-vector space and $\vr$ is a gsm on $E$. Its {\em algebraic subgradient} is defined by 
$$\partial^a \vr = \{(x,\varphi) \in E \times E^* \mid \varphi(y-x)  + \vr(x) \leq \vr(y) \text{ for all } y \in E\} $$
and we let $\partial^a \vr(x) = \{\varphi \in E^* \mid (x,\varphi) \in  \partial^a \vr\}$. Since $\vr$ is proper, $(x,\varphi) \in \partial^a \vr$ implies $x \in D(\vr)$.

\begin{proposition}\label{prop:estimate operator norm}
For every $(x,\varphi) \in \partial^a \vr$ we have $\varphi \in (M(\vr),\av{\cdot}_L)'$ with $\av{\varphi}_{M(\vr)'} \leq 1 + \varphi(x) - \vr(x)$. 
\end{proposition}
\begin{proof}
 By definition $(x,\varphi) \in \partial^a \vr$ implies 
 $$\varphi(y)   \leq \vr(y) + \varphi(x) - \vr(x) $$
 for all $y \in E$. By the symmetry of $\vr$ this inequality also holds for $|\varphi(y)|$ on the left side of the inequality.  Since  $\vr(y) \leq 1$ if and only if $\av{y}_L \leq 1$, taking the sup over such $y$ yields the claim. 
\end{proof}

For $x \in D(\vr)$ we define the {\em directional derivative of $\vr$ at $x$} by
$$d^+ \vr(x,\cdot) \colon E \to (-\infty,\infty], \quad d^+ \vr(x,y) = \liminf_{\varepsilon\to 0+} \frac{\vr(x + \varepsilon y) - \vr(x)}{\varepsilon}.$$
\begin{lemma}\label{lemma:lower derivative}
For all $x \in D(\vr)$ and $y \in E$ we have 
$$ d^+ \vr(x,y-x) + \vr(x) \leq \vr(y).$$
In particular, if $d^+ \vr(x,\cdot)$ is a linear functional on $M(\vr)$, then $d^+ \vr(x,\cdot) \in \partial^a \vr(x)$ (more precisely any linear extension of $d^+ \vr(x,\cdot)$ to $E$ belongs to $\partial^a \vr(x)$).
\end{lemma}
\begin{proof}
Using the convexity of $\vr$ we obtain 
 \[ \vr(x + \varepsilon (y-x)) = \vr(\varepsilon y + (1-\varepsilon)x) \leq \varepsilon \vr(y) + (1-\varepsilon) \vr(x) \]
for all $\eps \in (0,1)$.
Rearranging this inequality yields
 $$  \frac{\vr(x + \varepsilon (y - x)) - \vr(x)}{\varepsilon} + \vr(x)\leq \vr(y)$$
 and we arrive at the statement after letting $\varepsilon \to 0+$.
 
 If $d^+ \vr(x,\cdot)$ is linear on $M(\vr)$, it can be extended to a linear functional on $E$. By the previously shown inequality any such extension belongs to $\partial^a \vr(x)$. 
\end{proof}

Next we assume that $(E,\av{\cdot})$ is a normed space and that $\vr$ is a gsm on $E$. In this case, we write $\partial \vr = \partial^a \vr \cap (E \times E')$ for the {\em subgradient} of $\vr$. 

The {\em convex conjugate} of $\vr$ is defined by 
$$\vr^* \colon E' \to [-\infty,\infty], \quad \vr^*(\varphi) = \sup\{ \varphi(x) - \vr(x) \mid x \in E\}.$$
Since $\vr(0) = 0$, we obtain that $\vr^* \colon E' \to [0,\infty]$ is proper, convex and lower semicontinuous and that the supremum in its definition can be taken over $x \in D(\vr)$. Moreover, due to the symmetry of $\vr$ it is also symmetric.  The {\em Fenchel-Moreau} theorem (see e.g. \cite[Proposition 4.1]{ET}) states that 
$$\vr(x) = \sup \{\varphi(x) - \vr^*(\varphi) \mid \varphi \in E'\}, \quad x \in E, $$
if and only if $\rho$ is lower semicontinuous.
Furthermore, it follows from the definition of $\rho^*$ that $(x,\varphi) \in \partial \vr$ if and only if 
$$\varphi(x) = \vr(x) + \vr^*(\varphi) . $$

For $x \in E$ we let
$$D(x) = \{\varphi \in  E' \mid \varphi(x) = \av{x}^2 = \av{\varphi}^2\}.$$
The operator $D = \{(x,\varphi) \mid x \in E, \varphi \in D(x)\}$ is called {\em normalized duality mapping}. It is well known (and easily verified) that $D = \partial (\frac{1}{2}\av{\cdot}^2)$, see e.g. \cite{Cio90} for this fact and further properties of the duality mapping. In particular, if $E$ is a Hilbert space and identified with $E'$ via the usual map, then $D$ is simply the identity operator. Moreover, if $(E,\av{\cdot})$ is strictly convex, $D$ is single-valued. 

For $\lambda > 0$ we define 
$$\vr^{(\lambda)} \colon E \to [0,\infty), \quad \vr^{(\lambda)}(x) = \inf \{\vr(y) + \frac{1}{2\lambda} \av{x-y}^2 \mid y \in E\}.$$
%
%
Assume that $(E,\av{\cdot})$ is strictly convex. For $x \in E$ and $\lambda > 0$,  we let $J_\lambda x \in E$ be the unique element such that  
$$\vr^{(\lambda)}(x)  =  \vr(J_\lambda x) + \frac{1}{2\lambda} \av{x-J_\lambda x}^2$$
(if it exists). By standard weak compactness arguments, this minimizer always exists if $(E,\av{\cdot})$  is a reflexive Banach space and $\vr$ is lower semicontinous. 
In this case, the family of operators $(J_\lambda)_{\lambda > 0}$ is called {\em resolvent family} of $\vr$. With the help of the resolvent we obtain the following alternative description of $\vr$ on the closure of $D(\vr)$.  It is certainly well-known to experts but we include a proof for the convenience of the reader. 

\begin{proposition}[Alternative formula for $\vr$]\label{lemma:alternative formula for vr}
 Assume that $(E,\av{\cdot})$ is a strictly convex reflexive Banach space and that $\vr$ is a lower semicontinous gsm. For $x \in E$ and $\varphi_\lambda = \lambda^{-1}D(x - J_\lambda x) \in E'$, $\lambda > 0$,  the following hold: 
 \begin{enumerate}[(a)]
  \item $(J_\lambda x,\varphi_\lambda) \in \partial \vr$. 
  \item If $x \in D(\vr)$, then 
  $$\vr(x) = \lim_{\lambda \to 0+} \vr(J_\lambda x)\text{ and } \lim_{\lambda \to 0+} \frac{1}{\lambda} \av{x - J_\lambda x}^2 =  \lim_{\lambda \to 0+} \varphi_\lambda(x - J_\lambda x) = 0.$$
  \item If $x \in D(\vr)$, then 
  $$\vr(x) = \lim_{\lambda \to 0+} \left(\varphi_\lambda(x) - \vr^*(\varphi_\lambda) \right).$$
 \end{enumerate}
\end{proposition}
\begin{proof}
 (a): By definition the element $J_\lambda x$ is a minimizer of the functional $y \mapsto \vr(y) + \frac{1}{2\lambda} \av{x-y}^2$. Hence, 
 $$0 \in \partial \left(\vr (\cdot) + \frac{1}{2\lambda} \av{x - \cdot}^2\right)(J_\lambda x).$$
 Using this and $D = \partial (\frac{1}{2} \av{\cdot}^2)$, we infer 
 $$0 \in \partial \vr (J_\lambda x) - \frac{1}{\lambda} D(x - J_\lambda x).$$
 
 (b): For $x \in D(\vr)$ the inequality $\vr^{(\lambda)} \leq \vr$ implies  $\av{x - J_\lambda x}^2 \leq 2 \lambda  \vr(x)$, showing $J_\lambda x \to x$ in $E$, as $\lambda \to 0+$. 
With the help of the lower semicontinuity of $\vr$, we infer 
$$ \limsup_{\lambda \to 0+} \left(\vr(J_\lambda x) + \frac{1}{2\lambda} \av{x - J_\lambda x}^2 \right) \leq  \vr(x) \leq \liminf_{\lambda \to 0+} \vr(J_\lambda x).$$
 But this implies  $\lim_{\lambda \to 0+} \vr(J_\lambda x) = \vr(x)$ and $ 0 = \lim_{\lambda \to 0+} \lambda^{-1} \av{x - J_\lambda x}^2 = \lim_{\lambda \to 0+} \varphi_\lambda(x - J_\lambda x)$.

 (c): Since  $(J_\lambda x,\varphi_\lambda) \in \partial \vr$, we have 
 $$\varphi_\lambda(J_\lambda x) = \vr(J_\lambda x) + \vr^*(\varphi_\lambda)$$
 and all of these quantities are finite. Rearranging them yields 
 \[ \varphi_\lambda(x) - \vr^*(\varphi_\lambda) = \rho(J_\lambda x) + \varphi_\lambda(x - J_\lambda x) . \] 
 Hence, the statement follows from (b).
\end{proof}

\section{A simple inequality}

\begin{lemma}\label{lemma:basice inequality}
 Let $x,y \in \R^n$, let $\lambda \in \R$ with $|\lambda| \leq 1$ and let $p \geq 1$. Then 
$$|x + \lambda y|^p + |x-\lambda y|^p \leq |x + y|^p + |x-y|^p,$$
where $|\cdot|$ denotes the Euclidean norm (or any other norm induced by a scalar product). 
\end{lemma}
\begin{proof}
Without loss of generality we can assume $0 \leq \lambda \leq 1$. The triangle inequality and the convexity of the function $t \mapsto t^p$ imply 
\begin{align*}
|x + \lambda y|^p + |x-\lambda y|^p &\leq \lambda \left(|x + y|^p + |x-y|^p \right) + 2(1-\lambda)|x|^p \\
&= \lambda (|x + y|^p + |x-y|^p - 2|x|^p) + 2 |x|^p.
\end{align*}
Hence, it suffices to show $|x + y|^p + |x-y|^p - 2|x|^p \geq 0$, as in this case the right side of the previous inequality attains its maximum at $\lambda = 1$. We rewrite $|x + y|^p + |x-y|^p = |R  + t_0|^q + |R - t_0|^q$, with $R = |x|^2 + |y|^2$, $t_0 = 2\as{x,y}$ and $q = p/2 > 0$. Then $|t_0| \leq R$. Consider the function   
$$\varphi \colon [-R,R]\to \R, \quad \varphi(t) = |R  + t|^q + |R - t|^q.$$
Computing its derivative shows directly that $\varphi$ is decreasing on $[-R,0]$ and increasing on $[0,R]$. Hence, its minimum is attained at $t  = 0$, where it attains the value $\varphi(0) = 2 |R|^q = 2 ||x|^2 + |y|^2|^q \geq 2|x|^p$.
\end{proof}

\begin{corollary}\label{corllary:basic inequality}
 Let $a,b \geq 0$ and $c \in \R$ with $|c| \leq ab$. Then for all $\lambda \in \R$ with $|\lambda| \leq 1$ and $1 \leq p < \infty$ the following inequality holds:  
 $$(a^2 + 2 \lambda c + \lambda^2 b^2)^{\frac p 2} + (a^2 - 2 \lambda c + \lambda^2 b^2)^{\frac p 2} \leq (a^2 + 2  c +  b^2)^{\frac p 2} + (a^2 - 2  c +  b^2)^{\frac p 2}. $$
\end{corollary}
 \begin{proof}
 Using $|c| \leq ab$ we obtain $b^2 - c^2/a^2 \geq 0$. In $\R^2$ the vectors $x = (a,0)$ and $y = (c/a, \sqrt{b^2 - c^2/a^2})$ satisfy $a = |x|, b = |y|$ and $c = \as{x,y}$. Hence, the statement follows directly from the previous lemma. 
 \end{proof}

\bibliographystyle{plain}
 
\bibliography{literatur}

\end{document}